\documentclass{article}
 
\usepackage[neurips]{optional}

\PassOptionsToPackage{numbers, compress}{natbib}

\opt{arxiv}{
	\usepackage[preprint]{arxiv_neurips_2019}
}
\opt{neurips}{
	\usepackage[final]{neurips_2019}
}

\usepackage{amsthm}
\usepackage{amssymb,amsmath}
\usepackage[utf8]{inputenc} 
\usepackage[T1]{fontenc}    
\usepackage[colorlinks,citecolor=blue,urlcolor=black,linkcolor=blue,pdfborder={0 0 0}]{hyperref}       
\usepackage{url}            
\usepackage{booktabs}       
\usepackage{amsfonts}       
\usepackage{nicefrac}       
\usepackage{microtype}      
\usepackage{todonotes}  
\usepackage{xspace}
\usepackage{float}
\usepackage{wrapfig}

\usepackage[capitalize]{cleveref}  
\crefname{nlem}{Lemma}{Lemmas}
\crefname{nprop}{Proposition}{Propositions}
\crefname{ncor}{Corollary}{Corollaries}
\crefname{nthm}{Theorem}{Theorems}
\crefname{exa}{Example}{Examples}
\crefname{assumption}{Assumption}{Assumptions}
\crefname{appendix}{Appendix}{Appendices}
\crefname{equation}{}{}
\usepackage{autonum}  

\def\balign#1\ealign{\begin{align}#1\end{align}}
\def\baligns#1\ealigns{\begin{align*}#1\end{align*}}
\def\balignat#1\ealign{\begin{alignat}#1\end{alignat}}
\def\balignats#1\ealigns{\begin{alignat*}#1\end{alignat*}}
\def\bitemize#1\eitemize{\begin{itemize}#1\end{itemize}}
\def\benumerate#1\eenumerate{\begin{enumerate}#1\end{enumerate}}

\newenvironment{talign*}
 {\csname align*\endcsname}
 {\endalign}
\newenvironment{talign}
 {\csname align\endcsname}
 {\endalign}

\def\balignst#1\ealignst{\begin{talign*}#1\end{talign*}}
\def\balignt#1\ealignt{\begin{talign}#1\end{talign}}

\let\originalleft\left
\let\originalright\right
\renewcommand{\left}{\mathopen{}\mathclose\bgroup\originalleft}
\renewcommand{\right}{\aftergroup\egroup\originalright}

\def\tinycitep*#1{{\tiny\citep*{#1}}}
\def\tinycitealt*#1{{\tiny\citealt*{#1}}}
\def\tinycite*#1{{\tiny\cite*{#1}}}
\def\smallcitep*#1{{\scriptsize\citep*{#1}}}
\def\smallcitealt*#1{{\scriptsize\citealt*{#1}}}
\def\smallcite*#1{{\scriptsize\cite*{#1}}}

\def\R{\mathbb{R}}

\def\Q{\mathbb{Q}}
\def\N{\mathbb{N}}

\def\<{\left\langle} 
\def\>{\right\rangle}

\def\half{\frac{1}{2}}

\def\E{\mbb{E}} 

\def\P{\mbb{P}} 

\DeclareMathOperator{\Tr}{Tr} 
\providecommand{\argmin}{\mathop\mathrm{arg min}}

\providecommand{\diag}{\mathop\mathrm{diag}}
\providecommand{\tr}{\mathop\mathrm{tr}}

\def\rank#1{\mathrm{rank}({#1})}

\ifdefined\nonewproofenvironments\else
\ifdefined\ispres\else
\newtheorem{theorem}{Theorem}

\renewenvironment{proof}{\noindent\textbf{Proof}\hspace*{1em}}{\qed\\}
\newenvironment{proof-sketch}{\noindent\textbf{Proof Sketch}
  \hspace*{1em}}{\qed\bigskip\\}
\newenvironment{proof-idea}{\noindent\textbf{Proof Idea}
  \hspace*{1em}}{\qed\bigskip\\}
\newenvironment{proof-of-lemma}[1][{}]{\noindent\textbf{Proof of Lemma {#1}}
  \hspace*{1em}}{\qed\\}
\newenvironment{proof-of-theorem}[1][{}]{\noindent\textbf{Proof of Theorem {#1}}
  \hspace*{1em}}{\qed\\}
\newenvironment{proof-attempt}{\noindent\textbf{Proof Attempt}
  \hspace*{1em}}{\qed\bigskip\\}

\fi

\newtheorem{proposition}[theorem]{Proposition}

\fi

\newtheorem{Theorem}{Theorem}
\newtheorem{Proposition}{Proposition}

\newtheorem{Lemma}{Lemma}

\newcommand{\End}{\mathrm{End}}
\renewcommand{\tr}{\mathrm{Tr}}
\newcommand{\field}[1]{\Gamma \big(#1 \big)}

\newcommand{\St}{\mathcal S}              
\newcommand{\metric}[2]{\left< #1, #2 \right>} 
\renewcommand{\H}{{\mathcal H}}          
\renewcommand{\P}{{\mathbb P}}          
\renewcommand{\R}{\mathbb R}               
\newcommand{\D}{{\mathcal D}}          
\renewcommand{\Q}{\mathbb Q}               
\renewcommand{\E}{\mathbb E}               
\newcommand{\h}{\mathcal H}              
\newcommand{\M}{\mathcal M}              
\renewcommand{\N}{\mathcal N}              
\newcommand{\X}{\mathcal X}              
\newcommand{\hvf}[1]{\setbox0=\hbox{$#1$}%
  \ifdim\wd0>1em\widehat{#1}\else\hat{#1}\fi} 
\newcommand{\dd}{\mathrm{d}}
\newcommand{\diff}{\mathrm{d}}
\newcommand{\stein}{k^0}
\newcommand{\defn}{ \equiv}
 \newcommand{\leqsim}{\,\mbox{{\scriptsize $\stackrel{<}{\sim}$}}\,}

\newcommand{\IF}{\operatorname{IF}}
\newcommand{\steinop}{\mathcal{S}}
\newcommand{\sd}{\operatorname{SD}} 
\newcommand{\sm}{\operatorname{SM}} 
\newcommand{\dsm}{\operatorname{DSM}} 
\newcommand{\ksd}{\operatorname{KSD}} 
\newcommand{\dksd}{\operatorname{DKSD}} 
\newcommand{\kl}{\operatorname{KL}} 

\title{Minimum Stein Discrepancy Estimators}

\author{%
  Alessandro Barp \\
  Department of Mathematics\\
  Imperial College London\\
  \texttt{a.barp16@imperial.ac.uk} \\
   \And
   Fran\c{c}ois-Xavier Briol \\
  Department of Statistical Science\\
   University College London \\
   \texttt{f.briol@ucl.ac.uk} \\
   \AND
   Andrew B. Duncan \\
  Department of Mathematics\\
   Imperial College London \\
   \texttt{a.duncan@imperial.ac.uk} \\
   \And
   Mark Girolami \\
  Department of Engineering\\
   University of Cambridge \\
   \texttt{mag92@eng.cam.ac.uk} \\
    \And
     Lester Mackey \\
   Microsoft Research\\
   Cambridge, MA, USA \\
    \texttt{lmackey@microsoft.com} 
}

\begin{document}

\maketitle

\begin{abstract}
When maximum likelihood estimation is infeasible, one often turns to score matching, contrastive divergence, or minimum probability flow to obtain tractable parameter estimates. We provide a unifying perspective of these techniques as \emph{minimum Stein discrepancy estimators}, and use this lens to design new diffusion kernel Stein discrepancy (DKSD) and diffusion score matching (DSM) estimators with complementary strengths. We establish the consistency, asymptotic normality, and robustness of DKSD and DSM estimators, then derive stochastic Riemannian gradient descent algorithms for their efficient optimisation. The main strength of our methodology is its flexibility, which allows us to design estimators with desirable properties for specific models at hand by carefully selecting a Stein discrepancy.
 We illustrate this advantage for several challenging problems for score matching, such as non-smooth, heavy-tailed or light-tailed densities.  
\end{abstract}

\section{Introduction}

Maximum likelihood estimation \citep{Casella2001} is a de facto standard for estimating the unknown parameters in a statistical model $\{\P_\theta : \theta \in \Theta\}$.  However, the computation and optimization of a likelihood typically requires access to the normalizing constants of the model distributions. This poses difficulties for complex statistical models for which direct computation of the normalisation constant would entail prohibitive multidimensional integration of an unnormalised density.  Examples of such models arise naturally in modelling images \cite{Gutmann2012,koster2010two}, natural language \cite{mnih2012fast}, Markov random fields \cite{roth2009fields} and nonparametric density estimation \cite{Sriperumbudur2017density,wenliang2018learning}. To by-pass this issue, various approaches have been proposed to address parametric inference for unnormalised models, including Monte Carlo maximum likelihood \cite{geyer1994convergence}, contrastive divergence \cite{Hinton2002}, minimum probability flow learning \cite{sohl2009minimum}, noise-contrastive estimation \cite{ceylan2018conditional,gutmann2010noise,Gutmann2012} and score matching (SM) \cite{Hyvarinen2006,Hyvarinen2007}.  

The SM estimator is a minimum score estimator \cite{Dawid2016} based on the Hyv\"{a}rinen scoring rule that avoids normalizing constants by depending on $\P_\theta$ only through the gradient of its log density $\nabla_x \log p_{\theta}$. SM estimators have proven to be a widely applicable method for estimation for models with unnormalised smooth positive densities, with generalisations to bounded domains \cite{Hyvarinen2007} and  compact Riemannian manifolds \cite{mardia2016score}. Despite the flexibility of this approach, SM has three important and distinct limitations. Firstly, as the Hyv\"{a}rinen score depends on the Laplacian of the log-density, SM estimation will be expensive in high dimension and will break down for non-smooth models or for models in which the second derivative grows very rapidly.  Secondly, as we shall demonstrate, SM estimators can behave poorly for models with heavy tailed distributions. Thirdly, the SM estimator is not robust to outliers in many applications of interest.  Each of these situations arise naturally for energy models, particularly product-of-experts models and ICA models \cite{hyvarinen1999sparse}.  

In a separate strand of research, new approaches have been developed to measure discrepancy between an unnormalised distribution and a sample. In \cite{Gorham2015,Gorham2016,mackey2016multivariate,Gorham2017}, it was shown that Stein's method can be used to construct discrepancies that control weak convergence of an empirical measure to a target.  

In this paper we consider minimum Stein discrepancy (SD) estimators and show that SM, minimum probability flow and contrastive divergence estimators are all special cases.  Within this class we focus on SDs constructed from reproducing kernel Hilbert Spaces (RKHS), establishing the consistency, asymptotic normality and robustness of these estimators. We demonstrate that these SDs are appropriate for estimation of non-smooth distributions and heavy- or light- tailed distributions.  The remainder of the paper is organized as follows. In \cref{mininimum-stein-discrepancy-estimators} we introduce the class of minimum SD estimators, then investigate asymptotic properties of SD estimators based on kernels in \cref{large-sample-theory}, demonstrating consistency and asymptotic normality under general conditions, as well as conditions for robustness. \cref{numerics-dksd} presents three toy problems where SM breaks down, but our new estimators are able to recover the truth. All proofs are in the supplementary materials.


\section{Minimum Stein Discrepancy Estimators}\label{mininimum-stein-discrepancy-estimators}

Let $\mathcal P_\X$ the set of Borel probability measures on $\X$. Given identical and independent (IID) realisations from $\Q \in \mathcal P_\X$ on an open subset $\X \subset \R^d$, the objective is to find a sequence of measures $\mathbb{P}_{n}$ that approximate $\mathbb{Q}$ in an appropriate sense.
More precisely we will consider a family $\mathcal{P}_{\Theta} = \lbrace \mathbb{P}_{\theta} : \theta \in \Theta \rbrace \subset \mathcal{P}_\X$ together with a
function $D:\mathcal{P}_{\X}\times \mathcal{P}_{\X} \rightarrow \R_{+}$  which quantifies the discrepancy between any two measures in $\mathcal{P}_{\X}$,
and wish to estimate an optimal parameter $\theta^*$ satisfying
$
 \theta^* \in \arg\min_{\theta \in \Theta}D(\Q \| \mathbb{P}_\theta).
$
In practice, it is often difficult to compute the discrepancy $D$ explicitly, and it is useful to consider a random approximation $\hat{D}( \lbrace X_i \rbrace_{i=1}^{n} \| \mathbb{P}_{\theta})$ based on a IID sample $X_1, \ldots, X_n \sim \Q$, such that $\hat{D}( \lbrace X_i \rbrace_{i=1}^{n} \| \mathbb{P}_{\theta}) \xrightarrow{a.s.} D(\Q \| \mathbb{P}_\theta)$ as $n \rightarrow \infty$.  We then consider the sequence of estimators
\begin{talign}\label{minimum_distance_estimator}
\hat{\theta}^D_n  \in  \text{argmin}_{\theta \in \Theta} \hat D \left ( \{X_i\}_{i=1}^n \|\mathbb{P}_{\theta} \right).
\end{talign}
The choice of discrepancy will impact the consistency, efficiency and robustness of the estimators. 
  Examples of such estimators include minimum distance estimators \cite{Basu2011,Pardo2005} where the discrepancy will be a metric on probability measures,  including minimum maximum mean discrepancy (MMD) estimation \cite{Dziugaite2015,Li2015GMMN,Briol2019} and minimum Wasserstein estimation \cite{Frogner2015,Genevay2017,Bernton2017}.  

More generally, minimum scoring rule estimators \cite{Dawid2016} arise from proper scoring rules, for example Hyv\"{a}rinen, Bregman and Tsallis scoring rules. These discrepancies are often statistical divergences, i.e., $D(\mathbb{Q}\|\mathbb{P}) = 0 \Leftrightarrow \mathbb{P} = \mathbb{Q}$ for all $\mathbb{P},\mathbb{Q}$ in a subset of $ \mathcal{P}_{\X}$.
 Suppose that $\mathbb{P}_{\theta}$ and $\mathbb{Q}$ are absolutely continuous with respect to a common measure $\lambda$ on $\X$, with respective positive densities $p_\theta$ and $q$. 
 Then a well-known statistical divergence is the Kullback-Leibler (KL) divergence $\operatorname{KL}(\mathbb{Q}\|\mathbb{P}_\theta) \defn \int_{\X}  \log(\mathrm{d}\mathbb{Q}/\mathrm{d}\mathbb{P}_\theta) \mathrm{d}\mathbb{Q}=
\int_{\X} \log q \dd \Q  - \int_{\X} \log p_{\theta} \dd \Q $.
Minimising $\operatorname{KL}(\Q \|\P_{\theta})$
  is equivalent to maximising  $\int_{\X} \log p_{\theta} \dd \Q  $, which can be estimated using the likelihood  $ \hvf{ \operatorname{KL}}  ( \{X_i\}_{i=1}^n \|\mathbb{P}_{\theta})\defn \frac 1 n \sum_{i=1}^n \log p_{\theta}(X_i )$.
Informally, we see that minimising the KL-divergence is equivalent to performing maximum likelihood estimation. 

For our purposes we are interested in discrepancies that can be evaluated when  $\mathbb{P}_\theta$ is only known up to normalisation, precluding the use of KL divergence.  We instead consider a related class of discrepancies based on integral probability pseudometric (IPM) \citep{Muller1997} and Stein's method \citep{Barbour2005,Chen2011,Stein1972}. Let $\Gamma(\mathcal Y) \defn \Gamma(\mathcal{X}, \mathcal Y) \defn \{ f: \X \to \mathcal Y \}$. A map $\mathcal{S}_{\mathbb{P}}:\mathcal G \subset \Gamma(\R^d) \to \Gamma(\R)$ is a Stein operator over a Stein class $\mathcal{G}$ if $\int_{\X} \mathcal{S}_{\mathbb{P}}[f] \diff  \mathbb{P}  =  0 \; \forall f \in \mathcal{G}$ for any $\mathbb{P}$.
 We can then define an associated \emph{Stein discrepancy} (SD) \citep{Gorham2015} using an IPM with entry-dependent function space $\mathcal{F} \defn  \mathcal{S}_{\mathbb{P}_\theta}[\mathcal G]$
\begin{talign}
\sd_{\steinop_{\P_\theta}[\mathcal G]}\left(\mathbb{Q}  \| \mathbb{P}_\theta  \right)\defn  
\sup_{f \in \mathcal{S}_{\mathbb{P}_\theta}[\mathcal G]} \left|  \int_{\X} f \dd \mathbb{P}_\theta - \int_{\X} f \dd \mathbb{Q}  \right|
 = 
\sup_{g \in \mathcal{G}} \left| \int_{\X} \mathcal{S}_{\mathbb{P}_\theta}[g] \dd \mathbb{Q}  \right|. \label{Stein_discrepancy}
\end{talign}
The Stein discrepancy depends on $\mathbb{Q}$ only through expectations, and does not require the existence of a density, therefore permitting $\mathbb{Q}$ to be an empirical measure.
If $\mathbb{P}$ has a $C^1$ density $p$ on $\mathcal{X}$, one can consider the Langevin-Stein discrepancy arising from the Stein operator $\mathcal T_{p} [g] \defn \metric{\nabla\log p}{ g} +   \nabla \cdot g$~\citep{Gorham2015,Gorham2016}. 
 In this case, the Stein discrepancy will not depend on the normalising constant of $p$.
 
In this paper, for an arbitrary $m \in \Gamma(\mathbb{R}^{d\times d})$ which we call \emph{diffusion matrix},
 we shall consider the more general \emph{diffusion Stein operators} \citep{Gorham2016}: 
   $\steinop_{p}^m[g]  \defn (1/p) \nabla \cdot\left( {p} mg\right)$ , $\mathcal S_{p}^m[A] \defn   (1/p) \nabla \cdot \left( {p} mA \right),$ where $g \in \Gamma(\R^d)$, $A\in \Gamma(\mathbb{R}^{d\times d})$, 
and the associated \emph{minimum Stein discrepancy estimators} which minimise \cref{Stein_discrepancy}.
 As we will only have access to a sample $\{X_i\}_{i=1}^n \sim \Q$, we will focus on the estimators minimising an approximation  $\hvf \sd_{\steinop_{\P_{\theta}}[\mathcal G]}(\{X_i\}_{i=1}^n\|\P_\theta)$ based on a $U$-statistic of the $\Q$-integral:
\begin{talign}\label{minimum_Stein_estimator}
\hat{\theta}^{\text{Stein}}_n & \defn  \text{argmin}_{\theta \in \Theta} \hvf \sd_{\steinop_{\P_{\theta}}[\mathcal G]} (\{X_i\}_i^n\|\P_\theta).
\end{talign}
Related and complementary approaches to inference using SDs include the nonparametric estimator of \cite{Li2018implicit}, the density ratio approach of \cite{Liu2018Fisher} and the variational inference algorithms of \cite{Ma2017,Ranganath2016}. We now highlight several instances of SDs which will be studied in detail in this paper.


\subsection{Example 1: Diffusion Kernel Stein Discrepancy Estimators}\label{dksd-estimators}

A convenient choice of Stein class is the unit ball of reproducing kernel Hilbert spaces (RKHS) \citep{Berlinet2004} of a scalar kernel function $k$.
For the Langevin Stein operator $\mathcal T_p$, the resulting \emph{kernel Stein discrepancy} (KSD) first appeared in \cite{Oates2017} and has since been considered extensively in the context of hypothesis testing, measuring sample quality and approximation of probability measures in \cite{Chen2018,Chen2019,Chwialkowski2016,Detommaso2018,Gorham2017,Liu2016,Liu2016testing,Liu2017geomSVGD}. In this paper, we consider a more general class of discrepancies based on the diffusion Stein operator and matrix-valued kernels.

Consider an RKHS $\mathcal{H}^d$ of functions $f \in \Gamma( \mathbb R^d)$ with (matrix-valued) kernel $K  \in \Gamma( \mathcal X \times \mathcal X ,  \mathbb R^{d \times d })$, 
$K_x \defn K(x,\cdot)$ (see \cref{vector-RKHS} and \ref{separable-kernels} for further details). 
The Stein operator $\mathcal S^m_p[f]$ induces an operator  $\mathcal S_p^{m,2} \mathcal S_p^{m,1}:\field{\X\times \X, \R^{d\times d}} \to \field{\R}$ which acts  first on the first variable and then on the second one.
We briefly mention two simple examples of matrix kernels constructed from scalar kernels. 
If we want the components of $f$ to be orthogonal, we can use the diagonal kernel (i) $K = \mbox{diag}(\lambda_1 k^1,\ldots,\lambda_d k^d)$ where $\lambda_i > 0$ and $k^i$ is a $C^2$ kernel on $\X$, for $i=1,\ldots, d$; 
else we can ``correlate" the components by setting (ii) $K = Bk$ where  $k$ is a (scalar) kernel on $\X$ and $B$ is a (constant) symmetric positive definite matrix. 

We propose to study \emph{diffusion kernel Stein discrepancies} indexed by $K$ and $m$ (see \cref{appendix:derivations}):
 \begin{theorem}[\textbf{Diffusion Kernel Stein Discrepancy}]
 For any kernel $K$, 
 we find that 
 $\mathcal S^m_p[f](x) = \langle \mathcal S_p^{m,1}K_x,f\rangle_{\mathcal H^d}$  for any $f\in \mathcal H^d$.
Moreover if $x\mapsto \|\mathcal S^{m,1}_p K_x \|_{\mathcal H^d} \in L^1(\mathbb Q)$, 
we have 
\begin{talign}
\dksd_{K,m}(\mathbb{Q}\|\P)^2  &\;  \defn \; \sup_{\substack{h \in \mathcal{H}^d\\ \lVert h \rVert \leq 1}}\left | \int_{\X} \mathcal{S}^m_p[h]\dd\mathbb{Q}\right |^2 = \int_{\X} \int_{\X}\stein (x,y)\dd\mathbb{Q}(x)\dd\mathbb{Q}(y) 
\end{talign} 
\begin{talign} \label{stein-kernel}
\stein(x,y) \defn \mathcal{S}^{m,2}_p\mathcal{S}^{m,1}_p K(x, y)=\frac 1 {p(y)p(x)}  \nabla_y \cdot \nabla_x \cdot  \left( p(x)m(x)K(x,y) m(y)^{\top}p(y)\right).
\end{talign}
 \end{theorem}
In order to use these for minimum SD estimation, we propose the following $U$-statistic approximation:
\begin{talign}\label{dksd-approx}
\hvf\dksd_{K,m}( \{X_i\}_{i=1}^n\|\P_{\theta})^2 \;  = \; \frac{2}{n(n-1)} \sum_{1 \leq i < j \leq n}\stein_{\theta}(X_i, X_j)
	\;  = \; \frac{1}{n(n-1)} \sum_{i \neq j }\stein_{\theta}(X_i, X_j),
\end{talign} 
with associated estimators:
$ \hat \theta_n^{\dksd}  \in  \text{argmin}_{\theta \in \Theta} \hvf\dksd_{K,m}( \{X_i\}_{i=1}^n\|\P_{\theta})^2$.

As the proof shows, the Stein kernel $\stein$ is indeed a (scalar) kernel obtained from the feature map $\phi:\X \to \H^d$, $\phi(x) \defn   \steinop_p^{m,1}[K]|_x$.
 For $K=Ik$, $m=Ih$, $\dksd$ is a $\ksd$ with scalar kernel $h(x)k(x,y)h(y)$, and if $h=1$ our objective becomes the usual Langevin-based KSD of \cite{Chwialkowski2016,Gorham2017,Liu2016testing,Oates2017} (see \cref{app:special-cases-dksd}). The work of \cite{liu2017learning} discussed the potential of optimizing the KSD with gradient descent but did not evaluate its merits. In the sections to follow, we will see the advantages conferred by introducing more flexible diffusion operators, matrix kernels, and Riemannian optimization.

Now that our DKSD estimators are defined, an important remaining question is under which conditions can $\dksd$ discriminate distinct probability measures. To answer, we will need several definitions.
 We say a matrix kernel $K$ is in the Stein class of $\Q$ if $\int_\X  \steinop_q^{m,1}[K] \dd \Q=0$, and that it is strictly integrally positive definite (IPD) if 
$ \int_{\X \times \X} \dd \mu^{\top}(x) K(x,y) \dd \mu(y) >0 $
 for any finite non-zero signed vector Borel measure  $\mu$. 
 From $\mathcal S^m_p[f](x) = \langle \mathcal S_p^{m,1}K_x,f\rangle_{\mathcal H^d}$ we have that $f \in \mathcal H^d$ is in the Stein class (i.e., $\int_\X \mathcal S^m_q[f] \dd \Q=0$) when $K$ is also in the class. 
Setting $s_p \defn m^{\top}\nabla \log p \in \Gamma( \R^{d })$: 
\begin{Proposition}[\textbf{DKSD as a Statistical Divergence}] \label{DKSD-divergence} Suppose $K$ is IPD and in the Stein class of $\Q$, and $m(x)$ is invertible. 
If  $s_p -s_q \in L^1(\Q)$,
then $\dksd_{K,m}(\Q\|\P)^2=0$ iff $\Q=\P$.
\end{Proposition}
See \cref{Derivations-DKSD} for the proof. Note that this proposition generalises Proposition 3.3 from \cite{Liu2016testing} to a significantly larger class of SD. For the matrix kernels introduced above, the proposition below shows that $K$ is IPD when its associated scalar kernels are;
a well-studied problem \citep{Sriperumbudur2009}.
\begin{Proposition}[\textbf{IPD Matrix Kernels}]\label{IPD-matrix-kernels}
 (i) Let $K=\text{diag}(k^1,\ldots,k^d)$. Then $K$ is IPD iff each kernel $k^i$ is IPD.
(ii) Let $K=Bk$ for $B$ be symmetric positive definite. Then $K$ is IPD iff $k$ is IPD.
\end{Proposition} 

\subsection{Example 2: Diffusion Score Matching Estimators}

A well-known family of estimators are the score matching (SM) estimators (based on the Fisher or Hyvarinen divergence)~\cite{Hyvarinen2006,Hyvarinen2007}. As will be shown below, these can be seen as special cases of minimum SD estimators. The SM discrepancy is computable for sufficiently smooth densities:
\begin{talign}
\sm(\Q \| \P)
& \defn \int_{\X} \| \nabla \log p- \nabla \log q \|^2_2 \; \dd \Q  = \int_{\X} \left( \| \nabla \log q \|^2_2+ 
\| \nabla \log p \|^2_2+2 \Delta \log p \right) \dd \Q  
\end{talign} 
where $\Delta$ denotes the Laplacian and we have used the divergence theorem. If $\P = \mathbb P_{\theta}$, the first integral above does not depend on $\theta$, and the second one does not depend on the density of $\Q$,
  so we consider the approximation $\hvf \sm ( \{X_i\}_{i=1}^n \| \P_\theta ) \defn \frac{1}{n}\sum_{i=1}^n \Delta \log p_{\theta}(X_i) + \frac{1}{2} \|\nabla \log p_{\theta}(X_i)\|^2_2$ based on an unbiased estimation for the minimiser of the SM divergence, and its estimators
$\hat{\theta}^{\text{SM}}_n \defn \text{argmin}_{\theta \in \Theta} \hvf \sm ( \{X_i\}_{i=1}^n \| \P_\theta )$, for independent random vectors $X_i \sim \Q$.
 
The SM discrepancy can also be generalised to include higher-order derivatives of the log-likelihood \citep{Lyu2009} and does not require a normalised model. We will now introduce a further generalisation that we call \emph{diffusion score matching (DSM)} which is a SD constructed from the diffusion Stein operator (see \cref{derivations-DSM}):
\begin{theorem}[\textbf{Diffusion Score Matching}] \label{prop:SM_is_Stein}
Let $\X = \mathbb{R}^d$ and consider some diffusion Stein operator $\mathcal{S}^m_p$ for some function $m \in \Gamma( \R^{d\times d})$
and the Stein class
$\mathcal{G} \defn \{g = (g_1,\ldots,g_d) \in C^1(\X, \R^d) \cap L^2(\X;\mathbb{Q}) : 
\|g\|_{L^2(\X;\mathbb{Q})} \leq 1\}$. 
If $p,q>0$ are differentiable and $s_p -s_q \in L^2(\Q)$, then we define the \emph{diffusion score matching} divergence as the Stein discrepancy,
\begin{talign}
\dsm_m(\Q\| \P) 
\defn 
\sup_{f \in \mathcal{S}_{p}[\mathcal G]} \left|  \int_{\X} f \dd\Q - \int_{\X} f\dd\P  \right|^2
= \int_{\X} \left\| m^{\top} \left(\nabla \log q - \nabla \log p\right) \right\|_{2}^2 \dd \Q.
\end{talign}
This satisfies $\dsm_m(\Q\| \P) =0$ iff $\Q=\P$
when $m(x)$ is invertible.
Moreover, if $p$ is twice-differentiable, and $q  m m^{\top} \nabla \log p , \nabla \cdot (q  m m^{\top} \nabla \log p) \in L^1(\R^d)$, then Stoke's theorem gives
\begin{talign}\label{Stoke-Score}
\dsm_m( \Q\| \P)=\int_{\X} \left( \|  m^{\top}\nabla_x\log p \|_{2}^2 +\| m^{\top} \nabla \log q \|^2_2+2\nabla \cdot\left( m m^{\top} \nabla \log p\right)\right) \dd \Q.
\end{talign} 
\end{theorem}
Notably, $\dsm_m$ recovers $\sm$ when  $m(x)m(x)^{\top} = I$ and the (generalised) non-negative score matching estimator of \cite{Lyu2009} with the choice $m(x) \equiv \text{diag}(h_1(x_1)^{1/2},\ldots,h_d(x_d)^{1/2})$. Like standard SM, DSM is only defined for distributions with sufficiently smooth densities.
Since the $\theta$-dependent part of $\dsm_m(\Q \| \P_{\theta})$  does not depend on the density of $\Q$, and can be estimated using an empirical mean, leading to the estimators
$\hat{\theta}^{\dsm}_n  \defn  \text{argmin}
_{\theta \in \Theta} \hvf \dsm_m ( \{X_i\}_{i=1}^n \| \P_\theta )$ for
\begin{talign}
\hvf \dsm_m ( \{X_i\}_{i=1}^n \| \P_\theta ) \defn
\frac{1}{n}\sum_{i=1}^n \left( \|  m^{\top}\nabla_x\log p_{\theta} \|_{2}^2+2\nabla \cdot\left( m m^{\top} \nabla \log p_{\theta}\right)\right)(X_i) 
\end{talign}
where $\{X_i\}_{i=1}^n$ is a sample from $\Q$. Note that this is only possible if $m$ is independent of $\theta$, in contrast to $\dksd$ where $m$ can depend on $\X \times \Theta$, thus leading to a more flexible class of estimators.

 An interesting remark is that the $\dsm_m$ discrepancy may in fact be obtained as a limit of $\dksd$ over a sequence of target-dependent kernels:
see \cref{derivations-DSM} for the complete result which corrects and significantly generalises previously established connections between the SM divergence and KSD (such as in Sec. 5 of \cite{Liu2016testing}).

We conclude by commenting on the computational complexity. Evaluating the DKSD loss function requires $\mathcal{O}(n^2d^2)$ computation, due to the U-statistic and a matrix-matrix product. However, if $K=\diag(\lambda_1 k^1,\ldots, \lambda_d k^d)$ or $K=Bk$, and if $m$ is a diagonal matrix, then we can by-pass expensive matrix products and the cost is $\mathcal{O}(n^2d)$, making it comparable to that of KSD. Although we do not consider these in this paper, recent approximations to KSD could also be adapted to DKSD to reduce the computational cost to $\mathcal{O}(nd)$ \citep{huggins2018random,Jitkrittum2017}. The DSM loss function has computational cost $\mathcal{O}(n d^2)$, which is comparable to the SM loss. From a computational viewpoint, DSM will hence be preferable to DKSD for large $n$, whilst DKSD will be preferable to DSM for large $d$.

\subsection{Further Examples: Contrastive Divergence and Minimum Probability Flow}

Before analysing DKSD and DSM estimators further, we show that the class of minimum SD estimators also includes other well-known estimators for unnormalised models. 
Let $X_{\theta}^n$, $n \in \mathbb{N}$ be a Markov process with unique invariant probality measure $\P_{\theta}$, for example a Metropolis-Hastings chain. Let $P_{\theta}^n$ be the associated transition semigroup, i.e. $(P_{\theta}^n f)(x) = \mathbb{E}[f(X_{\theta}^n ) | X_{\theta}^0 = x]$.  Choosing the Stein operator $\steinop_{p} = I - P_{\theta}^n$ and Stein class $\mathcal{G} = \lbrace \log p_{\theta} + c \, : \, c \in \R\rbrace$, leads to the following SD: 
\begin{talign}
\text{CD}(\Q \| \P_{\theta}) = \int_{\mathcal{X}} \left(\log p_{\theta} - P^n_{\theta}\log p_{\theta}\right)\dd \Q =\mbox{KL}(\Q \| \P_{\theta}) - \mbox{KL}(\Q_{\theta}^n \| \P_{\theta}),
\end{talign}
where $\Q_{\theta}^{n}$ is the law of $  X_{\theta}^n | X_{\theta}^0 \sim \mathbb{Q}$ and assuming that $\Q \ll \P_{\theta}$ and $\Q_{\theta}^n \ll \P_{\theta}$, which is the loss function associated with contrastive divergence (CD) \cite{Hinton2002,liu2017learning}. 
Suppose now that $\mathcal{X}$ is a finite set.  Given $\theta\in \Theta$ let $P_{\theta}$ be the transition matrix for a Markov process with unique invariant distribution $\P_{\theta}$.   Suppose we observe data $\{x_i\}_{i=1}^n$ and let $q$ be the corresponding empirical distribution.   Choosing the Stein operator $\steinop_{p} = I - P_{\theta}$ and the Stein set $\mathcal{G} = \lbrace f \in \Gamma(\R) :   \|f\|_{\infty} \leq 1 \rbrace$.  
Note that, $g \in \arg\sup_{g\in G}|\Q(\steinop_{p}[g])|$ will satisfy $g(i) =\mbox{sgn}(q^{\top}(I - P_{\theta})_{i})$, and the resulting Stein discrepancy is 
the minimum probability flow loss objective function \cite{sohl2009minimum2009}:
\begin{talign}
\mbox{MPFL}(\Q \| \P) = \sum_{y} \left| ((I - P_{\theta})^{\top} q)_{y} \right|. 
\end{talign}

\subsection{Implementing Minimum SD Estimators: Stochastic Riemannian Gradient Descent}

In order to implement the minimum $\sd$  estimators, we propose to use a stochastic gradient descent (SGD) algorithm associated to the information geometry induced by the SD on the parameter space. 
More precisely, consider a parametric family $\mathcal P_{\Theta}$ of probability measures on $\X$ with $\Theta \subset \R^m$. 
Given a discrepancy $D:\mathcal P_{\Theta} \times \mathcal P_{\Theta} \to \R$ satisfying $D(\P_{\alpha} \| \P_{\theta})=0$ iff $\P_{\alpha} = \P_{\theta}$ (called a statistical divergence), its associated information matrix field on $\Theta$ is defined  as the map $\theta \mapsto g(\theta)$, where $g(\theta)$ is the symmetric bilinear form $g(\theta)_{ij} = - \half (\partial^2/\partial \alpha^i \partial \theta^j) D( \P_{\alpha}\|\P_{\theta}) |_{\alpha=\theta}$ \citep{Amari2016}. 
When $g$ is positive definite, we can use it to perform (Riemannian) gradient descent on the parameter space $ \Theta$. We provide below the information matrices of DKSD and DSM (and hence extends results of \cite{Karakida2016}):
\begin{Proposition}[\textbf{Information Tensor DKSD}]\label{information-dksd}  Assume the conditions of \cref{DKSD-divergence} hold.  
The information tensor associated to DKSD is positive semi-definite and has components
\begin{talign}
g_{\dksd}(\theta)_{ij}  =\int_{\X} \int_{\X}  \left( \nabla_x \partial_{\theta^j} \log p_{\theta}(x)  \right)^{\top}  m_{\theta}(x) K(x,y) m^{\top}_{\theta}(y) \nabla_y  \partial_{\theta^i}\log p_{\theta}(y) \dd \P_{\theta}(x) \dd \P_{\theta}(y).
\end{talign}
\end{Proposition}
\begin{Proposition}[\textbf{Information Tensor DSM}]\label{app:information-dsm} Assume the conditions of \cref{prop:SM_is_Stein} hold. The information tensor defined by DSM is positive semi-definite and has components 
 \begin{talign}
 g_{\dsm}(\theta)_{ij} = \int_{\X} \metric{ m^{\top} \nabla \partial_{\theta^i} \log p_{\theta} }{m^{\top} \nabla \partial_{\theta^j} \log p_{\theta}} \dd \P_{\theta}.
 \end{talign}
\end{Proposition}
See \cref{appendix:information-semi-metrics} for the proofs. Given an (information) Riemannian metric, recall the gradient flow of a curve $\theta$ on the Riemannian manifold $\Theta$  is the solution to $\dot \theta (t) = - \nabla_{\theta(t) } \sd(\Q\| \P_{\theta}) $, where $\nabla_{\theta}$ denotes the Riemannian gradient at $\theta$. It is the curve that follows the direction of steepest decrease (measured with respect to the Riemannian metric) of the function $\sd(\Q\| \P_{\theta})$ (see \cref{app:riemannian-grad-descent}). The well-studied natural gradient descent \citep{Amari1998,Amari2016} corresponds to the case in which the Riemannian manifold is $\Theta = \R^m$ equipped with the Fisher metric and $\sd$ is replaced by  $ \kl$.
 When $\Theta$ is a linear manifold with coordinates $(\theta^i)$ we have 
 $\nabla_\theta \sd(\Q\| \P_{\theta}) = g(\theta)^{-1}\dd_{\theta}  \sd(\Q\| \P_{\theta}) $, where $\dd_\theta f$ denotes the tuple $(\partial_{\theta^i}f)$. We will approximate this at step $t$ of the descent using the biased estimator 
 $\hat g_{ \theta_t}(\{X_i^t\}_i)^{-1} \dd_{\theta_t} \hvf \sd(\{X_i^t\}_{i=1}^n \| \P_\theta)$, 
 where $\hat g_{ \theta_t}(\{X_i^t\}_{i=1}^n)$ is an unbiased estimator for the information matrix $g(\theta_t)$ and $\{X^t_i \sim \Q\}_i$ is a sample at step $t$. In general, we have no guarantee that $\hat{g}_{\theta_t}$ is invertible, and so we may need a further approximation step to obtain an invertible matrix. Given a sequence $(\gamma_t)$ of step sizes we  will approximate the gradient flow with
 \begin{talign}
 \hat \theta_{t+1}= \hat \theta_t- \gamma_t \hat g_{ \theta_t}(\{X_i^t\}_{i=1}^n)^{-1} \dd_{\theta_t} \hvf \sd(\{X_i^t\}_{i=1}^n \| \P_\theta).
 \end{talign}
 
Minimum SD estimators hold additional appeal for exponential family models, since their densities have the form
$
p_{\theta}(x)  \propto \exp \left(\left\langle \theta, T(x)\right\rangle_{\R^m} \right)\exp(b(x))
$
for natural parameters $\theta \in \mathbb{R}^m$, sufficient statistics $T \in \Gamma( \R^m)$, and base measure $\exp(b(x))$.
For these models, the U-statistic approximations of DKSD and DSM are convex quadratics with closed form solutions whenever $K$ and $m$ are independent of $\theta$.
Moreover, since the absolute value of an affine function is convex, and the supremum of convex functions is convex, any $\sd$ with a diffusion Stein operator is convex in $\theta$, provided $m$ and the Stein class $\mathcal G$ are independent of $\theta$.

\section{Theoretical Properties for Minimum Stein Discrepancy Estimators}\label{large-sample-theory}

We now show that the DKSD and DSM estimators have many desirable properties such as consistency, asymptotic normality and bias-robustness. These results do not only provide us with reassuring theoretical guarantees on the performance of our algorithms, but can also be a practical tool for choosing a Stein operator and Stein class given an inference problem of interest.

We begin by establishing strong consistency and for DKSD; i.e. almost sure convergence: 
$ \hat{\theta}_n^{\dksd} \xrightarrow{a.s.} \theta_*^{\dksd} \defn  \text{argmin}_{\theta \in \Theta} \dksd_{K,m}(\mathbb{Q} \|\P_{\theta})^2$. This will be followed by a proof of asymptotic normality.
We will assume we are in the specified setting, so that $\Q =\P_{\theta_*^{\dksd}}\in \mathcal P_{\Theta}$. 
In the misspecified setting, we will need to also assume the existence of a unique minimiser.
\begin{theorem}[\textbf{Strong Consistency of DKSD}]  \label{DKSD-consistency}
Let $\X=\R^d$,  $\Theta \subset \R^m$.
Suppose that $K$ is  bounded with bounded derivatives up to order $2$, 
that $\stein(x,y)$ is continuously-differentiable on an $\R^m$-open neighbourhood of $\Theta$,
 and that for any compact subset $C \subset \Theta$ there exist functions $f_1, f_2, g_1, g_2$ such that for $\mathbb{Q}$-a.e. $x \in \mathcal{X}$,
\begin{enumerate}
  \item $\left\| m^{\top}(x) \nabla\log p_{\theta}(x) \right\| \leq f_1(x)$, where $f_1 \in L^1(\Q)$ and continuous,
  \item $\left\| \nabla_{\theta} \left(m(x)^{\top} \nabla\log p_{\theta}(x)\right) \right\| \leq g_1(x)$, where $g_1 \in L^1(\Q)$ is continuous,
  \item $\left\| m(x)\right\| + \left\| \nabla_{x}m(x) \right\| \leq f_2(x)$ where $f_2\in L^1(\Q)$ and continuous,
  \item $\left\| \nabla_{\theta} m(x) \right\| + \left\| \nabla_{\theta} \nabla_{x}m(x) \right\| \leq g_2(x)$ where $g_2\in L^1(\Q)$ is continuous.
\end{enumerate}
Assume further that $\theta \mapsto \P_{\theta}$ is injective. 
Then we have a unique minimiser $\theta_*^{\dksd}$, and if either $\Theta$ is compact, or $\theta_*^{\dksd} \in \text{int}(\Theta)$ and $\Theta$ and $\theta \mapsto \hvf\dksd_{K,m}( \{X_i\}_{i=1}^n \|\P_{\theta})^2$ are convex, then $\hat \theta_n^{\dksd}$ is strongly consistent.
\end{theorem}
\begin{theorem}[\textbf{Central Limit Theorem for DKSD}]\label{DKSD-normality}
Let $\X$ and $\Theta$ be open subsets of $\R^d$ and $\R^m$ respectively.
 Let $K$ be a bounded kernel with bounded derivatives up to order $2$ and suppose that $\hat \theta_n^{\dksd} \xrightarrow{p} \theta_*^{\dksd} $ and that there exists a compact neighbourhood $\mathcal{N} \subset \Theta$ of $\theta_*^{\dksd}$ such that $\theta \rightarrow \hvf \dksd_{K,m}( \{X_i\}_{i=1}^n,\P_{\theta})^2$ is twice continuously differentiable for $\theta \in \mathcal{N}$ and, for $\mathbb{Q}$-a.e. $x \in \mathcal{X}$,
\begin{enumerate}
   \item $\| m^{\top}(x) \nabla\log p_{\theta}(x) \| + \| \nabla_{\theta} \left(m(x)^{\top} \nabla\log p_{\theta}(x)\right) \| \leq f_1(x),$
  \item $\| m(x) \| + \| \nabla_{x}m(x) \| + \| \nabla_{\theta} m(x) \| + \| \nabla_{\theta} \nabla_{x}m(x) \| \leq f_2(x),$
   \item $\| \nabla_{\theta}\nabla_{\theta} \left(m(x)^{\top} \nabla\log p_{\theta}(x)\right) \| + \| \nabla_{\theta}\nabla_{\theta}\nabla_{\theta} \left(m(x)^{\top} \nabla\log p_{\theta}(x)\right) \| \leq g_1(x),$
  \item $\|\nabla_{\theta}\nabla_{\theta} m(x) \| + \|\nabla_{\theta}\nabla_{\theta} \nabla_{x}m(x) \| + \|\nabla_{\theta}\nabla_{\theta} \nabla_{\theta} m(x) \| + \|\nabla_{\theta}\nabla_{\theta} \nabla_{\theta} \nabla_{x}m(x) \| \leq g_2(x),$
\end{enumerate}
where $f_1,f_2 \in L^2(\Q)$,$g_1,g_2 \in L^1(\Q)$ are continuous.
Suppose also that the information tensor $g$ is invertible at $\theta_*^{\dksd}$.  Then 
\begin{talign}
 \sqrt{n} \left( \hat \theta_n^{\dksd}- \theta_*^{\dksd} \right) \xrightarrow[]{d} \mathcal N \left(0,g_{\dksd}^{-1}(\theta_*^{\dksd})\Sigma_{\dksd} g^{-1}_{\dksd}(\theta_*^{\dksd}) \right),
 \end{talign}
 where $\Sigma_{\dksd} = \int_\X 
 \left(\int_\X  \nabla_{\theta} \stein_{\theta_*^{\dksd}}(x,y)\dd \Q(y) \right)
\otimes  \left( \int_\X  \nabla_{\theta} \stein_{\theta_*^{\dksd}}(x,z)\dd \Q(z) \right)\dd \Q(x)$.
\end{theorem}
See \cref{appendix:asymptotic-derivations} for proofs. For both results, the assumptions on the kernel are satisfied by most kernels common in the literature, such as Gaussian, inverse-multiquadric (IMQ) and any Mat\'ern kernels with smoothness greater than $2$. Similarly, the assumptions on the model are very weak given that the diffusion tensor $m$ can be adapted to guarantee consistency and asymptotic normality. 

We now prove analogous results for DSM. This time we show weak consistency, i.e. convergence in probability: $ \hat{\theta}_n^{\dsm} \xrightarrow{p} \theta_*^{\dsm} \defn  \text{argmin}_{\theta \in \Theta} \dsm_{m}(\mathbb{Q} \|\P_{\theta}) = \text{argmin}_{\theta \in \Theta} \int_{\X} F_\theta(x) \mathrm{d}\Q(x)$. This will be a sufficient form of convergence for asymptotic normality.
\begin{theorem}[\textbf{Weak Consistency of DSM}]
Let $\X$ be an open subset of $\R^d$, and $\Theta \subset \R^m$.
Suppose  $\log p_{\theta}(\cdot) \in C^2(\X)$ and $m \in C^1(\X)$,
 and $ \| \nabla_x \log p_{\theta}(x) \|\leq f_1(x)$ for $\mathbb{Q}$-a.e. $x$.
 Suppose also that
 $\| \nabla_x \nabla_x  \log p_{\theta}(x) |\leq f_2(x)$ on any compact set $C \subset \Theta$  for $\mathbb{Q}$-a.e. $x$, where 
 $\| m^{\top} \| f_1 \in L^2(\Q)$,
 $ \| \nabla \cdot (mm^{\top}) \| f_1 \in L^1(\Q)$, $ \| mm^{\top}\|_{\infty} f_2 \in L^1(\Q)$. 
 If either $\Theta$ is compact, or $\Theta$ and $\theta \mapsto F_{\theta}$ are convex and $\theta_*^{\dsm} \in \text{int}(\Theta)$, then $\hat{\theta}^{\dsm}_n$ is weakly consistent for $\theta_*^{\dsm}$.
\end{theorem}

\begin{theorem}[\textbf{Central Limit Theorem for DSM}]
Let $\X,\Theta$ be open subsets of $\R^d$ and $\R^m$ respectively.
Suppose  $\hat{\theta}^{\dsm}_n  \xrightarrow[]{p} \theta_*^{\dsm}$, $\theta \mapsto \log p_\theta(x)$ is twice continuously differentiable on a closed ball $ \bar B(\epsilon,\theta_*^{\dsm}) \subset \Theta$,
and that for $\mathbb{Q}$-a.e. $x \in \mathcal{X}$,
\begin{enumerate}
\item[\textbf{(i)}] $\| m(x) m^{\top}(x)\|+\| \nabla_x \cdot(m(x) m^{\top}(x)) \| \leq f_1(x),$ and $ 
\| \nabla_x \log p_\theta(x) \|
+  \| \nabla_{\theta} \nabla_x \log p_\theta(x) \|+  \| \nabla_{\theta} \nabla_x \nabla_x \log p_\theta(x) \| \leq f_2(x)$, with $f_1f_2, f_1 f_2^2 \in L^2(\Q)$

\item[\textbf{(ii)}] for $\theta \in \bar B(\epsilon,\theta^*) $, $\| \nabla_\theta \nabla_x \log p_\theta \|^2+\| \nabla_x \log p_\theta \| \|\nabla _\theta \nabla_\theta \nabla_x \log p_\theta \| +\| \nabla_\theta\nabla_\theta \nabla_x \log p_\theta \|
 +\|\nabla_\theta \nabla_\theta \nabla_x\nabla_x \log p_\theta \| \leq g_1(x)$, and $f_1g_1 \in L^1(\Q)$.
\end{enumerate}
 Then, if the information tensor is invertible at $\theta_*^{\dsm}$, we have
 \begin{talign}
\sqrt{n} \left( \hat{\theta}^{\dsm}_n- \theta_*^{\dsm}\right) \xrightarrow[]{d} \mathcal N \left (0,g^{-1}_{\dsm}\left(\theta_*^{\dsm}\right)  \Sigma_{\dsm} g^{-1}_{\dsm}\left(\theta_*^{\dsm}\right) \right).
 \end{talign}
 where $  \Sigma_{\dsm} = \int_{\mathcal{X}}  \nabla_{\theta} F_{\theta_*^{\dsm}}(x) \otimes \nabla_{\theta} F_{\theta_*^{\dsm}}(x)  \mathrm{d}\Q(x)$.
\end{theorem}
All of the proofs can be found in \cref{app:DSM-asymptotics}. An important special case covered by our theory is that of natural exponential families, which admit densities of the form $\log p_{\theta}(x)  \propto \langle \theta, T(x)\rangle_{\R^m}+b(x)$.
If $K$ is IPD with bounded derivative up to order $2$, $\nabla T$ has linearly independent rows,  $m$ is invertible, 
and  $  \| \nabla T m\|, \|\nabla_x b \| \|m \|, \| \nabla_x m \|+ \|m \| \in L^2(\Q)$, 
then the sequence of minimum $\dksd$ and $\dsm$ estimators are strongly consistent and asymptotically normal (see \cref{app:CLT-exponential-family}).

Before concluding this section, we turn to a concept of importance to practical inference: robustness when subjected to corrupted data \cite{Huber2009}. We quantify the robustness of DKSD and DSM estimators in terms of their influence function, which can be interpreted as measuring the impact of an infinitesimal perturbation of a distribution $\P$ by a Dirac located at  a point $z \in \X$ on the estimator.  If $\theta_{\mathbb{Q}}$ denotes the unique minimum SD estimator for $\Q$, then the influence functions is given by $\mbox{IF}(z, \Q) \defn \partial_{t}\theta_{\mathbb{Q}_t}|_{t=0}$ if it exists, where $\Q_t = (1-t)\Q + t \delta_{z}$, for $t \in [0,1]$. An estimator is said to be bias robust if $\mbox{IF}(z, \Q)$ is bounded in $z$.

\begin{proposition}[\textbf{Robustness of DKSD estimators}]\label{prop:robustness_dksd}  Suppose that the map $\theta \rightarrow \P_{\theta}$ over $\Theta$ is injective, then $\IF(z,\P_\theta) = g_{\dksd}(\theta)^{-1} \int_X  \nabla_\theta \stein(z,y) \dd \P_\theta(y)
$.
  Moreover, suppose that  $y \mapsto F(x,y)$ is $\Q$-integrable for any $x$, where $F(x,y) = \lVert K(x,y)s_p(y)\rVert$,  $\lVert K(x,y)\nabla_{\theta}s_p(y)\rVert$,  $\lVert \nabla_{x} K(x,y)s_p(y)\rVert$,  $\lVert \nabla_{x} K(x,y)\nabla_{\theta}s_p(y)\rVert$, $\lVert \nabla_{y}\nabla_{x}\left(K(x,y)m(y)\right)\rVert$,$\lVert \nabla_{y}\nabla_{x}\left(K(x,y)\nabla_{\theta}m(y)\right)\rVert$. 
Then if $x \mapsto (\lVert s_p(x) \rVert + \lVert \nabla_\theta s_p(x) \rVert)  \int F(x,y) \mathbb{Q}(dy)|_{\theta_*^{\dksd}}$ is bounded, the $\dksd$ estimators are bias robust:  $\sup_{z\in\mathcal{X}} \|\IF(z,\Q)\| < \infty$.
\end{proposition}

The analogous results for DSM estimators can be found in \cref{appendix:robustness}. 
Consider a Gaussian location model, i.e. $p_{\theta} \propto \exp(-\lVert x - \theta \rVert^2_{2})$, for $\theta \in \R^d$. The Gaussian kernel satisfies the assumptions of \cref{prop:robustness_dksd} so that $\sup_{z}\|\IF(z,\Q)\| < \infty$, even when $m=I$. 
Indeed $\|\IF(z,\P_{\theta})\| \leq C(\theta) e^{- \|z-\theta\|^2/4}\|z-\theta \|$, where $z\mapsto e^{- \|z-\theta\|^2/4}\|z-\theta \|$  is uniformly bounded over $\theta$.
In contrast, the SM estimator has an influence function of the form $\IF(z,\Q) = z - \int_{\X} x \dd \Q(x)$, which is unbounded with respect to $z$, and is thus not robust. This clearly demonstrates the importance of carefully selecting a Stein class for use in minimum SD estimators. An alternative way of inducing robustness is to introduce a spatially decaying diffusion matrix in DSM.  To this end, consider the minimum DSM estimator with scalar diffusion coefficient $m$. Then $\theta_{\text{DSM}} = (\int_{\X} m^2(x) \dd \Q(x))^{-1}\left(\int_{\X} m^2(x) x \dd \Q(x) +  \int_{\X} \nabla m^2(x)\dd \Q(x) \right)$.   A straightforward calculation yields that the associated influence function will be bounded if both $m(x)$ and $\|\nabla m(x)\|$ decay as $\lVert x \rVert \rightarrow \infty$. This clearly demonstrates another significant advantage provided by the flexibility of our family of diffusion SD, where the Stein operator also plays an important role.

\section{Numerical Experiments}\label{numerics-dksd}

In this section, we explore several examples which demonstrate worrying breakpoints for SM, and highlight how these can be straightforwardly handled using KSD, DKSD and DSM.

\subsection{Rough densities: the symmetric Bessel distributions}

\begin{figure}[t]
\begin{center}

\includegraphics[width=0.2\textwidth,clip,trim = 0.2cm 0 0 0]{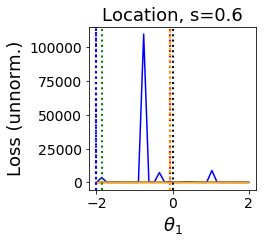}
\includegraphics[width=0.18\textwidth,clip,trim = 0 0 0 0]{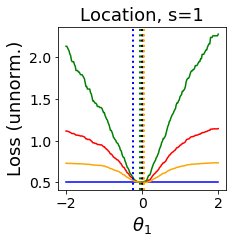}
\includegraphics[width=0.19\textwidth,clip,trim = 0 0 0 0]{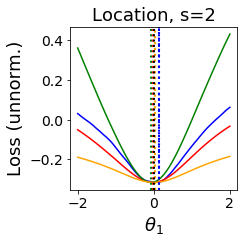}
\includegraphics[width=0.20\textwidth,clip,trim = 9cm 0 0 0]{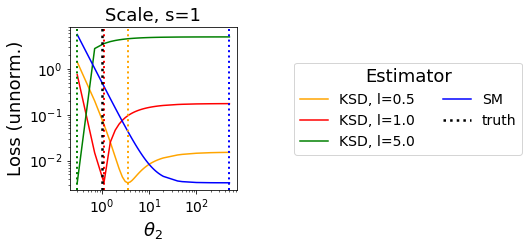} 
\vspace{-5mm}
\end{center}
\caption{\textit{Minimum SD Estimators for the Symmetric Bessel Distribution}. We consider the case where $\theta_1^*=0$ and $\theta_2^*=1$ and $n=500$ for a range of smoothness parameter values $s$ in $d=1$.}
\vspace{-2mm}
\label{fig:symmetric_bessel}
\end{figure}

A major drawback of SM is the smoothness requirement on the target density. However, this can be remedied by choosing alternative Stein classes, as will be demonstrated in the case of the symmetric Bessel distributions. Let $K_{s-d/2}$ denote the modified Bessel function of the second kind with parameter $s-d/2$. This distribution generalises the Laplace distribution \citep{Kotz2001} and has log-density:  
$\log p_{\theta}(x)  \propto (\|x-\theta_1\|_2/\theta_2)^{(s-d/2)} K_{s-d/2}(\|x-\theta_1\|_2/\theta_2)$ where $\theta_1 \in \mathbb{R}^d$ is a location parameter and $\theta_2>0$ a scale parameter. The parameter $s\geq d/2$ encodes smoothness.

We compared SM with KSD based on a Gaussian kernel and a range of lengthscale values in \cref{fig:symmetric_bessel}. These results are based on $n=500$ IID realisations in $d=1$. The case $s=1$ corresponds to a Laplace distribution, and we notice that both SM and KSD are able to obtain a reasonable estimate of the location. For rougher values, for example $s=0.6$, we notice that KSD outperforms SM for certain choices of lengthscales, whereas for $s=2$, SM and KSD are both able to recover the parameter. Analogous results for scale can be found in \cref{appendix:symbessel}, and \cref{appendix:experiments_robustness} illustrates the trade-off between efficiency and robustness on this problem.

\subsection{Heavy-tailed distributions: the non-standardised student-t distributions}

\begin{figure}[t]
\begin{center}
\includegraphics[width=0.25\textwidth,clip,trim = 0 0 0 0]{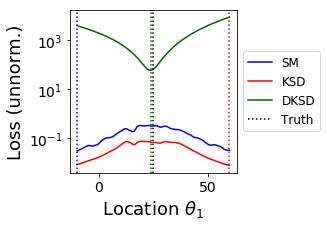}
\includegraphics[width=0.26\textwidth,clip,trim = 0 0 0 0]{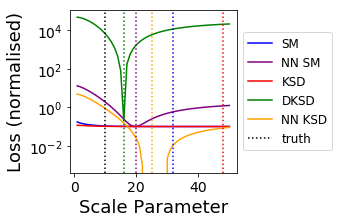}
\includegraphics[width=0.18\textwidth,clip,trim = 0 0 0 0]{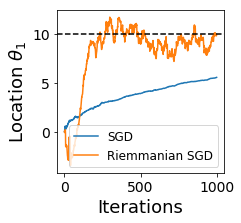}
\includegraphics[width=0.18\textwidth,clip,trim = 0 0 0 0]{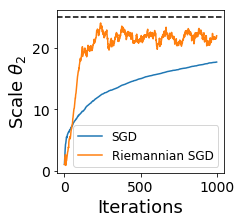} 
\vspace{-5mm}
\end{center}
\caption{\textit{Minimum SD Estimators for Non-standardised Student-t Distributions}. We consider a student-t problem with $\nu=5, \theta_1^*=25, \theta_2^*=10$ and $n=300$. }
\label{fig:studentt}
\vspace{-2mm}
\end{figure}

A second drawback of standard SM is that it is inefficient for heavy-tailed distributions. To demonstrate this, we focus on non-standardised student-t distributions: $p_\theta(x) \propto (1/\theta_2)(1+(1/\nu)\|x-\theta_1\|_2^2/\theta_2^2)^{-(\nu+1)/2}$ where $\theta_1\in \mathbb{R}$ is a location parameter and $\theta_2>0$ a scale parameter. The parameter $\nu$ determines the degrees of freedom: when $\nu = 1$, we have a Cauchy distribution, whereas $\nu = \infty$ gives the Gaussian distribution. For small values of $\nu$, the student-t distribution is heavy-tailed. 

We illustrate SM and KSD for $\nu=5$ in Fig. \ref{fig:studentt}, where we take an IMQ kernel $k(x,y;c,\beta) = (c^2+\|x-y\|^2_2)^{\beta}$ with $c=1.$ and $\beta=-0.5$. This choice of $\nu$ guarantees the first two moments exist, but the distribution is still heavy-tailed. In the left plot, both SM and KSD struggle to recover $\theta_1^*$ when $n=300$, and the loss functions are far from convex. However, DKSD with $m_\theta(x) = 1 + \|x-\theta_1\|^2/\theta_2^2$ can estimate $\theta_1$ very accurately. In the middle left plot, we instead estimate $\theta_2$ with SM, KSD and their correponding non-negative version (NNSM \& NNKSD, $m(x)=x$), which are particularly well suited for scale parameters. NNSM and NNKSD provide improvements on SM and KSD, but DKSD with $m_\theta(x)=((x-\theta_1)/\theta_2)(1+(1/\nu)\|x-\theta_1\|_2^2/\theta_2^2)$ provides significant further gains. On the right-hand side, we also consider the advantage of the Riemannian SGD algorithm over SGD by illustrating them on the KSD loss function with $n=1000$. Both algorithms use constant stepsizes and minibatches of size $50$. As demonstrated, Riemmannian SGD converges within a few dozen iterations, whereas SGD hasn't converged after $1000$ iterations. Additional experiments on the robustness of these estimators is also available in \cref{appendix:experiments_robustness}.

\subsection{Robust estimators for light-tailed distributions: the generalised Gamma distributions}

Our final example demonstrates a third failure mode for SM: its lack of robustness for light-tailed distributions.
We consider generalised gamma location models with likelihoods $p_\theta(x) \propto \exp(-(x-\theta_1)^{\theta_2})$ where $\theta_1$ is a location parameter and $\theta_2$ determines how fast the tails decay. The larger $\theta_2$, the lighter the tails will be and vice-versa. We set $n=300$ and corrupt $80$ points by setting them to the value $x=8$. A robust estimator should obtain a good approximation of $\theta^*$ even under this corruption. The left plot in \cref{fig:gengamma} considers a Gaussian model (i.e. $\theta_2^*=2$); we see that SM is not robust for this very simple model whereas DSM with $m(x)=1/(1+\|x\|^\alpha), \alpha=2$ is robust. The middle plot shows that DKSD with this same $m$ is also robust, and confirms the analytical results of the previous section. Finally, the right plot considers the case $\theta_2^*=5$ and we see that $\alpha$ can be chosen as a function of $\theta_2$ to guarantee robustness. In general, taking $\alpha \geq \theta_2^*-1$ will guarantee a bounded influence function. Such a choice allows us to obtain robust estimators even for models with very light tails.

\begin{figure}[t]
\begin{center}
\includegraphics[width=0.3\textwidth,trim={0.1cm 0 0.1cm 0},clip]{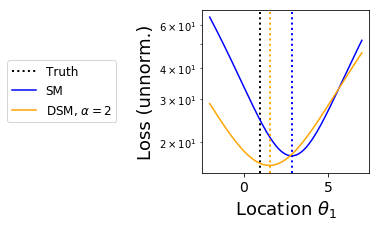}
\hspace{2mm}
\includegraphics[width=0.3\textwidth,trim={0.1cm 0 0.1cm 0},clip]{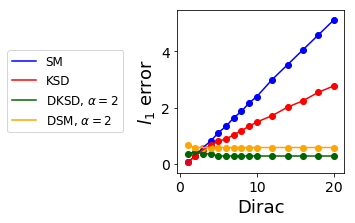}
\hspace{2mm}
\includegraphics[width=0.3\textwidth,trim={0.1cm 0 0.1cm 0},clip]{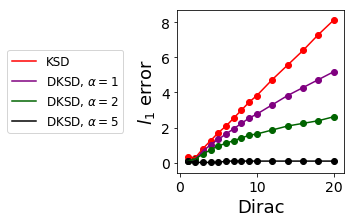}
\end{center}
\vspace{-5mm}
\caption{\textit{Minimum SD Estimators for Generalised Gamma Distributions under Corruption}. We consider the case where $\theta_1^*=0$ and $\theta_2^*=2$ (left and middle) or $\theta_2^*=5$ (right). Here $n=300$.}
\label{fig:gengamma}
\vspace{-3mm}
\end{figure}

\subsection{Efficient estimators for a simple unnormalised model}

\begin{wrapfigure}{r}{0.3\textwidth}
\vspace{-5mm}
   \centering
   \includegraphics[width=0.28\textwidth]{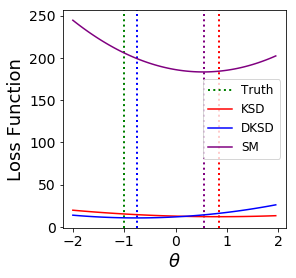}
   \vspace{-5mm}
   \caption{\textit{Estimators for a Simple Intractable Model}}
\end{wrapfigure}

Finally we consider a simple intractable model from~\cite{Liu2018Fisher}: $p_\theta(x) \propto \exp(\eta(\theta)^\top \psi(x))$ where $\psi(x) = (\sum_{i=1}^d x_i^2,\sum_{i=3}^d x_1 x_i,\tanh(x))^\top$ and $\tanh$ is applied elementwise to $x$ and $\eta(\theta) = (-0.5,0.2,0.6,0,0,0,\theta,0)$. This model is intractable since we cannot easily compute its normalisation constant due to the difficulty of integrating the unnormalised part of the model. Our results based on $n=200$ samples show that DKSD with $m(x)=\text{diag}(1/(1+x))$ is able to recover $\theta^*=-1$, whereas both SM and KSD provide less accurate estimates of the parameter. This illustrates yet again that a judicious choice of diffusion matrix can significantly improve the efficiency of our estimators.

\section{Conclusion}

This paper introduced a general approach for constructing minimum distance estimators based on Stein's method, and demonstrated that many popular inference schemes can be recovered as special cases. This class of algorithms gives us additional flexibility through the choice of an operator and function space (the Stein operator and Stein class), which can be used to tailor the inference scheme to trade-off efficiency and robustness. However, this paper only scratches the surface of what is possible with minimum SD estimators. Looking ahead, it will be interesting to identify diffusion matrices which increase efficiency for important classes of problems in machine learning. 
One example on which we foresee progress are the product of student-t experts models \citep{Kingma2010,Swersky2011,Welling2003}, whose heavy tails render estimation challenging for SM. Advantages could also be found for other energy models, such as large graphical models where the kernel could be adapted to the graph \citep{Vishwanathan2010}.

\subsubsection*{Acknowledgments}
AB was supported by a Roth scholarship from the Department of Mathematics at Imperial College London. FXB was supported by the EPSRC grants [EP/L016710/1, EP/R018413/1]. AD and MG were supported by the Lloyds Register Foundation Programme on Data-Centric Engineering, the UKRI Strategic Priorities Fund under the EPSRC Grant [EP/T001569/1] and the Alan Turing Institute under the EPSRC grant [EP/N510129/1]. MG was supported by the EPSRC grants [EP/J016934/3, EP/K034154/1, EP/P020720/1, EP/R018413/1].
Finally, we thank Jiaxin Shi for exposing that MPFL has two distinct definitions.

\bibliography{Ref-KSD_estimators}

\begin{thebibliography}{71}
\providecommand{\natexlab}[1]{#1}
\providecommand{\url}[1]{\texttt{#1}}
\expandafter\ifx\csname urlstyle\endcsname\relax
  \providecommand{\doi}[1]{doi: #1}\else
  \providecommand{\doi}{doi: \begingroup \urlstyle{rm}\Url}\fi

\bibitem[Amari(1998)]{Amari1998}
S.-I. Amari.
\newblock {Natural gradient works efficiently in learning}.
\newblock \emph{Neural Computation}, 10\penalty0 (2):\penalty0 251--276, 1998.

\bibitem[Amari(2016)]{Amari2016}
S.-I. Amari.
\newblock \emph{{Information Geometry and Its Applications}}, volume 194.
\newblock Springer, 2016.

\bibitem[Barbour and Chen(2005)]{Barbour2005}
A.~Barbour and L.~H.~Y. Chen.
\newblock \emph{{An introduction to Stein's method}}.
\newblock Lecture Notes Series, Institute for Mathematical Sciences, National
  University of Singapore, 2005.

\bibitem[Basu et~al.(2011)Basu, Shioya, and Park]{Basu2011}
A.~Basu, H.~Shioya, and C.~Park.
\newblock \emph{{Statistical Inference: The Minimum Distance Approach}}.
\newblock CRC Press, 2011.

\bibitem[Berlinet and Thomas-Agnan(2004)]{Berlinet2004}
A.~Berlinet and C.~Thomas-Agnan.
\newblock \emph{{Reproducing Kernel Hilbert Spaces in Probability and
  Statistics}}.
\newblock Springer Science+Business Media, New York, 2004.

\bibitem[Bernton et~al.(2019)Bernton, Jacob, Gerber, and Robert]{Bernton2017}
E.~Bernton, P.~E. Jacob, M.~Gerber, and C.~P. Robert.
\newblock {Approximate Bayesian computation with the Wasserstein distance}.
\newblock \emph{Journal of the Royal Statistical Society Series B: Statistical
  Methodology}, 81\penalty0 (2):\penalty0 235--269, 2019.

\bibitem[Bonnabel(2013)]{bonnabel2013stochastic}
S.~Bonnabel.
\newblock {Stochastic gradient descent on Riemannian manifolds}.
\newblock \emph{{IEEE Transactions on Automatic Control}}, 58\penalty0
  (9):\penalty0 2217--2229, 2013.

\bibitem[Briol et~al.(2019)Briol, Barp, Duncan, and Girolami]{Briol2019}
F.-X. Briol, A.~Barp, A.~B. Duncan, and M.~Girolami.
\newblock {Statistical inference for generative models with maximum mean
  discrepancy}.
\newblock \emph{arXiv:1906.05944}, 2019.

\bibitem[Casella and Berger(2001)]{Casella2001}
G.~Casella and R.~Berger.
\newblock \emph{{Statistical Inference}}.
\newblock 2001.

\bibitem[Ceylan and Gutmann(2018)]{ceylan2018conditional}
C.~Ceylan and M.~U. Gutmann.
\newblock Conditional noise-contrastive estimation of unnormalised models.
\newblock \emph{arXiv:1806.03664}, 2018.

\bibitem[Chen et~al.(2011)Chen, Goldstein, and Shao]{Chen2011}
L.~H.~Y. Chen, L.~Goldstein, and Q.-M. Shao.
\newblock \emph{{Normal Approximation by Stein's Method}}.
\newblock Springer, 2011.

\bibitem[Chen et~al.(2018)Chen, Mackey, Gorham, Briol, and Oates]{Chen2018}
W.~Y. Chen, L.~Mackey, J.~Gorham, F.-X. Briol, and C.~J. Oates.
\newblock {Stein points}.
\newblock In \emph{Proceedings of the International Conference on Machine
  Learning, PMLR 80:843-852}, 2018.

\bibitem[Chen et~al.(2019)Chen, Barp, Briol, Gorham, Girolami, Mackey, and
  Oates]{Chen2019}
W.~Y. Chen, A.~Barp, F.-X. Briol, J.~Gorham, M.~Girolami, L.~Mackey, and C.~J.
  Oates.
\newblock {Stein point Markov chain Monte Carlo}.
\newblock In \emph{International Conference on Machine Learning, PMLR 97},
  pages 1011--1021, 2019.

\bibitem[Chwialkowski et~al.(2016)Chwialkowski, Strathmann, and
  Gretton]{Chwialkowski2016}
K.~Chwialkowski, H.~Strathmann, and A.~Gretton.
\newblock {A kernel test of goodness of fit}.
\newblock In \emph{International Conference on Machine Learning}, pages
  2606--2615, 2016.

\bibitem[Dawid and Musio(2014)]{Dawid2014}
A.~P. Dawid and M.~Musio.
\newblock {Theory and applications of proper scoring rules}.
\newblock \emph{Metron}, 72\penalty0 (2):\penalty0 169--183, 2014.
\newblock ISSN 2281695X.
\newblock \doi{10.1007/s40300-014-0039-y}.

\bibitem[Dawid et~al.(2016)Dawid, Musio, and Ventura]{Dawid2016}
A.~P. Dawid, M.~Musio, and L.~Ventura.
\newblock {Minimum scoring rule inference}.
\newblock \emph{Scandinavian Journal of Statistics}, 43\penalty0 (1):\penalty0
  123--138, 2016.

\bibitem[Detommaso et~al.(2018)Detommaso, Cui, Marzouk, Spantini, and
  Scheichl]{Detommaso2018}
G.~Detommaso, T.~Cui, Y.~Marzouk, A.~Spantini, and R.~Scheichl.
\newblock A stein variational newton method.
\newblock In \emph{Advances in Neural Information Processing Systems 31}, pages
  9169--9179. 2018.

\bibitem[Dziugaite et~al.(2015)Dziugaite, Roy, and Ghahramani]{Dziugaite2015}
G.~K. Dziugaite, D.~M. Roy, and Z.~Ghahramani.
\newblock {Training generative neural networks via maximum mean discrepancy
  optimization}.
\newblock In \emph{Uncertainty in Artificial Intelligence}, 2015.

\bibitem[Frogner et~al.(2015)Frogner, Zhang, Mobahi, Araya-Polo, and
  Poggio]{Frogner2015}
C.~Frogner, C.~Zhang, H.~Mobahi, M.~Araya-Polo, and T.~Poggio.
\newblock {Learning with a Wasserstein loss}.
\newblock In \emph{Advances in Neural Information Processing Systems}, pages
  2053--2061, 2015.

\bibitem[Gabay(1982)]{gabay1982minimizing}
D.~Gabay.
\newblock Minimizing a differentiable function over a differential manifold.
\newblock \emph{Journal of Optimization Theory and Applications}, 37\penalty0
  (2):\penalty0 177--219, 1982.

\bibitem[Genevay et~al.(2018)Genevay, Peyr{\'{e}}, and Cuturi]{Genevay2017}
A.~Genevay, G.~Peyr{\'{e}}, and M.~Cuturi.
\newblock {Learning generative models with Sinkhorn divergences}.
\newblock In \emph{Proceedings of the Twenty-First International Conference on
  Artificial Intelligence and Statistics, PMLR 84}, pages 1608--1617, 2018.

\bibitem[Geyer(1994)]{geyer1994convergence}
C.~J. Geyer.
\newblock On the convergence of {M}onte {C}arlo maximum likelihood
  calculations.
\newblock \emph{Journal of the Royal Statistical Society: Series B
  (Methodological)}, 56\penalty0 (1):\penalty0 261--274, 1994.

\bibitem[Gorham and Mackey(2015)]{Gorham2015}
J.~Gorham and L.~Mackey.
\newblock {Measuring sample quality with Stein's method}.
\newblock In \emph{Advances in Neural Information Processing Systems}, pages
  226--234, 2015.

\bibitem[Gorham and Mackey(2017)]{Gorham2017}
J.~Gorham and L.~Mackey.
\newblock {Measuring sample quality with kernels}.
\newblock In \emph{Proceedings of the International Conference on Machine
  Learning}, pages 1292--1301, 2017.

\bibitem[Gorham et~al.(2016)Gorham, Duncan, Mackey, and Vollmer]{Gorham2016}
J.~Gorham, A.~Duncan, L.~Mackey, and S.~Vollmer.
\newblock {Measuring sample quality with diffusions}.
\newblock \emph{arXiv:1506.03039. To appear in Annals of Applied Probability.},
  2016.

\bibitem[Gutmann and Hyv{\"a}rinen(2010)]{gutmann2010noise}
M.~U. Gutmann and A.~Hyv{\"a}rinen.
\newblock Noise-contrastive estimation: A new estimation principle for
  unnormalized statistical models.
\newblock In \emph{Proceedings of the Thirteenth International Conference on
  Artificial Intelligence and Statistics}, pages 297--304, 2010.

\bibitem[Gutmann and Hyvarinen(2012)]{Gutmann2012}
M.~U. Gutmann and A.~Hyvarinen.
\newblock {Noise-contrastive estimation of unnormalized statistical models,
  with applications to natural image statistics}.
\newblock \emph{Journal of Machine Learning Research}, 13:\penalty0 307--361,
  2012.

\bibitem[Hinton(2002)]{Hinton2002}
G.~E. Hinton.
\newblock {Training products of experts by minimizing contrastive divergence}.
\newblock \emph{Neural Computation}, 14\penalty0 (8):\penalty0 1771--1800,
  2002.

\bibitem[Hoeffding(1948)]{hoeffding1948class}
W.~Hoeffding.
\newblock A class of statistics with asymptotically normal distribution.
\newblock \emph{The Annals of Mathematical Statistics}, pages 293--325, 1948.

\bibitem[Hoeffding(1961)]{hoeffding1961strong}
W.~Hoeffding.
\newblock The strong law of large numbers for {U}-statistics.
\newblock Technical report, North Carolina State University Department of
  Statistics, 1961.

\bibitem[Huber and Ronchetti(2009)]{Huber2009}
P.~J. Huber and E.~M. Ronchetti.
\newblock \emph{{Robust Statistics}}.
\newblock Wiley, 2009.

\bibitem[Huggins and Mackey(2018)]{huggins2018random}
J.~Huggins and L.~Mackey.
\newblock Random feature stein discrepancies.
\newblock In \emph{Advances in Neural Information Processing Systems}, pages
  1899--1909, 2018.

\bibitem[Hyv{\"a}rinen(1999)]{hyvarinen1999sparse}
A.~Hyv{\"a}rinen.
\newblock Sparse code shrinkage: Denoising of nongaussian data by maximum
  likelihood estimation.
\newblock \emph{Neural computation}, 11\penalty0 (7):\penalty0 1739--1768,
  1999.

\bibitem[Hyv{\"{a}}rinen(2006)]{Hyvarinen2006}
A.~Hyv{\"{a}}rinen.
\newblock {Estimation of non-normalized statistical models by score matching}.
\newblock \emph{Journal of Machine Learning Research}, 6:\penalty0 695--708,
  2006.

\bibitem[Hyv{\"{a}}rinen(2007)]{Hyvarinen2007}
A.~Hyv{\"{a}}rinen.
\newblock {Some extensions of score matching}.
\newblock \emph{Computational Statistics and Data Analysis}, 51\penalty0
  (5):\penalty0 2499--2512, 2007.

\bibitem[Jitkrittum et~al.(2017)Jitkrittum, Xu, Szabo, Fukumizu, and
  Gretton]{Jitkrittum2017}
W.~Jitkrittum, W.~Xu, Z.~Szabo, K.~Fukumizu, and A.~Gretton.
\newblock {A linear-time kernel goodness-of-fit test}.
\newblock In \emph{Advances in Neural Information Processing Systems}, pages
  261--270, 2017.

\bibitem[Karakida et~al.(2016)Karakida, Okada, and Amari]{Karakida2016}
R.~Karakida, M.~Okada, and S.-I. Amari.
\newblock {Adaptive natural gradient learning algorithms for unnormalized
  statistical models}.
\newblock \emph{Artificial Neural Networks and Machine Learning - ICANN}, 2016.

\bibitem[Kingma and LeCun(2010)]{Kingma2010}
D.~P. Kingma and Y.~LeCun.
\newblock {Regularized estimation of image statistics by score matching}.
\newblock In \emph{Advances in Neural Information Processing Systems}, pages
  1126--1134, 2010.

\bibitem[K{\"o}ster and Hyv{\"a}rinen(2010)]{koster2010two}
U.~K{\"o}ster and A.~Hyv{\"a}rinen.
\newblock A two-layer model of natural stimuli estimated with score matching.
\newblock \emph{Neural Computation}, 22\penalty0 (9):\penalty0 2308--2333,
  2010.

\bibitem[Kotz et~al.(2001)Kotz, Kozubowski, and Podgorski]{Kotz2001}
S.~Kotz, T.~J. Kozubowski, and K.~Podgorski.
\newblock \emph{{The Laplace Distribution and Generalizations}}.
\newblock Springer, 2001.

\bibitem[Li and Turner(2018)]{Li2018implicit}
Y.~Li and R.~E. Turner.
\newblock {Gradient estimators for implicit models}.
\newblock In \emph{International Conference on Learning Representations}, 2018.

\bibitem[Li et~al.(2015)Li, Swersky, and Zemel]{Li2015GMMN}
Y.~Li, K.~Swersky, and R.~Zemel.
\newblock {Generative moment matching networks}.
\newblock In \emph{Proceedings of the International Conference on Machine
  Learning}, volume~37, pages 1718--1727, 2015.

\bibitem[Liu and Zhu(2017)]{Liu2017geomSVGD}
C.~Liu and J.~Zhu.
\newblock {Riemannian Stein Variational Gradient Descent for Bayesian
  Inference}.
\newblock \penalty0 (i), 2017.
\newblock URL \url{http://arxiv.org/abs/1711.11216}.

\bibitem[Liu and Wang(2016)]{Liu2016}
Q.~Liu and D.~Wang.
\newblock {Stein variational gradient descent: A general purpose Bayesian
  inference algorithm}.
\newblock In \emph{Advances in Neural Information Processing Systems}, 2016.

\bibitem[Liu and Wang(2017)]{liu2017learning}
Q.~Liu and D.~Wang.
\newblock Learning deep energy models: Contrastive divergence vs. amortized
  mle.
\newblock \emph{arXiv preprint arXiv:1707.00797}, 2017.

\bibitem[Liu et~al.(2016)Liu, Lee, and Jordan]{Liu2016testing}
Q.~Liu, J.~Lee, and M.~Jordan.
\newblock {A kernelized Stein discrepancy for goodness-of-fit tests}.
\newblock In \emph{Proceedings of the International Conference on Machine
  Learning}, pages 276--284, 2016.

\bibitem[Liu et~al.(2018)Liu, Kanamori, Jitkrittum, and Chen]{Liu2018Fisher}
S.~Liu, T.~Kanamori, W.~Jitkrittum, and Y.~Chen.
\newblock {Fisher efficient inference of intractable models}.
\newblock \emph{arXiv:1805.07454}, 2018.

\bibitem[Lyu(2009)]{Lyu2009}
S.~Lyu.
\newblock {Interpretation and generalization of score matching}.
\newblock In \emph{Conference on Uncertainty in Artificial Intelligence}, pages
  359--366, 2009.

\bibitem[Ma and Barber(2017)]{Ma2017}
C.~Ma and D.~Barber.
\newblock {Black-box Stein divergence minimization for learning latent variable
  models}.
\newblock \emph{Advances in Approximate Bayesian Inference, NIPS 2017
  Workshop}, 2017.

\bibitem[Mackey and Gorham(2016)]{mackey2016multivariate}
L.~Mackey and J.~Gorham.
\newblock Multivariate {S}tein factors for a class of strongly log-concave
  distributions.
\newblock \emph{Electronic Communications in Probability}, 21, 2016.

\bibitem[Mardia et~al.(2016)Mardia, Kent, and Laha]{mardia2016score}
K.~V. Mardia, J.~T. Kent, and A.~K. Laha.
\newblock Score matching estimators for directional distributions.
\newblock \emph{arXiv preprint arXiv:1604.08470}, 2016.

\bibitem[Micchelli and Pontil(2005)]{micchelli2005learning}
C.~A. Micchelli and M.~Pontil.
\newblock On learning vector-valued functions.
\newblock \emph{Neural computation}, 17\penalty0 (1):\penalty0 177--204, 2005.

\bibitem[Micheli and Glaunes(2013)]{micheli2013matrix}
M.~Micheli and J.~A. Glaunes.
\newblock Matrix-valued kernels for shape deformation analysis.
\newblock \emph{arXiv preprint arXiv:1308.5739}, 2013.

\bibitem[Mnih and Teh(2012)]{mnih2012fast}
A.~Mnih and Y.~W. Teh.
\newblock {A fast and simple algorithm for training neural probabilistic
  language models}.
\newblock In \emph{Proceedings of the International Conference on Machine
  Learning}, pages 419--426, 2012.

\bibitem[Muller(1997)]{Muller1997}
A.~Muller.
\newblock {Integral probability metrics and their generating classes of
  functions}.
\newblock \emph{Advances in Applied Probability}, 29\penalty0 (2):\penalty0
  429--443, 1997.

\bibitem[Newey and McFadden(1994)]{newey1994large}
W.~K. Newey and D.~McFadden.
\newblock Large sample estimation and hypothesis testing.
\newblock \emph{Handbook of Econometrics}, 4:\penalty0 2111--2245, 1994.

\bibitem[Oates et~al.(2017)Oates, Girolami, and Chopin]{Oates2017}
C.~J. Oates, M.~Girolami, and N.~Chopin.
\newblock {Control functionals for Monte Carlo integration}.
\newblock \emph{Journal of the Royal Statistical Society B: Statistical
  Methodology}, 79\penalty0 (3):\penalty0 695--718, 2017.

\bibitem[Pardo(2005)]{Pardo2005}
L.~Pardo.
\newblock \emph{{Statistical Inference Based on Divergence Measures}}, volume
  170.
\newblock Chapman and Hall/CRC, 2005.

\bibitem[Pigola and Setti(2014)]{pigola2014global}
S.~Pigola and A.~G. Setti.
\newblock Global divergence theorems in nonlinear {PDE}s and geometry.
\newblock \emph{Ensaios Matem{\'a}ticos}, 26:\penalty0 1--77, 2014.

\bibitem[Ranganath et~al.(2016)Ranganath, Altosaar, Tran, and
  Blei]{Ranganath2016}
R.~Ranganath, J.~Altosaar, D.~Tran, and D.~M. Blei.
\newblock {Operator variational inference}.
\newblock In \emph{Advances in Neural Information Processing Systems}, pages
  496--504, 2016.

\bibitem[Roth and Black(2009)]{roth2009fields}
S.~Roth and M.~J. Black.
\newblock Fields of experts.
\newblock \emph{International Journal of Computer Vision}, 82\penalty0
  (2):\penalty0 205, 2009.

\bibitem[Sohl-Dickstein et~al.(2009)Sohl-Dickstein, Battaglino, and
  DeWeese]{sohl2009minimum2009}
J.~Sohl-Dickstein, P.~Battaglino, and M.~R. DeWeese.
\newblock Minimum probability flow learning.
\newblock \emph{arXiv preprint arXiv:0906.4779}, 2009.

\bibitem[Sohl-dickstein et~al.(2011)Sohl-dickstein, Battaglino, and
  DeWeese]{sohl2009minimum}
J.~Sohl-dickstein, P.~Battaglino, and M.~R. DeWeese.
\newblock {Minimum probability flow learning}.
\newblock In \emph{Proceedings of the 28th International Conference on
  International Conference on Machine Learning}, pages 905--912, 2011.

\bibitem[Sriperumbudur et~al.(2017)Sriperumbudur, Fukumizu, Gretton,
  Hyv{\"a}rinen, and Kumar]{Sriperumbudur2017density}
B.~Sriperumbudur, K.~Fukumizu, A.~Gretton, A.~Hyv{\"a}rinen, and R.~Kumar.
\newblock Density estimation in infinite dimensional exponential families.
\newblock \emph{Journal of Machine Learning Research}, 18\penalty0
  (1):\penalty0 1830--1888, 2017.

\bibitem[Sriperumbudur et~al.(2010)Sriperumbudur, Gretton, Fukumizu,
  Sch{\"{o}}lkopf, and Lanckriet]{Sriperumbudur2009}
B.~K. Sriperumbudur, A.~Gretton, K.~Fukumizu, B.~Sch{\"{o}}lkopf, and
  G.~Lanckriet.
\newblock {Hilbert space embeddings and metrics on probability measures}.
\newblock \emph{Journal of Machine Learning Research}, 11:\penalty0 1517--1561,
  2010.

\bibitem[Stein(1972)]{Stein1972}
C.~Stein.
\newblock {A bound for the error in the normal approximation to the
  distribution of a sum of dependent random variables}.
\newblock In \emph{Proceedings of 6th Berkeley Symposium on Mathematical
  Statistics and Probability}, pages 583--602. University of California Press,
  1972.

\bibitem[Swersky et~al.(2011)Swersky, Ranzato, Buchman, Marlin, and
  de~Freitas]{Swersky2011}
K.~Swersky, M.~A. Ranzato, D.~Buchman, B.~M. Marlin, and N.~de~Freitas.
\newblock {On autoencoders and score matching for energy based models}.
\newblock In \emph{International Conference on Machine Learning}, pages
  1201--1208, 2011.

\bibitem[Vishwanathan et~al.(2010)Vishwanathan, Schraudolph, Kondor, and
  Borgwardt]{Vishwanathan2010}
S.~V.~N. Vishwanathan, N.~Schraudolph, R.~Kondor, and K.~Borgwardt.
\newblock {Graph kernels}.
\newblock \emph{Journal of Machine Learning Research}, pages 1201--1242, 2010.

\bibitem[Welling et~al.(2003)Welling, Hinton, and Osindero]{Welling2003}
M.~Welling, G.~Hinton, and S.~Osindero.
\newblock {Learning sparse topographic representations with products of
  student-t distributions}.
\newblock In \emph{Advances in Neural Information Processing Systems}, pages
  1383--1390, 2003.

\bibitem[Wenliang et~al.(2018)Wenliang, Sutherland, Strathmann, and
  Gretton]{wenliang2018learning}
L.~Wenliang, D.~Sutherland, H.~Strathmann, and A.~Gretton.
\newblock Learning deep kernels for exponential family densities.
\newblock \emph{arXiv:1811.08357}, 2018.

\bibitem[Yeo and Johnson(2001)]{yeo2001uniform}
I.-K. Yeo and R.~A. Johnson.
\newblock A uniform strong law of large numbers for {U}-statistics with
  application to transforming to near symmetry.
\newblock \emph{Statistics \& Probability Letters}, 51\penalty0 (1):\penalty0
  63--69, 2001.

\end{thebibliography}
\bibliographystyle{abbrvnat}

\newpage
\appendix

{\Large
\begin{center}
\textbf{Supplementary Material} \label{supp-mat-dksd}
\end{center}
}

This document provides additional details for the paper ``Minimum Stein Discrepancy Estimators''. \cref{appendix:Background_material} contains background technical material required to understand the paper, \cref{appendix:derivations} derives the minimum SD estimators from first principles and \cref{appendix:information-semi-metrics} derives the information metrics for DKSD and DSM. \cref{appendix:asymptotic-derivations} contains proof of all asymptotic results including consistency and central limit theorems for DKSD and DSM, whilst \cref{appendix:robustness} discusses their robustness.

Our derivations will use standard operators from vector calculus which we summarise in \cref{vec-calculus}. We will additionally introduce the following notation. 
We write $f \leqsim g$ if there is a constant $C>0$ for which $f(x) \leq C g(x)$ for all $x$.
We set $\Q f \defn \int f \dd \Q$ and use $\Gamma(\mathcal W, \mathcal Y)$ for the set of maps $\mathcal W \to \mathcal Y$ when $\mathcal W \neq \mathcal \X$.


\section{Background Material} \label{appendix:Background_material}

In this section, we provide background material which is necessary to follow the proofs in the following sections. This includes background in vector calculus, stochastic optimisation over manifolds and vector-valued reproducing kernel Hilbert spaces.

\subsection{Background on Vector Calculus}\label{vec-calculus}

The following section contains background and important identities from vector calculus. For a function $g \in \field{\X,\R}$, $v\in \field{\X, \R^d}$ and $A \in \field{\X, \R^{d\times d}}$ with components
$A_{ij}$, $v_i$, $g$, we have $(\nabla g)_i = \partial_i g$, $(v\cdot A)_i = v_j A_{ji} = (v^{\top}A)_i$, $(\nabla \cdot A)_i \;=\; \partial_j A_{ji}$ which must be interpreted as the components of row-vectors; $(Av)_i \;=\;  A_{ij}v_j$ which are the components of a column vector.
Moreover
 $(\nabla v )_{ij}= \partial_j v_i$,
 $\nabla^2f \defn \nabla(\nabla f)$,   $A:B \defn \langle A,B \rangle = \tr(A^{\top}B)= A_{ij}B_{ij}$.
 We have the following identities (where in the last equality we treat $\nabla \cdot A$ and $\nabla g$ as column vectors)
 \begin{talign}
 \nabla \cdot(gv) &= \partial_i(gv_i)= v_i\partial_i g  +g \partial_i v_i  = (\nabla g) v+ g \nabla \cdot v=\nabla g \cdot v+g \nabla \cdot v,\\
 \nabla \cdot (gA) &= \partial_i(gA_{ij})e_j= \left( A_{ij}\partial_i g + g \partial_i A_{ij} \right)e_j=\nabla g\cdot A + g \nabla \cdot A=\nabla g^{\top}A + g \nabla \cdot A,\\
 \nabla \cdot(Av)&=\partial_i(A_{ij}v_j)=
 (\nabla \cdot A)v+\tr[A \nabla v]= (\nabla \cdot A)\cdot v+ \tr[A \nabla v].
\end{talign}  

\subsection{Background on Norms}

For $F \in \Gamma(\X,\R^{n_1 \times n_2})$
we set $\|F \|_p^p \defn \int \|F(x)\|_p^p \dd \Q(x)$, where $\|F(x) \|_p$ is the vector $p$-norm on $\R^{n_1 \times n_2}$ when $n_2=1$, else it is the induced operator norm.
If $v \in \Gamma(\X,\R^{n_1})$,
then $\| v \|_p^p = \int \| v(x) \|_p^p \dd x = \int \sum_i |v_i(x) |^p \dd x = \sum_i \|v_i \|_p^p$, hence 
 $v  \in L_p(\Q)$ iff $v_{i} \in L_p(\Q)$ for all $i$, and similarly $F  \in L_p(\Q)$ iff $F_{ij} \in L_p(\Q)$ for all $i,j$ since the induced norm $\|F(x) \|_p$ and the vector norm $\|F \|^p_{vec} \defn \sum_{ij}|F_{ij}(x)|^p$ are equivalent.


\subsection{Background on Vector-valued RKHS}\label{vector-RKHS}
A Hilbert space $\H$ of functions $\X \to \R^d$ is a RKHS if $ \| f(x)\|_{\R^d} \leq C_x \| f \|_{\h}$.
It follows that the evaluation ``functional" $\delta_x:\h \to \R^d$ is continuous, for any $x$. Moreover for any $x \in \X,v \in \R^d$, the linear map $f\mapsto v \cdot f(x)$ is cts. By the Riesz representation theorem, there exists $K_xv \in \h$ s.t. $ v \cdot f(x) = \langle K_xv , f \rangle $. From this we see that $K_xv$ is linear in $v$ (turns out linear combinations of $K_{x_i}v_i$ are dense in $\h$), and $K_x^{\ast} = \delta_x$. We define $K: \X \times \X \to \End(\R^d)$ by
\begin{talign} 
 K(x,y)v & \defn (K_yv)(x) =\delta_x \delta^{\ast}_y v. 
 \end{talign}
 It follows that $ K(x,y) = K(y,x)^{\ast}$ and $ u \cdot K(x,y)v = \langle K_y v, K_x u \rangle $. Denote by $e_i$ the $i^\text{th}$ vector in the standard basis of $\R^d$. From this we can get the components of the matrix: 
\begin{talign}
\left( K(x,y) \right)_{ij}& = \langle K_x e_i , K_y e_j \rangle.
\end{talign}
We have for any $v_i,x_j$, $\sum_{j,k} v_j \cdot K(x_j, x_k) v_k \geq 0$.


\subsection{Background on Separable Kernels}\label{separable-kernels}

Consider the $d$ dimensional product space $\mathcal{H}^d$ of function $f: \mathcal X \;\to\; \mathbb R^d$ with components $f_i \in \mathcal H_i$ and $\mathcal{H}_i$ 
is a RKHS  with  kernel $C^2$ kernel $k^i: \mathcal X \times \mathcal X  \;\rightarrow \; \mathbb R$.
  Let $K   : \mathcal X \times \mathcal X \;\rightarrow \; \End(\mathbb R^d) \; \cong \; \mathbb R^{d \times d }$ be the kernel of $\mathcal H^d $ (see \cref{vector-RKHS}). 
Note if $K_x \defn K(x,\cdot): \mathcal X \; \rightarrow \; \End (\mathbb R^d)$, and if $v \in \mathbb R^d$, then $K_x v  \in \mathcal H^d$. The reproducing property then states that $\forall f \in \mathcal H^d$: $\langle f(x), v \rangle_{\mathbb R^d}  = \langle f, K(\cdot,x)v \rangle_{\mathcal H^d}$. 
Moreover for the kernel $ K=\mathrm{diag}(\lambda_1 k^1, \ldots , \lambda_d k^d)$ we will prove below that $\langle f, g \rangle_{\mathcal H^d} = \frac 1 {\lambda_i} \sum_i \langle f_i, g_i \rangle_{\mathcal H_i}$, whereas for $K=Bk$ where $B$ is symmetric and invertible we should have $ \langle f, g \rangle_{\mathcal H^d}  = \sum_{ij} B^{-1}_{ij} \langle f_i, g_j \rangle_{\mathcal H}$.

Given a real-valued kernel $k_i$ on $\X$, consider $K = \text{diag}(\lambda_1 k_1,\ldots, \lambda_n k_n)$. Let $f = \sum_j \delta_{x_j}^{\ast} v_j$. 
Recall this is a dense subset of $\H^d$: we will derive the RKHS norm for this dense subset and by continuity this will hold for any function.
Given the norm, the formula for the inner product will follow by the polarization identity.
We have 
\begin{talign} 
f_i(x) & = \delta_x(f) \cdot e_i = \delta_x  \delta_{x_j}^{\ast} v_j \cdot e_i = K(x,x_j) v_j \cdot e_i \\
& = \text{diag}(\lambda_1 k_1,\ldots, \lambda_n k_n)(x,x_j) v_j \cdot e_i = \lambda_i k_i(x,x_j) v_j^i
\end{talign}
\begin{talign}
\| f \|_{\h_K}^2 = \langle \delta_{x_j}^{\ast} v_j , \delta^{\ast}_{x_l} v_l \rangle_{\h_K} = 
 v_j \cdot K(x_j,x_l) v_l = v^i_j \lambda_i k_i(x_j,x_l) v^i_l
\end{talign}
On the other hand, $ \sum_i \frac 1 {\lambda_i} \langle f_i, f_i \rangle_{k_i} =  \sum_i \frac 1 {\lambda_i} \lambda_i^2 v^i_j v^i_l k_i(x_j,x_l)$. Thus $\|f\|_{\h_K}^2 = \frac{1}{ \lambda_i} \sum_i \langle f_i, f_i \rangle_{k_i}$.

For a symmetric positive definite matrix $B$, consider the kernel on $\h$ $ K(x,y) \defn k(x,y)B $. Let $f = \sum_j \delta_{x_j}^{\ast} v_j$. We have:
\begin{talign}
f_i(x) = \delta_x(f) \cdot e_i = \delta_x  \delta_{x_j}^{\ast} v_j \cdot e_i = K(x,x_j) v_j \cdot e_i = Bv_j \cdot e_i k_{x_j}(x)
\end{talign} 
 This implies $f_i \in \h_k$. Then 
\begin{talign}
\| f \|_{\h_K}^2 = \langle \delta_{x_j}^{\ast} v_j , \delta^{\ast}_{x_l} v_l \rangle_{\h_K} = 
 v_j \cdot K(x_j,x_l) v_l = k(x_j,x_l) v_j \cdot B v_l.
\end{talign}
On the other hand $ \langle f_i, f_j \rangle_k = e_i^{\top} Bv_r e_j^{\top} Bv_s k(x_s,x_r)$. Notice 
\begin{talign}
B^{-1}_{ij} e_i^{\top} Bv_r & = 
 B^{-1}_{ij}  B_{il}v_r^l= \delta_{lj} v_r^l = v^j_r.
 \end{talign}
So we have:
\begin{talign}
B^{-1}_{ij}\langle f_i, f_j \rangle_k = v^j_r e_j^{\top} Bv_s k(x_s,x_r) = v^j_r B_{ja}v^a_s k(x_s,x_r) = v_r \cdot B v_s k(x_s,x_r)
\end{talign}

\subsection{Background on Stochastic Optimisation on Riemmannian Manifolds} \label{app:riemannian-grad-descent}

The gradient flow of a curve $\theta$ on a complete connected Riemannian manifold $\Theta$ (for example a Hilbert space) is the solution to $\dot \theta (t) = - \nabla_{\theta(t) } \sd(\Q\| \P_{\theta}) $, where $\nabla_{\theta}$ is the Riemannian gradient at $\theta$.
Typically
\footnote{See sec 4.4 \cite{gabay1982minimizing} for Riemannian Newton method} the gradient flow is approximated by the update equation
$\theta(t+1)= \exp_{\theta(t)}(-\gamma_t H(Z_t,\theta))$ where $\exp$ is the Riemannian exponential map, $(\gamma_t)$ is a sequence of step sizes with 
$\sum \gamma^2_t < \infty$, $\sum \gamma_t = +\infty$, 
and $H$ is an unbiased estimator of the loss gradient, $\E[H(Z_t,\theta)] = \nabla_\theta \sd(\Q\| \P_{\theta})$.
 When the Riemannian exponential is computationally expensive, it is convenient to replace it by a retration $\mathcal R$, that is a first-order approximation which stays on the manifold. This leads to the update 
 $\theta(t+1)= \mathcal R_{\theta(t)}(-\gamma_t H(Z_t,\theta))$~\cite{bonnabel2013stochastic}.
 When $\Theta$ is a linear manifold
it is common to take $ \mathcal R_{\theta(t)}(-\gamma_t H(Z_t,\theta))
\defn \theta(t)-\gamma_t H(Z_t,\theta(t))$.
 In local coordinates $(\theta^i)$ we have 
 $\nabla_\theta \sd(\Q\| \P_{\theta}) = g(\theta)^{-1}\dd_{\theta}  \sd(\Q\| \P_{\theta}) $, where $\dd_\theta f$ denotes the tuple $(\partial_{\theta^i}f)$, which we will approximate using the biased estimator 
 $H(\{X^t_i\}_i,\theta) \defn \hat g_{ \theta(t)}(\{X_i^t\}_{i=1}^n)^{-1} \dd_{\theta} \hvf \sd(\{X_i^t\}_{i=1}^n \| \P_\theta)$, 
 where $\hat g_{ \theta(t)}(\{X_i^t\}_{i=1}^n)$ is an unbiased estimator for the information matrix $g(\theta(t))$ using a sample $\{X^t_i \}_{i=1}^n\sim \Q$. 
 We thus obtain the following Riemannian gradient descent algorithm 
 \begin{talign}
 \theta(t+1)=\theta(t)- \gamma_t\hat g_{ \theta(t)}(\{X_i^t\}_{i=1}^n)^{-1} \dd_{\theta(t)} \hvf \sd(\{X_i^t\}_{i=1}^n \| \P_\theta).
 \end{talign}
When $\Theta = \R^m$, $\gamma_t = \frac{1}{t} $, $g$ is the Fisher metric and $ \hvf \sd(\{X_i^t\}_{i=1}^n \| \P_\theta)$ is replaced by  $ \hvf{ \operatorname{KL}}  ( \{X_i^t\}_{i=1}^n \|\mathbb{P}_{\theta})$
this recovers the natural gradient descent algorithm \citep{Amari1998}.


 \section{Derivation of Diffusion Stein Discrepancies} \label{appendix:derivations}

 In this appendix, we carefully derive the diffusion SD studied in this paper. We begin by providing details on the diffusion Stein operator, then move on to the DKSD and DSM divergences and corresponding estimators.

For any matrix kernel
 we will show in  \cref{gen-Stein-operator} that $\forall f\in \mathcal H^d$: $\mathcal S^m_p[f](x) = \langle \mathcal S_p^{m,1}K_x,f\rangle_{\mathcal H^d}$.
In \cref{Stein-divergence} we prove that if $x\mapsto \|\mathcal S^{m,1}_p K_x \|_{\mathcal H^d} \in L^1(\mathbb Q)$, 
then
\begin{talign}
\dksd_{K,m}(\mathbb{Q}\|\P)^2  &\;  \defn \; \sup_{\substack{h \in \mathcal{H}^d\\ \lVert h \rVert \leq 1}}\left | \int_{\X} \mathcal{S}^m_p[h]\dd\mathbb{Q}\right |^2 = \int_{\X} \int_{\X} \mathcal{S}^{m,2}_p\mathcal{S}^{m,1}_p K(x, y)\dd\mathbb{Q}(x)\dd\mathbb{Q}(y).
\end{talign}

In \cref{Stein-kernel} we further show the Stein kernel satisfies
\begin{talign} 
\stein(x,y) \defn \mathcal{S}^{m,2}_p\mathcal{S}^{m,1}_p K(x, y)=\frac 1 {p(y)p(x)}  \nabla_y \cdot \nabla_x \cdot  \left( p(x)m(x)K(x,y) m(y)^{\top}p(y)\right).
\end{talign} 

\subsection{Stein Operator}\label{gen-Stein-operator}
By definition for $f\in \field{\X, \R^d}$ and $A \in \field{\X, \R^{d\times d}}$ 
\begin{talign}
\mathcal S_p[f] &=\frac 1 p \nabla \cdot (pmf)=m^{\top} \nabla \log p \cdot f + \nabla \cdot (mf), \\
\mathcal S_p[A] &=\frac 1 p \nabla \cdot (pmA)= m^{\top} \nabla \log p \cdot A  + \nabla \cdot (mA) 
\end{talign}
which are operators $\field{\X, \R^d} \to \field{\X, \R}$ and $\field{\X, \R^{d\times d}} \to \field{\X, \R^d}$ respectively.

\begin{proposition}
Let $\mathcal X$ be an open (connected) subset of $\R^d$, $m$ is continuously differentiable, and 
$K: \mathcal X \times \mathcal X \to \R^{d \times d}$ is the matrix kernel of $\mathcal H^d$. Suppose for any $j \in [1,d]$, $K, \partial_{1^j} \partial_{2^j}K$ are separately continuous and locally bounded. 
Then for any $f \in \mathcal H^d$ 
\begin{align}
\mathcal S_p[f](x) = \langle \mathcal S_p^1[K]|_x,f\rangle_{\mathcal H^d}
\end{align}
\end{proposition}

\begin{proof}

Note that technically the kernel $K$ of $\H^d$ takes value in the set of (bounded) linear operators on $\R^d$, and we view these linear operators as matrices by defining the components
 $\left( K(x,y)\right)_{ji} \defn e_j \cdot K(x,y)e_i$,
where $(e_l)$ is the canonical basis of $\R^d$.
For any $f \in \mathcal H^d$
\begin{align}
 \langle f(x), m(x)^{\top}\nabla \log p(x) \rangle_{\R^d}
 & = \langle f, K(\cdot,x)m(x)^{\top}\nabla \log p(x) \rangle_{\mathcal H^d} \\
&= \langle f, K_x^{\top}m(x)^{\top}\nabla \log p(x) \rangle_{\mathcal H^d}\\
&= \langle f, m(x)^{\top}\nabla \log p(x)\cdot K_x \rangle_{\mathcal H^d} .
\end{align}
Moreover, under these assumptions the RKHS $\mathcal H^d$ is continuously embedded in the topological space $C^1(\mathcal X,\mathbb R^d)$, so its elements are continuously differentiable.
Then for any $f\in \mathcal H^d$, by theorem 2.11~\cite{micheli2013matrix}
 \begin{talign}
 \langle f, \partial_{2^j}K(\cdot,x) e_r \rangle_{\mathcal H^d}= \langle e_r,  \partial_{j} f|_x \rangle_{\R^d} = \partial_{j} f_r|_x.
 \end{talign}
 Hence
\begin{talign}
    \langle f, \nabla \cdot (mK)|_x \rangle_{\mathcal H^d} &= \langle f, \partial_{1^j}(m_{jr}K_{ri})|_x e_i \rangle_{\mathcal H^d}  =
    \langle f, \partial_{j}m_{jr}|_x K_{ri}(x,\cdot)e_i+ m_{jr}(x)\partial_{1^j}K_{ri}|_x e_i \rangle_{\mathcal H^d} \\
    &= 
    \partial_{j}m_{jr}|_x \langle f,  K_{ir}(\cdot,x)e_i \rangle_{\mathcal H^d} + m_{jr}(x)\langle f, \partial_{1^j}K_{ri}(x,\cdot) e_i \rangle_{\mathcal H^d} \\
    & = 
     \partial_{j}m_{jr}|_x \langle f,  K(\cdot,x)e_r \rangle_{\mathcal H^d} + m_{jr}(x)\langle f, \partial_{2^j}K_{ir}(\cdot,x) e_i \rangle_{\mathcal H^d}\\
     & = 
     \partial_{j}m_{jr}|_x \langle f,  K(\cdot,x)e_r \rangle_{\mathcal H^d} + m_{jr}(x)\langle f, \partial_{2^j}K(\cdot,x) e_r \rangle_{\mathcal H^d}\\
     &=
      \partial_{j}m_{jr}|_x f_r(x)+ m_{jr}(x) \partial_{j}f_r|_x \\
      &=
      \langle \nabla \cdot m, f(x) \rangle_{\R^d}+ \tr[m(x)\nabla_x f ] \\
      &= \nabla_x \cdot (mf).
\end{talign}

Therefore, we conclude that $\mathcal S_p[f](x) = \langle \steinop_p^1K_x,f\rangle_{\mathcal H^d}$ where $\steinop_p^1K_x \defn \steinop_p^1[K]|_x$ means applying $\steinop_p$ to the first entry of $K$ and evaluate it $x$, so informally
$\steinop_p^1[K]|_x : y \mapsto \frac1 p \nabla_x \cdot\left(p(x)m(x)K(x,y)\right)$.
\end{proof}

\subsection{Diffusion Kernel Stein Discrepancies}\label{Stein-divergence}

\begin{proposition}
Suppose $\mathcal S_p[f](x) = \langle \steinop_p^1[K]|_x,f\rangle_{\mathcal H^d}$ for any $f \in \mathcal H^d$. 
Let  $m$ and $K$ be $C^2$, and $x \mapsto \steinop_p K_x$ be $\Q $-Bochner integrable.
Then 
\begin{talign}
 \dksd_{K,m}(\mathbb{Q},\P)^2 =  \int_{\X} \int_{\X} \steinop^2_p\steinop^1_p K(x, y)\dd\mathbb{Q}(x)\dd\mathbb{Q}(y).
\end{talign}
\end{proposition}
\begin{proof}

Let us identify $\mathcal H_1 \otimes \mathcal H_2 \;\cong\; L(\mathcal H_1 \times \mathcal H_2, \mathbb R) \; \cong \; L( \mathcal H_2, \mathcal H_1)$
 with $(v_1 \otimes v_2) \sim v_1 \langle v_2 ,\cdot \rangle_{\mathcal H_2}$ (since $\mathcal H_2 \;\cong \; \mathcal H_2^{\ast}$), 
  so that  $(v_1 \otimes v_2)u_2  \; \defn  \;v_1 \langle v_2,u_2 \rangle_{\mathcal H_2}$ 
  (here $L(V,W)$ is the space of linear maps from $V$ to $W$). Then 
\begin{talign}
  \metric{u_1 \otimes u_2}{v_1 \otimes v_2}_{HS} &\defn \langle u_1,v_1 \rangle_{\mathcal H_1}\langle u_2,v_2 \rangle_{\mathcal H_2} \; =\;  \left \langle u_1, (v_1\otimes u_2)v_2 \right \rangle_{\mathcal H_1}.
\end{talign}   
 For simplicity we will  write $\steinop_p K_x \defn  \steinop_p^1[K]|_x$.
  Using the fact $x \mapsto \steinop_p K_x$ is $\Q $-Bochner integrable, 
  then by Cauchy-Schwartz $x \mapsto  \langle h, \steinop_p K_x \rangle_{\mathcal H^d}$ is $\Q$-integrable. Then
\begin{talign}
	\dksd_{K,m}(\mathbb{Q},\P)^2 
  & =  \sup_{\substack{h \in \mathcal{H}^d\\ \lVert h \rVert \leq 1}}\left\langle \int_{\X} \steinop_p[h](x)\dd\mathbb{Q}(x), \int_{\X} \steinop_p[h](y)\dd\mathbb{Q}(y)\right\rangle_\R\\
		&  =  \sup_{\substack{h \in \mathcal{H}^d\\ \lVert h \rVert \leq 1}} \int_{\X} \langle h, \steinop_p K_x \rangle_{\mathcal H^d} \dd\mathbb{Q}(x) \int_{\X} \langle h, \steinop_p K_y \rangle_{\mathcal H^d}\dd\mathbb{Q}(y)\\
		&  =  \sup_{\substack{h \in \mathcal{H}^d\\ \lVert h \rVert \leq 1}}  \int_{\X}  \int_{\X} \left\langle h, \steinop_p K_x \right\rangle_{\mathcal H^d}\langle h, \steinop_p K_y \rangle_{\mathcal H^d}  \dd\mathbb{Q}(x)\dd\mathbb{Q}(y)\\
		& = \sup_{\substack{h \in \mathcal{H}^d\\ \lVert h \rVert \leq 1}}  \int_{\X}  \int_{\X} \left\langle h, \steinop_p K_x\otimes \steinop_p K_y h \right\rangle_{\mathcal H^d}  \dd\mathbb{Q}(x)\dd\mathbb{Q}(y)\\
		&  = \sup_{\substack{h \in \mathcal{H}^d\\ \lVert h \rVert \leq 1}}  \int_{\X} \int_{\X} \left\langle h \otimes h, \steinop_p K_x\otimes \steinop_p K_y  \right\rangle_{HS}  \dd\mathbb{Q}(x)\dd\mathbb{Q}(y) \\
		\end{talign}
Moreover $\int_{\X} \|  \steinop_p K_x\otimes \steinop_p K_y  \|_{HS} \dd\Q( x) \dd\Q( y ) < \infty$,
since 
\begin{talign}
& \int_{\X}  \|  \steinop_p K_x\otimes \steinop_p K_y  \|_{HS} \dd \Q( x)\otimes \dd \Q( y )\\
 &= \int_{\X} \int_{\X}  \sqrt{\metric{\steinop_p K_x}{\steinop_p K_x}_{\mathcal{H}^d} \metric{\steinop_p K_y}{\steinop_p K_y}_{\mathcal{H}^d}}\dd \Q(x) \dd \Q (y)\\
&= \left ( \int_{\X}  \sqrt{\metric{\steinop_p K_x}{\steinop_p K_x}_{\mathcal{H}^d} }\dd \Q(x)  
\right )^2 \\
& = \left ( \int_{\X}  \| \steinop_p K_x \|_{\mathcal{H}^d} \dd \Q(x)  
\right )^2 < \infty
\end{talign}
since by assumption $x \mapsto \steinop_p K_x$ is $\Q $-Bochner integrable. Thus
\begin{talign}
	\dksd_{K,m}(\mathbb{Q},\P)^2	&\;  = \; \sup_{\substack{h \in \mathcal{H}^d\\ \lVert h \rVert \leq 1}} \left\langle h \otimes h,  \int_{\X}  \int_{\X}  \steinop_p K_x\otimes \steinop_p K_y   \dd\mathbb{Q}(x)\dd\mathbb{Q}(y)  \right\rangle_{HS}\\
		& =   \left\| \int_{\X}  \int_{\X}  \steinop_p K_x\otimes \steinop_p K_y   \dd\mathbb{Q}(x)\dd\mathbb{Q}(y)  \right\|_{HS} \\
		& =  
		 \left\| \int_{\X} \steinop_p K_x \dd\mathbb{Q}(x)\otimes \int_{\X} \steinop_p K_y   \dd\mathbb{Q}(y)  \right\|_{HS}\\
		 & =  \left\| \int_{\X}    \steinop_p K_x \dd\mathbb{Q}(x) \right\|^2_{\mathcal H^d} \\
		 & =   \left\langle \int_{\X}    \steinop_p K_x \dd\mathbb{Q}(x) , \int_{\X}    \steinop_p K_y \mathbb{Q}(\dd y) \right\rangle_{\mathcal H^d} 
		  \\
		 & =  \int_{\X}   \int_{\X}    \left\langle  \steinop_p K_x  ,     \steinop_p K_y  \right\rangle_{\mathcal H^d} \dd\mathbb{Q}(x) \dd\mathbb{Q}(y)
		 \\
		&  =  \int_{\X} \int_{\X}\steinop^2_p\steinop^1_p K(x, y)\dd\mathbb{Q}(x)\dd\mathbb{Q}(y).
\end{talign}
To show the penultimate equality (exchange integral and inner product),
we use the fact $  \steinop_p K_x$ is $\Q$-Bochner integrable, 
and that the operator $W: f \mapsto \langle f, \int_{\X} \steinop_p K_y \mathbb{Q}(\dd y) \rangle_{\mathcal H^d}$ is bounded, from which it follows that
\begin{talign}
\left\langle \int_{\X}  \steinop_p K_x \dd\mathbb{Q}(x) , \int_{\X} \steinop_p K_y \mathbb{Q}(\dd y) \right \rangle_{\mathcal H^d} 
& = W \left[ \int_{\X} \steinop_p K_x \dd\mathbb{Q}(x) \right] = 
\int_{\X} W \left[ \steinop_p K_x \dd\mathbb{Q}(x) \right] \\
&=\int_{\X} \left \langle  \steinop_p K_x  ,   \int_{\X} \steinop_p K_y \dd\mathbb{Q}(y) \right \rangle_{\mathcal H^d} \dd\mathbb{Q}(x)  \\
&= \int_{\X} \int_{\X}  \left\langle \steinop_p K_x  , \steinop_p K_y\right \rangle_{\mathcal H^d} \dd\mathbb{Q}(x) \dd\mathbb{Q}(y)
\end{talign}
Hence $\dksd_{K,m}(\mathbb{Q},\P)^2 =  \int_{\X} \int_{\X} \steinop^2_p\steinop^1_p K(x, y)\dd\mathbb{Q}(x)\dd\mathbb{Q}(y)$.

Note that from this proof we have 
\begin{talign}
\stein (x,y) \defn \steinop^2_p\steinop^1_p K(x, y)=
 \left\langle \steinop_p K_x  , \steinop_p K_y\right \rangle_{\mathcal H^d},
\end{talign}
which shows the map 
$\phi: \X \to \H^d$, $\phi(x) \defn   \steinop_p^1[K]|_x$ is a feature map (more precisely it is dual to the feature map)
 for the scalar reproducing kernel $\stein$, and its RKHS consists of functions $g(\cdot) = \langle \phi(\cdot) , f \rangle_{\H^d}$ for $f \in \H^d$ \citep{micchelli2005learning}.
\end{proof}

\subsection{The Stein Kernel Corresponding to the Diffusion Kernel Stein Discrepancy}
\label{Stein-kernel}

Note the Stein kernel satisfies 
\begin{talign}
\stein =\frac 1 {p(y)p(x)}  \nabla_y \cdot \nabla_x \cdot  \left( p(x)m(x)K m(y)^{\top}p(y)\right)
\end{talign}
since
\begin{talign}
 k^0&=\mathcal{S}^2_p\mathcal{S}^1_p K(x, y) = \frac 1 {p(y)p(x)} \nabla_y \cdot \left(p(y)m(y) \nabla_x \cdot \left(p(x)m(x)K \right) \right)\\
 &= \frac 1 {p(y)p(x)} \nabla_y \cdot \left(p(y)m(y) \partial_{x^i}\left(p(x)m(x)_{ir}K_{rs} \right)e_s \right) \\
 &=  \frac 1 {p(y)p(x)} \nabla_y \cdot \left(p(y)m(y)_{ls} \partial_{x^i}\left(p(x)m(x)_{ir}K_{rs} \right)e_l \right) \\
 &= \frac 1 {p(y)p(x)} \partial_{y^l}  \left(p(y)m(y)_{ls} \partial_{x^i}\left(p(x)m(x)_{ir}K_{rs} \right)\right)
 \\
 &= \frac 1 {p(y)p(x)} \partial_{y^l} \partial_{x^i} \left( p(x)m(x)_{ir}K_{rs} m(y)_{sl}^{\top}p(y)\right) \\
 &= \frac 1 {p(y)p(x)}  \nabla_y \cdot \nabla_x \cdot  \left( p(x)m(x)K m(y)^{\top}p(y)\right).
\end{talign}

Note it is also possible to view $m(x)K m(y)^{\top}$ as a new matrix kernel.
That is the matrix field $m$ defines a new kernel $K_m: (x,y) \mapsto m(x)K(x,y)m^{\top}(y)$, since 
$K_m(y,x)^{\top}= m(x)K(y,x)m(y)^{\top}=K_m(x,y)$
and for any $v_j \in \R^d,x_i \in \X$,
\begin{talign}
v_j \cdot K_m(x_j,x_l) v_l = v_j \cdot m(x_j)K(x_j,x_l)m(x_l)^{\top}v_l
= \left(m(x_j)^{\top}v_j \right)\cdot K(x_j,x_l) \left(m(x_l)^{\top}v_l \right)  \geq 0
\end{talign} 
We can expand the Stein kernel using the following expressions:
\begin{talign}
 &\nabla_y \cdot \left(p(y)m(y) \nabla_x \cdot \left(p(x)m(x)K \right) \right)\\
 &=\nabla_y \cdot \left(p(y)m(y) \left(Km(x)^{\top}\nabla_x p +p(x)\nabla_x\cdot(m(x)K) \right) \right).
 \end{talign}
 \begin{talign}
 &\nabla_y \cdot \left(p(y)m(y) Km(x)^{\top}\nabla_x p \right)\\
 &=m^{\top}(x) \nabla_x p \cdot Km(y)^{\top} \nabla_y p+p(y)\nabla_y\cdot\left( m(y) Km(x)^{\top}\nabla_x p\right)\\
 &=m^{\top}(x) \nabla_x p \cdot Km(y)^{\top} \nabla_y p+p(y)\nabla_y\cdot\left( m(y) K\right) \cdot m(x)^{\top}\nabla_x p,
 \end{talign}
\begin{talign}
 &\nabla_y \cdot \left( p(y)m(y)  p(x)\nabla_x\cdot\left(m(x)K\right) \right)\\
 &=p(x)\left( \nabla_y \cdot \left( p(y)m(y) \right) \cdot \nabla_x\cdot \left(m(x)K\right)+p(y) \tr\left[m(y) \nabla_y \nabla_x\cdot \left(m(x)K\right) \right] \right)\\
 &=p(x)p(y)\tr\left[m(y) \nabla_y \nabla_x\cdot(m(x)K) \right] \\
 &\; +p(x)\nabla_x\cdot \left(m(x)K\right)\cdot \left(m(y)^{\top} \nabla_y p+p(y) \nabla_y \cdot m \right).
\end{talign}
Hence
\begin{talign}
k^0 
&= m^{\top}(x) \nabla_x \log p \cdot Km(y)^{\top} \nabla_y \log p & \\
&+ \nabla_y\cdot\left( m(y) K\right) \cdot m(x)^{\top}\nabla_x \log p+
\nabla_x\cdot \left(m(x)K\right)\cdot m(y)^{\top} \nabla_y \log p\\
&+\nabla_x\cdot \left(m(x)K\right)\cdot  \nabla_y \cdot m +\tr\left[m(y) \nabla_y \nabla_x\cdot(m(x)K) \right] \\
&= 
\metric{s_p(x)}{ Ks_p(y)} + \metric{\nabla_y\cdot\left( m(y) K\right)}{ s_p(x)}+
\metric{\nabla_x\cdot \left(m(x)K\right)}{ s_p(y)}\\
&+\metric{\nabla_x\cdot \left(m(x)K\right)}{  \nabla_y \cdot m} +\tr\left[m(y) \nabla_y \nabla_x\cdot(m(x)K) \right] 
\end{talign}

\subsection{Special Cases of Diffusion Kernel Stein Discrepancy}\label{app:special-cases-dksd}
Consider
\begin{talign}
\stein =\frac{1}{p(y)p(x)}  \nabla_y \cdot \nabla_x \cdot  \left( p(x)m(x)K(x,y) m(y)^{\top}p(y)\right)
\end{talign}
and decompose $m(x)K(x,y) m(y)^{\top} \defn gA$ where $g$ is scalar and $A$ is matrix-valued.
Then we
\begin{talign}
\stein 
& = g \metric{\nabla_y \log p}{ A \nabla_x \log p} + \metric{\nabla_y \log p}{ A \nabla_x g}+\metric{\nabla_y g}{ A \nabla_x \log p} &\\
& +  \Tr[A \nabla_x \nabla_y g]+ g \nabla_y \cdot \nabla_x \cdot A+ \metric{\nabla_x \cdot A}{\nabla_y g} + \metric{\nabla_y \cdot A^{\top}}{\nabla_x g} \\
& +g \metric{\nabla_y \cdot A^{\top}}{\nabla_x \log p}  +g \metric{\nabla_x \cdot A}{\nabla_y \log p}.
\end{talign}
For the case, $K =\text{diag}(k^1,\ldots,k^d)$, setting $\mathcal T^x_i \defn   \frac{1}{p(x)} \partial_{x^i}\left(p(x) \cdot \right)$ then
\begin{talign}
\mathcal S^2_p \mathcal S^1_p[\text{diag}(k^1,\ldots,k^d)] = \mathcal T^y_l \left( m_{li}(y) \mathcal T^x_c \left(k^i(x,y) m^{\top}_{ic}(x) \right)\right)= \mathcal T^y_l \mathcal T^x_c \left( m_{li}(y) k^i(x,y) m_{ci}(x) \right).
\end{talign}
If $K=Ik$ in components
\begin{talign}
 \mathcal S^2_p \mathcal S^1_p[Ik]
 &= (s_p(x))_i k(x,y)(s_p(y))_i 
 +\partial_{y^i}(m_{ir}k)(s_p(x))_r
 +\partial_{x^i}(m(x)_{ir}k) (s_p(y))_r \\
 &+ \partial_{x^i}(m(x)_{ir}k) \partial_{y^l}(m_{lr})+ m(y)_{ir}\partial_{y^i} \partial_{x^s}(m(x)_{sr}k)
\end{talign}
When $p=p_\theta$ we are often interested in the gradient $\nabla_\theta k^0_\theta$. Note
 $\nabla_y\cdot\left( m(y) K\right)= k\nabla_y \cdot m +\nabla_y k \cdot m(y)$, 
so \footnote{More generally 
 $\nabla_y\cdot\left( m(y) K\right)= (\nabla_y \cdot m) \cdot K +\Tr[\nabla_yK \otimes m(y)]$ where $\Tr[\nabla_yK \otimes m]_r=\partial_{y^i}K_{jr}m_{ij}$
 and if $K=Bk$
 \begin{talign}
\partial_{\theta^i} \left[ (\nabla_y \cdot m) \cdot K s_p(x) \right ] &= k B_{sr} \partial_{\theta^i} \left((\nabla_y \cdot m)_s (s_p(x))_r\right)= k \Tr[B \partial_{\theta^i} (s_p(x) \otimes \nabla_y \cdot m )] \\
\partial_{\theta^i}\left[\nabla_y k ^{\top}m(y)Bs_p(x)\right] &=
 \partial_{y^s}k B_{jr}\partial_{\theta^i} [m_{sj}(y) (s_p(x))_r]
\end{talign}
}
\begin{talign}
\partial_{\theta^i} \left[ k\metric{\nabla_y \cdot m}{ s_p(x)} \right ]
 &= k  \partial_{\theta^i} \metric{ \nabla_y \cdot m }{s_p(x)} 
\\
\partial_{\theta^i}\left[\metric{\nabla_y k \cdot m(y)}{s_p(x)}\right]
 &=
\metric{\nabla_y k}{\partial_{\theta^i}[m(y)s_p(x)]} \\
\tr\left[m(y) \nabla_y \nabla_x\cdot(m(x)K) \right]
&= \nabla_y k^{\top}m(y) \nabla_x \cdot m+\Tr[m(y)m(x)^{\top}\nabla_y \nabla_x k]
\end{talign}
and the terms in  $\partial_{\theta^i} \stein$ reduce to
\begin{talign}
\partial_{\theta^i}\metric{s_p(x)}{ Ks_p(y)}&= k\partial_{\theta^i}\metric{s_p(x)}{ s_p(y)} 
\\
\partial_{\theta^i}\metric{\nabla_y\cdot \left(m(y)K\right)}{ s_p(x)}& =
 k  \partial_{\theta^i} \metric{ \nabla_y \cdot m }{s_p(x)} + \metric{\nabla_y k}{\partial_{\theta^i}[m(y)s_p(x)]}
\\
\partial_{\theta^i}\metric{\nabla_x\cdot \left(m(x)K\right)}{ s_p(y)}& =
 k  \partial_{\theta^i} \metric{ \nabla_x \cdot m }{s_p(y)} + \metric{\nabla_x k}{\partial_{\theta^i}[m(x)s_p(y)]}
\\
\partial_{\theta^i}\metric{\nabla_x\cdot \left(m(x)K\right)}{  \nabla_y \cdot m} &= 
k \partial_{\theta^i}\metric{\nabla_x \cdot m}{\nabla_y \cdot m} +\partial_{\theta^i}\metric{\nabla_x k \cdot m(x)}{\nabla_y \cdot m}.
\end{talign}
When $K=kI$ and  we further have a diagonal matrix $m= \text{diag}(f_i)$, $m(y)m(x)^{\top}= \text{diag}(f_i(y) f_i(x))$. 
If $u \odot v$ denotes the vector given by the pointwise product of vectors, i.e., $(u \odot v)_i = u_iv_i$, and $f$ is the vector, then $m(x) \nabla_x \log p =f(x) \odot \nabla_x \log p$
and 
$\left(\nabla_y \cdot m \right)_i= \partial_{y^i} f_i$, 
$\left(\nabla_x \cdot (mk) \right)_i= \partial_{x^i} (f_ik)$, 
\begin{talign}
 s_p(x) \cdot Ks_p(y) &= k(x,y)f_i(x) \partial_{x^i} \log p f_i(y) \partial_{y^i} \log p \\
\nabla_y\cdot\left( m(y) K\right) \cdot s_p(x)&=  \partial_{y^i} (f_i(y)k) f_i(x) \partial_{x^i} \log p \\
\nabla_x\cdot \left(m(x)K\right)\cdot  \nabla_y \cdot m &= \partial_{x^i} (f_i(x)k) \partial_{y^i} (f_i(y))\\
\Tr \left [m(y) \nabla_y  \nabla_x \cdot(m k) \right] &= f_i(y) \partial_{x^i} \left( f_i(x) \partial_{y^i} k \right)
\end{talign}
and if $m \mapsto m I$ (is scalar), (this is just KSD with $k(x,y) \mapsto m(x)k(x,y)m(y)$):
\begin{talign}
k^0 &= m(x) m(y) k(x,y) \nabla_x \log p \cdot \nabla_y \log p \\
& + m(x)\nabla_y\left( m(y) k\right) \cdot \nabla_x \log p+
m(y)\nabla_x \left( m(x) k \right)\cdot \nabla_y \log p\\
&+\nabla_x \left( m(x) k \right)\cdot  \nabla_y  m +m(y)  \nabla_x \cdot (m(x) \nabla_y k),
\end{talign}
When $m=I$, we recover the usual definition of kernel-Stein discrepancy (KSD):
\begin{talign}
\ksd\left(\Q\|\P\right) ^2
=
\int_{\X} \int_{\X}  \frac{1}{p(y)p(x)}  \nabla_y \cdot \nabla_x   \left( p(x)k(x,y)p(y)\right)\mathrm{d}\Q(x) \mathrm{d}\Q(y).
\end{talign}

\subsection{Diffusion Kernel Stein Discrepancies as Statistical Divergences} \label{Derivations-DKSD}

In this section, we prove that DKSD is a statistical divergence and provide sufficent conditions on the matrix-valued kernel.
\subsubsection{Proof of {\cref{DKSD-divergence}}: DKSD as statistical divergence}
By Stoke's theorem $\int_\X \St_q[v] \dd \Q= \int_\X \nabla \cdot(qmv) \dd x =0$, thus $\int_\X \St_p[v] \dd \Q =\int_\X (\St_p[v] -\St_q[v]) \dd \Q 
= \int_\X (s_p-s_q)\cdot v \dd \Q$, 
and by assumption $\int_\X \St_q[K] \dd \Q= \int_\X \nabla \cdot(qmK) \dd x=0$.
Moreover, with $s_p = m^{\top} \nabla \log p$, and $\delta_{p,q} \defn s_p -s_q$. Hence
\begin{talign}
\dksd_{K,m}(\Q,\P)^2 &= \int_\X \int_\X \mathcal{S}^2_p\left[\mathcal{S}^1_p K(x, y)\right]\dd\mathbb{Q}(y)\dd\mathbb{Q}(x) \\
& =  \int_\X \int_\X (s_p(y)-s_p(y)) \cdot\left[\mathcal{S}^1_p K(x, y)\right]\dd\mathbb{Q}(y)\dd\mathbb{Q}(x)\\
 &= \int_\X  (s_p(y)-s_p(y)) \dd\mathbb{Q}(y) \cdot\int_\X\left[\mathcal{S}^1_p K(x, y)\right]\dd\mathbb{Q}(x)\\
 &=\int_\X  (s_p(y)-s_p(y)) \dd\mathbb{Q}(y) \cdot\int_\X\left[\mathcal{S}^1_p K(x, y)-\St_q^1K(x,y)\right]\dd\mathbb{Q}(x) \\
 &=
 \int_\X  (s_p(y)-s_p(y)) \dd\mathbb{Q}(y)\cdot \int_\X \left[(s_p(x)-s_p(x))\cdot K(x,y)\right]\dd\mathbb{Q}(x) \\
 &= \int_\X \int_\X q(x)\delta_{p,q}(x)^{\top}K(x,y) \delta_{p,q}(y)q(y)  \dd x \dd y \\
 &=\int_\X  \int_\X \dd \mu^{\top}(x)K(x,y) \dd \mu(y).
\end{talign}
where $\mu(\dd x) \defn q(x)\delta_{p,q}(x) \dd x$, which is a finite measure by assumption.
If $\St(q,p)=0$, then since $K$ is IPD we have $q \delta_{p,q} \equiv 0$, 
and since $q>0$ and $m$ is invertible we must have $\nabla \log p = \nabla \log q$ and thus $q=p$.

\subsubsection{Proof of {\cref{IPD-matrix-kernels}}: IPD matrix kernels}
Let $\mu$ be a finite signed vector measure.
$(i)$ If each $k^i$ is IPD, then $\int \dd\mu^{\top} K \dd\mu = \int k^i(x,y) \dd \mu_i(x) \dd \mu_i(y) \geq 0$ with equality iff $\mu_i \equiv 0$ for all $i$. Conversely suppose $\int k^i(x,y) \dd \mu_i(x) \dd \mu_i(y) \geq 0$ with equality iff $\mu_i \equiv 0$ for all $i$ .
Suppose $k^j$ is not IPD for some $j$, then there exists a finite non-zero signed measure $\nu$ s.t., $\int k^j\dd \nu \otimes \dd \nu \leq 0$, so if we define the vector measure $\mu_i \defn \delta_{ij} \nu$, which is non-zero and finite, then $\int k^i(x,y) \dd \mu_i(x) d\mu_i(y) \leq 0$ which contradicts the assumption.
For $(ii)$, we first diagonalise $B=R^{\top}DR$ where $R$ is orthogonal and $D$ diagonal with positive entries $\lambda_i>0$.
Then \begin{talign}
\int \dd \mu^{\top} K \dd \mu = \int k\dd \mu^{\top}R^{\top}D R\dd \mu = \int k \left( R \dd \mu\right)^{\top}D \left( R \dd \mu\right) = \int k(x,y) \lambda_i \dd \nu_i(x) \dd \nu_i(y),
\end{talign}
where $\nu \defn R\mu$ is finite and non-zero, since $\mu$ is non-zero and $R$ is invertible, thus maps non-zero vectors to non-zero vectors.
Clearly if $k$ is IPD then $\int \dd \mu^{\top} K \dd \mu\geq 0$ with equality iff $\nu_i\equiv 0$ for all $i$. 
Suppose $K$ is IPD but $k$ is not, then there exists finite non-zero signed measure $\nu$ for which $\int k \dd \nu \otimes \dd \nu \leq 0$,
 but then setting $\mu \defn R^{\top}\xi$, with $\xi_i  \defn \delta_{ij} \nu$ which is finite and non-zero, implies 
$\int \dd \mu^{\top} K \dd \mu = \int k \dd \xi^{\top}D \dd \xi= \lambda_j \int k \dd \nu \otimes \dd \nu \leq 0$.

\subsection{Diffusion Score Matching} \label{derivations-DSM}

Another example of SD is the diffusion score matching (DSM) discrepancy, as introduced below:
\subsubsection{Proof of {\cref{prop:SM_is_Stein}}: Diffusion Score Matching}
Note that the Stein operator satisfies
\begin{talign}
\mathcal{S}_{p}[g]
& =
 \frac{  \nabla \cdot \left( p mg \right)}{p} = \frac{ \metric{\nabla p }{ mg} + p \nabla \cdot(mg)}{p}= 
\metric{\nabla \log p }{ mg} + \nabla \cdot (m g) =
\metric{m^{\top}\nabla \log p }{ g} + \nabla \cdot (m g).
\end{talign}
Since $\int_\X \mathcal S_q[g] \dd \Q=0$, we have 
\begin{talign}
D(\Q\| \P) 
& = 
\sup_{g \in \mathcal{G}} \left| \int_{\X} \mathcal{S}_{p}[g](x) \Q(\mathrm{d}x)\right|^2
 = 
\sup_{g \in \mathcal{G}} \left| \int_{\X} (\mathcal{S}_p[g](x) -  \mathcal{S}_{q}[g](x)) \Q(\mathrm{d}x)\right|^2 \\
& =  
\sup_{g \in \mathcal{G}} \left| \int_{\X}\left( \left(\nabla \log p-\nabla \log q \right) \cdot( mg)\right) \dd \Q\right|^2, \\
&= \sup_{g \in \mathcal{G}} \left| \left\langle
m^{\top}\left(\nabla \log p-\nabla \log q \right), g \right\rangle_{L^2(\Q)} 
\right|^2 \\
&= \left\|m^{\top}\left(\nabla \log p-\nabla \log q \right) \right\|_{L^2(\Q)}^2\\
&= \int_{\X} \left\| m^{\top} \left(\nabla \log p - \nabla \log q\right) \right\|_{2}^2 \dd \Q,
\end{talign}
where we have used the fact that $\mathcal{G}$ is dense in the unit ball of $L^2(\Q)$ (since smooth functions with compact support are dense in $L^2(\Q)$), 
and that the supremum over a dense subset 
of the continuous functional $F(\cdot) \defn\left\langle
m^{\top}\left(\nabla \log p-\nabla \log q \right), \cdot \right\rangle_{L^2(\Q)}$ is equal to the supremum over the closure,
$\text{sup}_{\mathcal G}F = \text{sup}_{\overline{ \mathcal G}}F$. 
Suppose $D(\Q\|\P)=0$. Then since $q>0$ we must have 
$\left\| m^{\top} \left(\nabla \log p - \nabla \log q\right) \right\|_{2}^2=0$, i.e., $ m^{\top} \left(\nabla \log p - \nabla \log q\right) =0$, i.e.,
$ \nabla (\log p- \log q)=0$. Thus  $\log (p/q) =c$, so $p = qe^c$ and integrating implies $c=0$, so $D(\Q\|\P)=0$ iff $\Q=\P$ a.e..

To obtain the estimator we will use the divergence theorem, which holds for example if $X, \nabla \cdot X \in L^1(\R^d)$ for $X=q  m m^{\top} \nabla \log p $ (see theorem 2.36, 2.28~\cite{pigola2014global} or theorem 2.38 for weaker conditions). 
Note
\begin{talign}
\left\| m^{\top} \left(\nabla \log p - \nabla \log q\right) \right\|_{2}^2=
\| m^{\top} \nabla \log p\|_2^2+
\| m^{\top} \nabla \log q \|^2_2-
2m^{\top} \nabla \log p \cdot m^{\top} \nabla \log q 
\end{talign}
thus we have 
\begin{talign}
\int_\X \metric{m^{\top} \nabla \log p }{ m^{\top} \nabla \log q}  \dd \Q & = \int_\X \metric{\nabla \log q }{ m m^{\top} \nabla \log p} \dd Q \\
&= \int_\X \metric{\nabla  q }{ m m^{\top} \nabla \log p} \dd x \\
&= \int_\X \left(\nabla  \cdot \left( q  m m^{\top} \nabla \log p \right)-q  \nabla \cdot\left( m m^{\top} \nabla \log p\right) \right)\dd x\\
&= -  \int_\X q  \nabla \cdot\left( m m^{\top} \nabla \log p\right) \dd x \\
&=  -  \int_\X   \nabla \cdot\left( m m^{\top} \nabla \log p\right) \dd \Q.
\end{talign} 

\subsubsection{Diffusion Score Matching Estimators}
As for the standard SM estimator, 
the DSM is only defined for distributions with sufficiently smooth densities.
However the $\theta$-dependent part of $\dsm_m(\Q,\P_{\theta})$ \footnote{ Here we use 
$\nabla \cdot\left( m m^{\top} \nabla \log p\right) = \metric{\nabla \cdot( m m^{\top} )}{ \nabla \log p}+ \tr \left[mm^{\top}  \nabla^2\log p \right]$}
\begin{talign}
&\int_{\X} \left( \left\|  m^{\top}\nabla_x\log p_{\theta} \right\|_{2}^2+2\nabla \cdot\left( m m^{\top} \nabla \log p_{\theta}\right)\right) \dd \Q \\
&=\int_{\X} \left( \left\|  m^{\top}\nabla_x\log p_{\theta} \right\|_{2}^2+2\left( \metric{\nabla \cdot( m m^{\top} ) }{ \nabla \log p} + \tr \left[mm^{\top}  \nabla^2\log p \right]\right)\right) \dd \Q,
\end{talign}
 does not depend on the density of $\Q$. 
An unbiased estimator for this quantity follows by replacing $\Q$ with the empirical random measure $\Q_n \defn \frac 1 n \sum_i \delta_{X_i}$ where $X_i \sim \Q$ are independent.
Hence we consider the estimator
\begin{talign}
 \hat{\theta}^{\dsm}_n & \defn  \text{argmin}
_{\theta \in \Theta} \Q_n \left( \left\|  m^{\top}\nabla_x\log p_{\theta} \right\|_{2}^2 +2\left( \metric{\nabla \cdot( m m^{\top} )}{ \nabla \log p_{\theta}} + \tr \left[mm^{\top}  \nabla^2\log p_{\theta} \right]\right)\right). 
\end{talign}
In components, this corresponds to:
\begin{talign}
\hat{\theta}^{\dsm}_n 
&= \text{argmin}_{\theta \in \Theta} \int_{\X} \dd \Q(x) \| m(x)^{\top}\nabla_x \log p(x|\theta) \|_{2}^2 + 2 \sum_{j,k,l=1}^d \partial_{x^j} \partial_{x^k} \log p(x|\theta) m_{kl}(x) m_{jl}(x)\\ 
& \qquad + 2 \sum_{j,k,l=1}^d \partial_{x^k} \log p(x|\theta) \left(\partial_{x^j} m_{kl}(x) m_{jl}(x) + m_{kl}(x) \partial_{x^j} m_{jl}(x)\right)
\end{talign}

\subsubsection{Proof of {\cref{DSM-DKSD-limit}}: DSM as a limit of DKSD}
We now consider the the limit in which DKSD converges to DSM:

\begin{theorem}[\textbf{DSM as a limit of DKSD}]\label{DSM-DKSD-limit}
Let $ \Q $ be a distribution on $\mathbb{R}^{d}$ with $q>0$ and suppose $s_{p}-s_{q} \in C(\mathbb{R}^{d})\cap L^{2}(\Q)$. Let $\Phi_{\gamma}(s) \defn \gamma^{-d}\Phi(s/\gamma)$, $\gamma >0$, $\Phi\in L^{1}(\mathbb{R}^{d})$, $\Phi > 0$ and $\int_{\mathbb{R}^d}\Phi(s) \dd s = 1$.
Consider the reproducing kernel $k_{\gamma}^{q}(x,y) \;=\;k_{\gamma}(x,y)/\sqrt{q(x)q(y)}=\Phi_{\gamma}(x-y)/\sqrt{q(x)q(y)}$, 
and set $K_{\gamma}^q \defn B k^q_{\gamma}$.
Then, $\dksd_{K^q_{\gamma},m}(\Q \|\P)^{2} \rightarrow \dsm_m(\Q\| \P)$, as $\gamma \rightarrow 0$.
\end{theorem}

 We use the following lemma as a stepping stone.
\begin{Lemma} Suppose $\Phi\in L^{1}(\mathbb{R}^{d}),$
 $\Phi > 0$ and $\int\Phi(s)\,\dd s  \;=\; 1$. 
 Let $f,\,g\in C(\mathbb{R}^{d})\cap L^{2}(\mathbb{R}^{d}),$
then defining $K_{\gamma} \defn B \Phi_{\gamma}$ where $\Phi_{\gamma}(s) \defn \gamma^{-d}\Phi(s/\gamma)$ and $\gamma \;>\;0$,
we have
\begin{talign}
\int\int f(x)^{\top}K_{\gamma}(x,y)g(y)\dd x\,\dd y \; \rightarrow \; \int f(x)^{\top}Bg(x)\,\dd x,\quad\mbox{as}\qquad\gamma \; \rightarrow \; 0.
\end{talign}
\end{Lemma}

\begin{proof}
We rewrite 
\begin{talign}
\int_\X \int_\X f(x)^{\top}B\Phi_{\gamma}(x-y)g(y)\,\dd x\,\dd y 
& =  \int_\X \int_\X  f(x)^{\top}Bg(x-s)\dd x\,\Phi_{\gamma}(s)\,\dd s = \int_\X H(s)\Phi_{\gamma}(s)\,\dd s,
\end{talign}
where $H:\X \to \R$ is defined by
\begin{talign}
H(s) \defn \int_\X f(x)^{\top}Bg(x-s)\,\dd x  
& = \int_\X \langle f(x), Bg(x-s)\rangle_{\R^d} \dd x \defn \int_\X \langle f(x), g(x-s)\rangle_B \dd x.
\end{talign}
 Since $f,g\in C(\mathbb{R}^{d})\cap L^{2}(\mathbb{R}^{d}),$
the function $H(s)$ is continuous, bounded, $|H(s)| \leq A\lVert f\rVert_{L^{2}(\mathbb{R}^{d})}\lVert g\rVert_{L^{2}(\mathbb{R}^{d})}$ for a constant $A>0$ depending only on $B$, and $H(0)  =\int f(x)^{\top}Bg(x)\,\dd x$.
Given $\delta \;>\;0$, we can split the integral as follows:
\begin{talign}
\int_{|s|<\delta}H(s)\Phi_{\gamma}(s)\,\dd s\;+\;\int_{|s|>\delta}H(s)\Phi_{\gamma}(s)\,\dd s \defn I_{1}+I_{2}.
\end{talign}

By continuity, given $\epsilon \in (0,1)$ there exists $\delta >0$ such
that $|H(s)-H(0)|<\epsilon$ for all $|s| <\delta$. 
Let $I_{<\delta} \defn \int_{|y|<\delta}\Phi_{\gamma}(y)\,\dd y >0$ since $\Phi>0$.
Consider
\begin{talign}
I_{1}-H(0) &\;=\;
\int_{|s|<\delta}\Phi_{\gamma}(s)H(s) \dd s -H(0) \;=\; \int_{|s|<\delta}\Phi_{\gamma}(s)\left(H(s)-\frac{H(0)}{I_{<\delta}} \right) \dd s \\
&\;=\; \int_{|s|<\delta}\frac{\Phi_{\gamma}(s)}{I_{<\delta}}\left(H(s)I_{<\delta} \;-\;H(0)\right)\,\dd s.
\end{talign}
Clearly $\int \Phi_{\gamma}(s) \dd s=\int \gamma^{-d}\Phi(s/\gamma) \dd s= \int \Phi(z) \dd z=1$, since $z \defn s/\gamma$ implies $\dd z = \gamma^{-d} \dd s$, 
so 
\begin{talign}
I_{<\delta}\;=\;1-I_{>\delta} \;=\; 1-\int_{|y|>\delta/\gamma}\Phi(y)\,\dd y.
\end{talign}
Then since $\Phi$ is integrable, there exists $\gamma_{0}(\delta)>0$ s.t. for $\gamma\;<\;\gamma_{0}(\delta)$ we have
 $\int_{|y|>\delta/\gamma}\Phi(y)\,\dd y<\epsilon$ 
and thus $0<1-\epsilon <I_{<\delta}<1$. Therefore, for $\gamma \;<\;\gamma_{0}(\delta):$
\begin{talign}
\left|I_{1}-H(0)\right| & = \left| \int_{|s|<\delta}\frac{\Phi_{\gamma}(s)}{I_{<\delta}}\left(H(s)I_{<\delta}-H(0)\right) \dd s \right|\\
 & \leq \int_{|s|<\delta}\frac{\Phi_{\gamma}(s)}{I_{<\delta}}\left|\left(\left(H(s)-H(0)\right)I_{<\delta}+H(0)\left(I_{<\delta}-1\right)\right)\right| \dd s
 \\
 & \leq \int_{|s|<\delta}\frac{\Phi_{\gamma}(s)}{I_{<\delta}}\left(\left|H(s)-H(0)\right|I_{<\delta}+\left|1-I_{<\delta}\right|H(0)\right)\dd s\\
 & \leq \int_{|s|<\delta}\frac{\Phi_{\gamma}(s)}{I_{<\delta}}\left(\epsilon I_{<\delta}+\epsilon H(0)\right) \dd s\\
 & \leq \epsilon \int_{|z|<\delta/\gamma}\Phi(z)\dd z +H(0)\epsilon \leq \left(1+H(0)\right) \epsilon.
\end{talign}
\\
For the second term, since $H$ is bounded we have
\begin{talign}
I_{2} & = \int_{|s|>\delta}H(s)\Phi_{\gamma}(s)\dd s  
= \int_{|s|>\delta/\gamma}H(\gamma s)\Phi(s)\dd s 
\leq  \| H \|_{\infty}\int_{|s|>\delta/\gamma}\Phi(s)\dd s,
\end{talign}
so that, $|I_{2}| \;\leq \; \| H \|_{\infty}\epsilon$, for $\gamma\;<\;\gamma_{0}(\delta)$.
It follows that
\begin{talign}
\left|\int\int f(x)^{\top}K_{\gamma}(x,y)g(y)\dd x\,\dd y-\int f(x)^{\top}Bg(x)\,\dd x\right| 
 &\; = \; \left| \int H(s) \Phi_{\gamma}(s) \dd s -H(0) \right|  \\
 &\; = \; \left|I_1+I_2 -H(0) \right |\\
&\;\leq \;|I_{1}-H(0)|+|I_{2}| \; \rightarrow \; 0,
\end{talign}
as $\gamma \rightarrow  0$ as required.
\end{proof}
We note that $f\in L^{2}(\Q)$ if and only if $f\sqrt{q}\in L^{2}(\mathbb{R}^{d}).$
Therefore applying the previous result, we have that
\begin{talign}
\int_{\X} \int_{\X} f(x)^{\top}K_{\gamma}^{q}(x,y)g(y)\,\dd \Q(x)\,\dd \Q(y)
 & = \int_{\X} \int_{\X} \left(\sqrt{q(x)}f(x)\right)^{\top}K_{\gamma}(x,y) \left(g(y)\sqrt{q(y)}\right)\dd x\dd y\\
 & \rightarrow  \int_{\X} f(x)^{\top}Bg(x)\dd \Q(x),\quad\mbox{as}\quad\gamma  \rightarrow  0.
\end{talign}
Note that if $k$ is a (scalar) kernel function, then $(x,y) \mapsto r(x) k(x,y)r(y)$ is a kernel for any function $r:\X \to \R$, and thus $k_{\gamma}^q$ defines a sequence of kernels parametrised by a scale parameter $\gamma >0$.
It follows that the sequence of DKSD paramaterised by $K^q_{\gamma}$
\balignt\dksd_{K^q_{\gamma},m}(\Q \|\P)^{2}=\int_{\X} \int_{\X} q(x)\delta_{p,q}(x)^{\top}K^q_{\gamma}(x,y) \delta_{p,q}(y)q(y)  \dd x \dd y
\ealignt
converges to DSM with inner product $\langle \cdot,\cdot \rangle_B \defn \langle \cdot, B \cdot \rangle_2$ on $\R^d$.
\balignt \dsm_m(\Q\| \P) = \int_{\X} \delta_{q,p}(x)^{\top}B \delta_{q,p}(x) \dd \Q = \int_{\X} \| m^{\top}\left( \nabla \log p -\nabla \log q \right) \|^2_B \dd \Q\ealignt


\section{Information Semi-Metrics of Minimum Stein Discrepancy Estimators}\label{appendix:information-semi-metrics}

In this section, we derive expressions for the metric tensor of DKSD and DSM.
Let $\mathcal P_{\Theta}$ be a parametric family of probability measures on $\X$. 
Given a map $D:\mathcal P_{\Theta} \times \mathcal P_{\Theta} \to \R$, for which $D(\P_1 \| \P_2)=0$ iff $\P_1 = \P_2$, its associated information semi-metric is defined  as the map $\theta \mapsto g(\theta)$, where $g(\theta)$ is the symmetric bilinear form $g(\theta)_{ij} = - \half \frac{\partial^2}{\partial \alpha^i \partial \theta^j}D( \P_{\alpha}\|\P_{\theta}) |_{\alpha=\theta}$. 
When $g$ is positive definite, we can use it to perform (Riemannian) gradient descent on $\mathcal P_{\Theta} \cong \Theta$.

\subsection{Proof of {\cref{information-dksd}}: Information Semi-Metric of Diffusion Kernel Stein Discrepancy}

From \cref{DKSD-divergence} we have
\begin{talign}
\dksd_{K,m}( \P_{\alpha},\P_{\theta})^2 = \int_\X \int_\X p_{\alpha}(x)\delta_{p_{\theta},p_{\alpha}}(x)^{\top}K(x,y) \delta_{p_{\theta},p_{\alpha}}(y)p_{\alpha}(y)  \dd x \dd y  
\end{talign}
where $ \delta_{p_{\theta},p_{\alpha}} = m^{\top}_{\theta}\left(\nabla \log p_{\theta}-\nabla \log p_{\alpha} \right)$. Thus 
\begin{talign}
\partial_{\alpha^i} \partial_{\theta^j} \dksd_{K,m}( \P_{\alpha},\P_{\theta})^2 
&=
  \partial_{\alpha^i} \partial_{\theta^j}  \int_\X \int_\X p_{\alpha}(x)\delta_{p_{\theta},p_{\alpha}}(x)^{\top}K(x,y) \delta_{p_{\theta},p_{\alpha}}(y)p_{\alpha}(y)  \dd x \dd y  \\
  &=  \partial_{\alpha^i}   \int_\X \int_\X p_{\alpha}(x)\partial_{\theta^j}\delta_{p_{\theta},p_{\alpha}}(x)^{\top}K(x,y) \delta_{p_{\theta},p_{\alpha}}(y)p_{\alpha}(y)  \dd x \dd y \\
  &+
  \partial_{\alpha^i}   \int_\X \int_\X p_{\alpha}(x)\delta_{p_{\theta},p_{\alpha}}(x)^{\top}K(x,y) \partial_{\theta^j}\delta_{p_{\theta},p_{\alpha}}(y)p_{\alpha}(y)  \dd x \dd y ,
\end{talign}
and using $\delta_{p_{\theta},p_{\theta}}=0$, we get:
\begin{talign}
&\partial_{\alpha^i}   \int_\X \int_\X p_{\alpha}(x)\partial_{\theta^j}\delta_{p_{\theta},p_{\alpha}}(x)^{\top}K(x,y) \delta_{p_{\theta},p_{\alpha}}(y)p_{\alpha}(y)  \dd x \dd y \big|_{\alpha=\theta}\\
&=
\partial_{\alpha^i}   \int_\X \int_\X p_{\alpha}(x)\left( \partial_{\theta^j}m^{\top}_{\theta}( \nabla \log p_{\theta}-\nabla \log p_{\alpha})+ m^{\top}_{\theta} \partial_{\theta^j}\nabla \log p_{\theta}  \right)^{\top} K(x,y) \delta_{p_{\theta},p_{\alpha}}(y)p_{\alpha}(y)  \dd x \dd y \big|_{\alpha=\theta} \\
&=
  \int_\X \int_\X p_{\alpha}(x)\left( m^{\top}_{\theta} \partial_{\theta^j}\nabla \log p_{\theta}  \right)^{\top} K(x,y) \partial_{\alpha^i} \delta_{p_{\theta},p_{\alpha}}(y)p_{\alpha}(y)  \dd x \dd y \big|_{\alpha=\theta} \\
  & =
 - \int_\X \int_\X p_{\alpha}(x)\left( m^{\top}_{\theta} \partial_{\theta^j}\nabla \log p_{\theta}  \right)^{\top} K(x,y)  \left( m^{\top}_{\theta} \partial_{\alpha^i}\nabla \log p_{\alpha}  \right)(y)p_{\alpha}(y)  \dd x \dd y \big|_{\alpha=\theta} \\
 & = 
  - \int_\X \int_\X \left( m^{\top}_{\theta} \partial_{\theta^j}\nabla \log p_{\theta}  \right)^{\top}(x) K(x,y)  \left( m^{\top}_{\theta} \partial_{\theta^i}\nabla \log p_{\theta}  \right)(y) \dd \P_{\theta}(x) \dd \P_{\theta}(y).
\end{talign}

Similarly, we also get:

\begin{talign}
 &\partial_{\alpha^i}   \int_\X \int_\X p_{\alpha}(x)\delta_{p_{\theta},p_{\alpha}}(x)^{\top}K(x,y) \partial_{\theta^j}\delta_{p_{\theta},p_{\alpha}}(y)p_{\alpha}(y)  \dd x \dd y  \big|_{\alpha=\theta} \\
 &=
 - \int_\X \int_\X \left( m^{\top}_{\theta} \partial_{\theta^i}\nabla \log p_{\theta}  \right)^{\top}(x) K(x,y)  \left( m^{\top}_{\theta} \partial_{\theta^j}\nabla \log p_{\theta}  \right)(y) \dd \P_{\theta}(x) \dd \P_{\theta}(y) \\
 &= - \int_\X \int_\X \left( m^{\top}_{\theta} \partial_{\theta^i}\nabla \log p_{\theta}  \right)^{\top}(y) K(y,x)  \left( m^{\top}_{\theta} \partial_{\theta^j}\nabla \log p_{\theta}  \right)(x) \dd \P_{\theta}(y) \dd \P_{\theta}(x) 
  \\
 &= - \int_\X \int_\X \left( m^{\top}_{\theta} \partial_{\theta^i}\nabla \log p_{\theta}  \right)^{\top}(y) K(x,y)^{\top}  \left( m^{\top}_{\theta} \partial_{\theta^j}\nabla \log p_{\theta}  \right)(x) \dd \P_{\theta}(y) \dd \P_{\theta}(x) \\
 &= - \int_\X \int_\X  \left( m^{\top}_{\theta} \partial_{\theta^j}\nabla \log p_{\theta}  \right)(x)^{\top} K(x,y)\left( m^{\top}_{\theta} \partial_{\theta^i}\nabla \log p_{\theta}  \right)(y)   \dd \P_{\theta}(y) \dd \P_{\theta}(x). 
\end{talign}
Hence, we conclude that
\begin{talign}
\half \partial_{\alpha^i} \partial_{\theta^j} \dksd_{K,m}(\P_{\alpha},\P_{\theta})^2 = - \int_\X \int_\X  \left( m^{\top}_{\theta} \partial_{\theta^j}\nabla \log p_{\theta}  \right)(x)^{\top} K(x,y)\left( m^{\top}_{\theta} \partial_{\theta^i}\nabla \log p_{\theta}  \right)(y)   \dd \P_{\theta}(y) \dd \P_{\theta}(x)
\end{talign}
The information tensor is positive semi-definite. Indeed writing $V_\theta(y)\defn m^{\top}_{\theta}(y) \nabla_y  \metric{v}{\nabla_{\theta} \log p_{\theta}}$:
\begin{talign}
\metric{v}{g(\theta)v}
& = v^ig_{ij}(\theta)v^j  \\
&=\int_{\X} \int_{\X}  \left( m^{\top}_{\theta}(x) \nabla_x \metric{v}{\nabla_{\theta} \log p_{\theta}}  \right)^{\top} K(x,y)\left( m^{\top}_{\theta}(y) \nabla_y  \metric{v}{\nabla_{\theta} \log p_{\theta}}  \right)  \dd \P_{\theta}(x) \dd \P_{\theta}(y) \\
&= 
\int_{\X} \int_{\X}  \metric{ m^{\top}_{\theta}(x) \nabla_x \metric{v}{\nabla_{\theta} \log p_{\theta}} }{ K(x,y) m^{\top}_{\theta}(y) \nabla_y  \metric{v}{\nabla_{\theta} \log p_{\theta}} } \dd \P_{\theta}(x) \dd \P_{\theta}(y)\\
&= 
\int_{\X} \int_{\X}  \metric{ V_\theta(x)}{ K(x,y)V_\theta(y) } \dd \P_{\theta}(x) \dd \P_{\theta}(y) \geq 0
\end{talign}
since $K$ is IPD.

\subsection{Proof of \cref{app:information-dsm}: Information Semi-Metric of Diffusion Score Matching}

\begin{proof}
The information metric is given by $g(\theta)_{ij} = - \half \frac{\partial^2}{\partial \alpha^i \partial \theta^j}\dsm( p_{\alpha}\|p_{\theta}) |_{\alpha=\theta}$.
Recall 
\begin{talign}
\dsm( p_{\alpha}\|p_{\theta})= \int_{\X} \left\| m^{\top} \left(\nabla \log p_{\theta} - \nabla \log p_{\alpha}\right) \right\|_{2}^2 p_{\alpha}\dd x.
\end{talign}
Moreover 
\begin{talign}
  \half\partial_{\alpha^i} \partial_{\theta^j}
   \dsm( p_{\alpha}\|p_{\theta})\big|_{\alpha=\theta}
   &=  \half\partial_{\alpha^i} \partial_{\theta^j} \int_{\X} \left\| m^{\top} \left(\nabla \log p_{\theta} - \nabla \log p_{\alpha}\right) \right\|_{2}^2 p_{\alpha}\dd x\big|_{\alpha=\theta}\\
   &=
   \partial_{\alpha^i}  \int_{\X} \left( m^{\top} \left(\nabla \log p_{\theta} - \nabla \log p_{\alpha}\right) \right)\cdot \left(m^{\top} \partial_{\theta^j} \nabla \log p_{\theta}\right) p_{\alpha}\dd x\big|_{\alpha=\theta}\\
   &=   \int_{\X} \left( m^{\top} \left(\nabla \log p_{\theta} - \nabla \log p_{\alpha}\right) \right)\cdot \left(m^{\top} \partial_{\theta^j} \nabla \log p_{\theta}\right)  \partial_{\alpha^i} p_{\alpha}\dd x\big|_{\alpha=\theta}\\
   &-  \int_{\X} \left( m^{\top} \partial_{\alpha^i}\nabla \log p_{\alpha} \right)\cdot \left(m^{\top} \partial_{\theta^j} \nabla \log p_{\theta}\right) p_{\alpha}\dd x \big|_{\alpha=\theta} \\
   &= -  \int_{\X} \left( m^{\top} \partial_{\theta^i}\nabla \log p_{\theta} \right)\cdot \left(m^{\top} \partial_{\theta^j} \nabla \log p_{\theta}\right) \dd \P_{\theta}.
   \end{talign}    
 Finally $g$ is semi-positive definite, 
 \begin{talign}
 \metric {v}{g(\theta)v}=v^i g_{ij}(\theta)v^j 
 &=
 \int_{\X} v^i  m^{\top}_{rs} \partial_{x^s} \partial_{\theta^i} \log p_{\theta} m^{\top}_{rl} \partial_{x^l} \partial_{\theta^j} \log p_{\theta}v^j \dd \P_{\theta} \\
 & =
 \int_{\X}   m^{\top}_{rs} \partial_{x^s} \metric{v}{\nabla_{\theta} \log p_{\theta}} m^{\top}_{rl} \partial_{x^l} \metric{v}{\nabla_{\theta} \log p_{\theta}} \dd \P_{\theta}  \\
 & = 
 \int_{\X}  \metric{ m^{\top} \nabla_{x} \metric{v}{\nabla_{\theta} \log p_{\theta}}}{ m^{\top} \nabla_x \metric{v}{\nabla_{\theta} \log p_{\theta}}} \dd \P_{\theta} \\
 & = 
  \int_{\X} \| m^{\top} \nabla_{x} \metric{v}{\nabla_{\theta} \log p_{\theta}} \|^2\dd \P_{\theta} 
  \geq 0
 \end{talign}  
\end{proof}


\section{Proofs of Consistency and Asymptotic Normality for minimum Stein Discrepancy Estimators}\label{appendix:asymptotic-derivations}

In this appendix, we prove several results concerning the consistency and asymptotic normality of DKSD and DSM estimators.

\subsection{Diffusion Kernel Stein Discrepancies}
Given the Stein kernel \eqref{stein-kernel}
we want to estimate
$ \theta_*^{\dksd}  \defn  \text{argmin}_{\theta \in \Theta} \dksd_{K,m}(\mathbb{Q},\P_{\theta})^2 =  \text{argmin}_{\theta \in \Theta} \int_\X \int_\X \stein_{\theta}(x, y)\mathbb{Q}(\dd x)\mathbb{Q}(\dd y)  $
using a sequence of estimators $ \hat \theta_n^{\dksd}  \in  \text{argmin}_{\theta \in \Theta} \hvf\dksd_{K,m}( \mathbb{Q},\P_{\theta})^2$ that minimise the $U$-statistic approximation \eqref{dksd-approx}. 
We will assume we are in the specified setting $\Q =\P_{\theta_*^{\dksd}}\in \mathcal P_{\Theta}$. 
In the misspecified setting 
it is necessary to further assume the existence of a unique minimiser.

\subsubsection{Strong Consistency }
We first prove a general strong consistency result based on an equicontinuity assumption:
\begin{Lemma}
Let $\X=\R^d$.
Suppose  $\{\theta \mapsto \stein_{\theta}(x,y)\}, \{ \theta \mapsto \Q_z \stein_\theta(x,z)\}$ are equicontinuous on any compact subset $C \subset \Theta$ for $x,y$ in a sequence of sets whose union has full $\Q$-measure, and 
$\|s_{p_{\theta}}(x) \| \leq f_1(x) $, $\|\nabla_x \cdot m_{\theta}(x) \| \leq f_2(x)$, $\| \nabla_x \cdot (m_\theta (x)K(x,y))\| \leq f_3(x,y)$, $|\tr\left[m(y) \nabla_y \nabla_x\cdot(m(x)K) \right]| \leq f_4(x,y)$ hold on $C$, where 
 $f_1(x) \sqrt{K(x,x)_{ii}} \in L^1(\Q)$, and $f_4,f_3f_2,f_1 f_3 \in L^1(\Q \otimes \Q)$.
Assume further that $\theta \mapsto \P_{\theta}$ is injective.  
Then we have a unique minimiser $\theta_*^{\dksd}$, and if either $\Theta$ is compact, or $\theta_*^{\dksd} \in \text{int}(\Theta)$ and $\Theta$ and $\theta \mapsto \hvf\dksd_{K,m}( \{X_i\}_{i=1}^n,\P_{\theta})^2$ are convex, then $\hat \theta_n^{\dksd}$ is strongly consistent.
\end{Lemma}
\begin{proof}

Note $\dksd_{K,m}( \mathbb{Q},\P_{\theta})^2=0$ iff $\P_{\theta}=\P_{\theta_*^{\dksd}}$ by \cref{DKSD-divergence},  which implies $\theta = \theta_*^{\dksd}$ since  $\theta \mapsto \P_{\theta}$ is injective. Thus we have a unique minimiser at $\theta_*^{\dksd}$.

Suppose first $\Theta$ is compact and take $C=\Theta$. 
Note
\begin{talign}
|\stein (x,y)| \leq 
& \left|\metric{s_p(x)}{  Ks_p(y)}\right| + |\metric{\nabla_y\cdot\left( m(y) K\right)}{ s_p(x)}|+
|\metric{\nabla_x\cdot \left(m(x)K\right)}{ s_p(y)}|\\
&+\left|\metric{\nabla_x\cdot \left(m(x)K\right)}{  \nabla_y \cdot m}| +|\tr\left[m(y) \nabla_y \nabla_x\cdot(m(x)K) \right] \right| \\
& \leq  \left|\metric{s_p(x)}{ Ks_p(y)}\right| 
 +f_3(y,x)f_1(x)  + f_3(x,y) f_1(y)+f_3(x,y) f_2(y) + f_4(x,y),
\end{talign}
From the reproducing property $f(x)= \metric{f}{K(\cdot,x)v}_{\H^d}$, for any $f \in \H^d$, $v \in \R^d$.
Using $K(y,x) = K(x,y)^{\top}$ we have 
$K(\cdot,x)_{,i}= K(x,\cdot)_{i,}$, where $K(\cdot,x)_{,i}$ and $K(x,\cdot)_{i,}$ denote the $i^{\text{th}}$ column and row respectively, 
which implies that $K(x,\cdot)_{i,}, K(\cdot,x)_{,i}\in \H^d$ and $f(x)_i =\metric{f}{K(\cdot,x)_{,i}}_{\H^d}$.
Choosing $f=K(\cdot,y)_{,j}$ implies 
\begin{talign}
K(x,y)_{ij}=
\metric{K(\cdot,y)_{,j}}{K(\cdot,x)_{,i}}_{\H^d} &\leq \| K(\cdot,y)_{,j} \|_{\H^d}
\| K(\cdot,x)_{,i} \|_{\H^d} \\
& = \sqrt{\metric{K(\cdot,y)_{,j}}{K(\cdot,y)_{,j}}_{\H^d}} \sqrt{\metric{K(\cdot,x)_{,i}}{K(\cdot,x)_{,i}}_{\H^d}} \\
& = \sqrt{K(y,y)_{jj}}\sqrt{K(x,x)_{ii}}.
\end{talign}

It follows that 
\begin{talign}
\metric{s_p(x)}{  Ks_p(y)} &= (s_p)_i(x)K(x,y)_{ij}(s_p)_j(y) \leq (s_p)_i(x) \sqrt{K(x,x)_{ii}} \sqrt{K(y,y)_{jj}}(s_p)_j(y) \\
&\leq  \|s_p(x)\|_\infty \sqrt{K(x,x)_{ii}} \sqrt{K(y,y)_{jj}}\|s_p(y)\|_\infty  \\
&\leq C f_1(x)\sqrt{K(x,x)_{ii}} \sqrt{K(y,y)_{jj}} f_1(y),
\end{talign}  
where the constant $C>0$ arises from the norm-equivalence of $\|s_p(y)\|$ and $\|s_p(y)\|_\infty$.
Hence $\stein$ is integrable.
Thus by theorem 1 \cite{yeo2001uniform}, 
\begin{talign}
\sup_{\theta} \left| \hvf \dksd_{K,m}( \{X_i\}_{i=1}^n,\P_{\theta})^2-\dksd_{K,m}( \mathbb{Q},\P_{\theta})^2\right| &\xrightarrow[]{a.s.} 0
\end{talign}
and $\theta \mapsto \dksd_{K,m}( \mathbb{Q},\P_{\theta})^2$ are continuous. 
By theorem 2.1 \cite{newey1994large} then $\hat\theta_n^{\dksd} \xrightarrow[]{a.s.} \theta_*^{\dksd}$.

On the other hand, if $\Theta$ is convex we follow a similar strategy to the proof of theorem 2.7 \cite{newey1994large}. 
Since $\theta_*^{\dksd} \in \text{int}(\Theta)$, we can find a $\epsilon >0$ for which $C = \overline B(\theta_*^{\dksd},2\epsilon) \subset \Theta$ is a closed ball containing $\theta_*^{\dksd}$
(which is compact since $\Theta \subset \R^m$). 
Using the compact case, we know any sequence of estimators $\tilde \theta^{\dksd}_n \in \argmin_{\theta \in C}\hvf \dksd_{K,m}( \{X_i\}_{i=1}^n,\P_{\theta})^2$ is strongly consistent for $\theta_*^{\dksd}$. 
In particular, there exists $N_0$ a.s. s.t. for $n>N_0$, $\| \tilde \theta^{\dksd}_n - \theta_*^{\dksd} \| < \epsilon$ . 
If $\theta \notin C$, there exists $\lambda \in [0,1)$ s.t. $\lambda \tilde \theta^{\dksd}_n +(1-\lambda) \theta$ lies on the boundary of the closed ball $C$. 
Using convexity and the fact $\tilde \theta^{\dksd}_n$ is a minimiser over $C$,  
\begin{talign}
& \hvf \dksd_{K,m}( \{X_i\}_{i=1}^n,\P_{\tilde \theta^{\dksd}_n})^2 \\
&\qquad \leq 
\hvf \dksd_{K,m}( \{X_i\}_{i=1}^n,\P_{ \lambda \tilde \theta^{\dksd}_n +(1-\lambda) \theta })^2 \\
&\qquad  \leq \lambda \hvf \dksd_{K,m}( \{X_i\}_{i=1}^n,\P_{\tilde \theta^{\dksd}_n})^2 +(1-\lambda) \hvf \dksd_{K,m}( \{X_i\}_{i=1}^n,\P_{\theta})^2   
\end{talign}
which implies 
$\hvf \dksd_{K,m}( \{X_i\}_{i=1}^n,\P_{\tilde \theta^{\dksd}_n})^2 \leq\hvf \dksd_{K,m}( \{X_i\}_{i=1}^n,\P_{\theta})^2$ and $\tilde \theta^{\dksd}_n$ is the global minimum of $\theta \mapsto \hvf \dksd_{K,m}( \{X_i\}_{i=1}^n,\P_{\theta})^2$ for $n> N_0$.
\end{proof}

When $\stein$ is Fréchet differentiable on $\Theta$ equicontinuity can be obtained 
using the Mean value theorem, which simplifies the assumptions under which strong consistency holds.

We now prove our main result for consistency of minimum DKSD estimators: \cref{DKSD-consistency}:

\begin{proof}
Let $\lVert K \rVert + \lVert \nabla_x  K \rVert +  \lVert \nabla_x \nabla_y  K \rVert  \leq K_{\infty}$.
Note
 $\|\nabla_y\cdot\left( m(y) K\right)\| \leq 2f_2(y)K_{\infty}$ and
 $|\tr\left[m(y) \nabla_y \nabla_x\cdot(m(x)K) \right]|\leq 2 f_2(y)f_2(x) K_{\infty}$
 so 
 \begin{talign}
 |\stein_{\theta}(x,y)| \leq f_1(x)K_{\infty} f_1(y) + 2f_2(x)K_{\infty}f_1(y) + 2f_2(y)K_{\infty}f_1(x) + 3K_{\infty}f_2(x)f_2(y)
 \end{talign}
  which is symmetric and integrable by assumption.  Let $S_m$, $m=1,2, \ldots$ be an increasing sequence of closed balls in $\R^d$, such that  $\cup_{m=1}^{\infty}S_m = \mathbb{R}^d$. 
  Moreover,
  \begin{talign}
   \|\nabla_\theta \metric{s_p(x)}{  Ks_p(y)} \| & \leq  g_1(x)f_1(y)K_{\infty} + g_1(y)f_1(x)K_{\infty}\\
     \|\nabla_\theta \metric{\nabla_y\cdot\left( m(y) K\right)}{ s_p(x)} \| & \leq 2 K_{\infty} g_2(y)f_1(x)+2 f_2(y) g_1(x) K_\infty\\
     \| \nabla_\theta\metric{\nabla_x\cdot \left(m(x)K\right)}{  \nabla_y \cdot m}\|& \leq 2K_\infty g_2(x) f_2(y)+2K_{\infty} f_2(x)g_2(y)
      \\
    \| \nabla_\theta\tr\left[m(y) \nabla_y \nabla_x\cdot(m(x)K) \right] \|  & \leq
    2 K_{\infty} g_2(y) f_2(x) +2K_\infty f_2(y)g_2(x)
\end{talign}
  thus
$\|\nabla_{\theta} \stein_{\theta}(x,y)\|$ is bounded above by a continuous  integrable symmetric function, $(x,y)\mapsto s(x,y)$, which attains a maximum on the compact spaces $S_m \times S_m$. 
By the MVT applied on the $\R^m$-open neighbourhood of $\Theta$,
$ |  \stein_{\theta}(x,y)-\stein_{\alpha}(x,y)| \leq \|\nabla_{\theta} \stein_{\theta}(x,y)\| \|\theta - \alpha \| \leq s(x,y ) \|\theta - \alpha \| \leq
\max_{x,y \in S_m}s(x,y ) \|\theta - \alpha \|  $, 
and $\stein_{\theta}(x,y)$ is equicontinuous in $\theta \in C$ for $x,y \in S_m$.
Similarly, since $s$ is integrable,
 $| \int_\X \stein_{\theta}(x,y)\Q(dy)- \int_\X \stein_{\alpha}(x,z)\Q(dz)|
 \leq 
 \|\nabla_{\theta} \int_\X \stein_{\theta}(x,z) \dd \Q(z)\| \|\theta - \alpha \|
 \leq 
  \int_\X \|\nabla_{\theta} \stein_{\theta}(x,z) \| \dd \Q(z) \|\theta - \alpha \|
\leq 
 \max_{x\in S_m} \Q_zs(x,z)  \|\theta - \alpha \| \leq  
$ is equicontinuous in $\theta \in C$ for $x \in S_m$.
The rest follows as in the previous proposition.
\end{proof}

\subsubsection{Proof of \cref{DKSD-normality}: Asymptotic Normality}

\begin{proof}
Note that $\nabla_{\theta}\hvf \dksd_{K,m}( \{X_i\}_{i=1}^n,\P_{\theta})^2
= \frac{1}{N(N-1)} \sum_{i \neq j } \nabla_{\theta}\stein_{\theta}(X_i, X_j)$. Let $\mu(\theta) \defn \Q \otimes \Q[\nabla_{\theta}\stein_{\theta} ]$.  Assumptions 1 and 2 imply that $\Q \otimes \Q[\|\nabla_{\theta}k_{\theta}^0 \|^2] < \infty$.  By \cite[Theorem 7.1 ]{hoeffding1948class} it follows that  
\begin{talign}
\sqrt{n} \left( \nabla_{\theta}\hvf \dksd_{K,m}( \{X_i\}_{i=1}^n,\P_{\theta})^2 - \mu(\theta)\right) \xrightarrow[]{d} \mathcal N(0,4\Sigma(\theta)) 
\end{talign}
where 
\begin{talign}
\Sigma &= \Q\left [\Q_2 \left[ \nabla_\theta \stein_{\theta} - \mu(\theta)\right] \otimes \Q_2 \left[ \nabla_\theta \stein_{\theta} - \mu(\theta)\right] \right] \\
&= \int_\X 
 \left( \int_\X  \nabla_\theta \stein_{\theta}(x,y) \dd \Q(y)- \mu(\theta)\right)
\otimes  \left( \int_\X \nabla_\theta \stein_{\theta}(x,z) \dd \Q(z) - \mu(\theta)\right)\dd \Q(x)
\end{talign} 
Note that $\mu(\theta_*^{\dksd})= \Q \otimes \Q[\nabla_{\theta}\stein_{\theta} |_{\theta_*^{\dksd}} ] = \nabla_{\theta} \left(\Q \otimes \Q[\stein_{\theta} ] \right)|_{\theta = \theta_*^{\dksd}}$,
 and if $\Q \otimes \Q[\stein_{\theta} ]$ is differentiable around $\theta_*^{\dksd}$, then the first order optimality condition implies $\mu(\theta_*^{\dksd})=0$. 
 
Consider now $\nabla_{\theta}\nabla_{\theta}\hvf \dksd_{K,m}( \{X_i\},\P_{\theta})^2
= \frac{1}{n(n-1)} \sum_{i \neq j } \nabla_{\theta}\nabla_{\theta}\stein_{\theta}(X_i, X_j)$.
Note
\begin{talign}
\|\nabla_\theta \nabla_\theta  \nabla_\theta \metric{s_p(x)}{ Ks_p(y)} \| & \leqsim 
g_1(x) K_\infty f_1(y)+f_1(x) K_\infty g_1(y)
\\
\|\nabla_\theta \nabla_\theta  \nabla_\theta \metric{\nabla_y\cdot\left( m(y) K\right)}{ s_p(x)} \| & \leqsim  g_2(y) K_\infty f_1(x)+f_2(y) K_\infty g_1(x) 
\\
\|\nabla_\theta \nabla_\theta  \nabla_\theta \metric{\nabla_x\cdot \left(m(x)K\right)}{  \nabla_y \cdot m} \| & \leqsim 
f_2(y)K_\infty g_2(x)+g_2(y)K_\infty f_2(x)
\\
 \|\nabla_\theta \nabla_\theta  \nabla_\theta \tr\left[m(y) \nabla_y \nabla_x\cdot(m(x)K)   \right] & \leqsim
 g_2(y)K_\infty f_2(x)+f_2(y) K_\infty g_2(x)
\end{talign}
Hence by Assumptions 1-4 $\|\nabla_\theta \nabla_\theta  \nabla_\theta \stein_\theta \|$ is bounded above by a  continuous integrable symmetric function and we can apply the MVT to show equicontinuity as
in the proof above.
Moreover the  conditions of \cite[Theorem 1]{yeo2001uniform} hold for the components of $\nabla_{\theta}\nabla_{\theta}\stein_{\theta}$, so that 
$\sup_{\theta \in \mathcal{N}} \left|  \frac{1}{n(n-1)} \sum_{i \neq j } \partial_{\theta^a}\partial_{\theta^b}\stein_{\theta}(X_i, X_j)- \Q \otimes \Q  \partial_{\theta^a}\partial_{\theta^b}\stein_{\theta} \right |  \xrightarrow[]{a.s.} 0$ as $n\rightarrow \infty$, for all $a$ and $b$.

Finally we observe that $\Q \otimes \Q  \partial_{\theta^a}\partial_{\theta^b}\stein_{\theta} \big|_{\theta=\theta_*^{\dksd}} = g_{ab}(\theta_*^{\dksd})$, where $g$ is the information metric associated with $\dksd_{K,m}$. 
Indeed using $\delta_{p,q}=0$ if $p=q$
\begin{talign}
& \Q \otimes \Q  \partial_{\theta^a}\partial_{\theta^b}\stein_{\theta} \big|_{\theta=\theta_*^{\dksd}}\\
&=\partial_{\theta^a} \partial_{\theta^b}  \int_\X \int_\X p_{\theta_*^{\dksd}}(x)\delta_{p_{\theta},p_{\theta_*^{\dksd}}}(x)^{\top}K(x,y) \delta_{p_{\theta},p_{\theta_*^{\dksd}}}(y)p_{\theta_*^{\dksd}}(y)  \dd x \dd y \big|_{\theta=\theta_*^{\dksd}}\\
&=
\partial_{\theta^a}  \int_\X \int_\X p_{\theta_*^{\dksd}}(x)\partial_{\theta^b}\delta_{p_{\theta},p_{\theta_*^{\dksd}}}(x)^{\top}K(x,y) \delta_{p_{\theta},p_{\theta_*^{\dksd}}}(y)p_{\theta_*^{\dksd}}(y)  \dd x \dd y \big|_{\theta=\theta_*^{\dksd}}\\
&+ 
\partial_{\theta^a}  \int_\X \int_\X p_{\theta_*^{\dksd}}(x)\delta_{p_{\theta},p_{\theta_*^{\dksd}}}(x)^{\top}K(x,y) \partial_{\theta^b}\delta_{p_{\theta},p_{\theta_*^{\dksd}}}(y)p_{\theta_*^{\dksd}}(y)  \dd x \dd y \big|_{\theta=\theta_*^{\dksd}}
\\
&=
 \int_\X \int_\X p_{\theta_*^{\dksd}}(x)\partial_{\theta^b}\delta_{p_{\theta},p_{\theta_*^{\dksd}}}(x)^{\top}K(x,y) \partial_{\theta^a}\delta_{p_{\theta},p_{\theta_*^{\dksd}}}(y)p_{\theta_*^{\dksd}}(y)  \dd x \dd y \big|_{\theta=\theta_*^{\dksd}}\\
&+ 
 \int_\X \int_\X p_{\theta_*^{\dksd}}(x)\partial_{\theta^a}\delta_{p_{\theta},p_{\theta_*^{\dksd}}}(x)^{\top}K(x,y) \partial_{\theta^b}\delta_{p_{\theta},p_{\theta_*^{\dksd}}}(y)p_{\theta_*^{\dksd}}(y)  \dd x \dd y \big|_{\theta=\theta_*^{\dksd}} \\
 &= 
 2\int_\X \int_\X  \left( m^{\top}_{{\theta_*^{\dksd}}}(x) \nabla_x \partial_{{\theta^j}_*^{\dksd}} \log p_{{\theta_*^{\dksd}}}  \right)^{\top} K(x,y)\\
 & \qquad \qquad \qquad \left( m^{\top}_{{\theta_*^{\dksd}}}(y) \nabla_y  \partial_{{\theta^i}_*^{\dksd}}\log p_{{\theta_*^{\dksd}}}  \right)  \dd \P_{{\theta_*^{\dksd}}}(x) \dd \P_{{\theta_*^{\dksd}}}(y),
\end{talign}
so $\Q \otimes \Q  \partial_{\theta^a}\partial_{\theta^b}\stein_{\theta} \big|_{\theta=\theta_*^{\dksd}} = g_{ab}(\theta_*^{\dksd})$.
 The conditions of \cite[Theorem 3.1]{newey1994large} hold, from which the advertised result follows.
\end{proof}

\subsection{Diffusion Score Matching}\label{app:DSM-asymptotics}

Recall that the DSM is given by:
\begin{talign}
\dsm(\Q\|\P_{\theta})=\int_{\X} \left( \left\|  m^{\top}\nabla_x\log p_{\theta} \right\|_{2}^2 +\| m^{\top} \nabla \log q \|^2_2+2\nabla \cdot\left( m m^{\top} \nabla \log p_{\theta}\right)\right) \dd \Q
\end{talign} 
and we wish to estimate
\begin{talign}
 \theta_*^{\dsm} &=  \text{argmin}
_{\theta \in \Theta} \int_{\X} \left( \left\|  m^{\top}\nabla_x\log p_{\theta} \right\|_{2}^2 +2\nabla \cdot\left( m m^{\top} \nabla \log p_{\theta}\right)\right) \dd \Q \defn  \text{argmin}
_{\theta \in \Theta} \int_{\X} F_\theta \dd \Q 
\end{talign}
with a sequence of $M$-estimators $ \hat{\theta}^{\dsm}_n =  \text{argmin}
_{\theta \in \Theta} \frac1 n \sum_{i}^n F_{\theta}(X_i)$. 
Recall also we have
\begin{talign}
F_\theta(x)= \left\|  m^{\top}\nabla_x\log p_{\theta} \right\|_{2}^2+2 \metric{\nabla \cdot( m m^{\top} ) }{ \nabla \log p_\theta} + 2\tr \left[mm^{\top}  \nabla^2\log p_\theta \right].
\end{talign}

We will have a unique minimiser $\theta_*^{\dsm}$ whenever the map $\theta \mapsto \P_{\theta}$ is injective.


\subsubsection{Weak Consistency of DSM}

\begin{theorem}[\textbf{Weak Consistency of DSM}]
Suppose $\X$ be open subset of $\R^d$, and $\Theta \subset \R^m$.
Suppose  $\log p_{\theta}(\cdot)$ is $C^2(\X)$ and $m \in C^1(\X)$,
 and $ \| \nabla_x \log p_{\theta}(x) \|\leq f_1(x)$.
 Suppose also that
 $\| \nabla_x \nabla_x  \log p_{\theta}(x) |\leq f_2(x)$ on any compact set $C \subset \Theta$, where 
 $\| m^{\top} \| f_1 \in L^2(\Q)$,
 $ \| \nabla \cdot (mm^{\top}) \| f_1 \in L^1(\Q)$, $ \| mm^{\top}\|_{\infty} f_2 \in L^1(\Q)$. 
 If either $\Theta$ is compact, or $\Theta$ and $\theta \mapsto F_{\theta}$ are convex and $\theta^* \in \text{int}(\Theta)$, then $\hat{\theta}^{\dsm}_n$ is weakly consistent for $\theta^*$.
\end{theorem}

\begin{proof}
By assumption $\theta \mapsto F_\theta(x)$ is continuous.
 Suppose $\Theta$ is compact, taking $C=\Theta$, note
 \begin{talign}
 |F_{\theta}| &=\left| \left\|  m^{\top}\nabla_x\log p_{\theta} \right\|_{2}^2 +2\nabla \cdot\left( m m^{\top} \nabla \log p_{\theta}\right)\right|  \\
 &= \left |\left\|  m^{\top}\nabla_x\log p_{\theta} \right\|_{2}^2 +2\left(\nabla \cdot( m m^{\top} ) \cdot \nabla \log p_{\theta} + \tr \left[mm^{\top}  \nabla^2\log p_{\theta} \right]\right) \right | \\
 & 	\leqsim  \| m^{\top} \|^2 f_1^2 +  2\| \nabla \cdot (mm^{\top}) \| f_1 + 2  \| mm^{\top}\|_{\infty} f_2
\end{talign} 
which is integrable, so the conditions of Lemma 2.4 \cite{newey1994large} are satisfied so $\theta \mapsto \Q F_\theta$ is continuous, 
 and $\sup_{\Theta} | \frac1 n \sum_{i}^n F_{\theta}(X_i) - \Q F_{\theta} | \xrightarrow[]{p}  0$, and thus from theorem 2.1 \cite{newey1994large} $\hat{\theta}^{\dsm}_n \xrightarrow[]{p} \theta_*^{\dsm}$.
 If $\Theta$ is convex, note that the sum of convex functions is convex, so 
 $\theta \mapsto \frac1 n \sum_{i}^n F_{\theta}(X_i)$ is convex,
 and we can follow a derivation analogous to the one in \cref{DKSD-consistency}.
\end{proof}


\subsubsection{Asymptotic Normality of DSM}

\begin{theorem}[\textbf{Asymptotic Normality of DSM}]
Suppose $\X,\Theta$ be open subsets of $\R^d$ and $\R^m$ respectively.
If \textbf{(i)}  $\hat{\theta}^{\dsm}_n  \xrightarrow[]{p} \theta^* $, \textbf{(ii)} $\theta \mapsto \log p_\theta(x)$ is twice continuously differentiable on a closed ball $ \bar B(\epsilon,\theta^*) \subset \Theta$,
and
\begin{enumerate}
\item[\textbf{(iii)}] $\| mm^{\top}\|+\| \nabla_x \cdot(mm^{\top}) \| \leq f_1(x),$ and $ 
\| \nabla_x \log p \|
+  \| \nabla_{\theta^*} \nabla_x \log p \|+  \| \nabla_{\theta^*} \nabla_x \nabla_x \log p \| \leq f_2(x)$, with $f_1f_2, f_1 f_2^2 \in L^2(\Q)$
\end{enumerate} 

\begin{enumerate}
\item[\textbf{(iv)}] for $\theta \in \bar B(\epsilon,\theta^*) $ $\| \nabla_\theta \nabla_x \log p \|^2+\| \nabla_x \log p \| \|\nabla _\theta \nabla_\theta \nabla_x \log p \| +\| \nabla_\theta\nabla_\theta \nabla_x \log p \|
 +\|\nabla_\theta \nabla_\theta \nabla_x\nabla_x \log p \| \leq g_1(x)$, and $f_1g_1 \in L^1(\Q)$,
\end{enumerate}
and  \textbf{(v)} and the information tensor is invertible at $\theta^*$. 
 Then
 $$ \sqrt{n} \left( \hat{\theta}^{\dsm}_n- \theta^*\right) \xrightarrow[]{d} \mathcal N \left (0,g^{-1}(\theta^*)  \Q \left [ \nabla_{\theta^*} F_{\theta}\otimes \nabla_{\theta^*} F_{\theta}  \right ] g^{-1}(\theta^*) \right)$$
\end{theorem}

\begin{proof}
From $\textbf{(ii)}$ 
 $\theta \mapsto F_{\theta}$ is twice continuously differentiable on a ball $ B(\epsilon,\theta^*) \subset \Theta$.
Note $\nabla_\theta \frac1 N \sum_{i}^N F_{\theta}(X_i)= \frac1 N \sum_{i}^N \nabla_\theta F_{\theta}(X_i)$, 
then $\Q[  \nabla_{\theta} F_{\theta_*^{\dsm}}(X_i)]  = \nabla_{\theta}\Q[  F_{\theta_*^{\dsm}}(X_i)] =0$. Note
\begin{talign}
\|\nabla_{\theta}F_{\theta_*^{\dsm}}(x)\| & \leqsim \| mm^{\top} \|  \| \nabla_x \log p \|\| \nabla_{\theta} \nabla_x \log p \|
+ \| \nabla_x \cdot(mm^{\top}) \| \| \nabla_{\theta} \nabla_x \log p \|\\
&+ \|mm^{\top} \| \| \nabla_{\theta} \nabla_x \nabla_x \log p \| \\
 &\leqsim f_1(x)f_2(x)[f_2(x)+2].
\end{talign}
Hence $\nabla_{\theta} F_{\theta_*^{\dsm}} \in L^2(\Q)$, so by the CLT 
\begin{talign}
{\sqrt n} \nabla_{\theta} \frac1 n \sum_{i}^n F_{\theta_*^{\dsm}}(X_i ) 
&\xrightarrow[]{d} \mathcal N \left(0,  \Q \left [ \nabla_{\theta} F_{\theta_*^{\dsm}}\otimes \nabla_{\theta} F_{\theta_*^{\dsm}}  \right ] \right).
\end{talign}
Now $ \theta \mapsto \nabla_\theta \nabla_\theta F_\theta(x)$ is continuous on  $\overline  B(\epsilon,\theta^*)$ so we have:
\begin{talign}
  \| \nabla_\theta \nabla_\theta F_\theta(x) \| \leqsim \| mm^{\top} \| & \left(\| \nabla_\theta \nabla_x \log p \|^2+\| \nabla_x \log p \| \|\nabla _\theta \nabla_\theta \nabla_x \log p \| \right)\\
 &+\| \nabla \cdot(mm^{\top}) \| \| \nabla_\theta\nabla_\theta \nabla_x \log p \|
 +\| mm^{\top} \|  \|\nabla_\theta \nabla_\theta \nabla_x\nabla_x \log p \| \\
 &\leqsim f_1(x)g_1(x)
\end{talign}
Combining the above, we have that the assumptions of Lemma 2.4  \cite{newey1994large} applied to $\overline  B(\epsilon,\theta^*)$ hold, and 
 $\sup_{\overline  B(\epsilon,\theta^*)} \left| \frac1 n \sum_i^n \partial_{\theta^a} \partial_{\theta^b}F_{\theta} |_{\theta^*} (X_i) - \Q \partial_{\theta^a} \partial_{\theta^b}F_{\theta} |_{\theta^*} \right| \xrightarrow[]{p}0$. 
 As in \cref{DKSD-normality}  $\Q\partial_{\theta^a} \partial_{\theta^b}F_{\theta} |_{\theta^*}=g_{ab}(\theta^*)$  is the information tensor, which is continuous at $\theta^*$ by Lemma 2.4.
 The result follows by theorem 3.1 \cite{newey1994large}.
\end{proof}


\subsection{Strong Consistency and Central Limit Theorems for Exponential Families} \label{app:CLT-exponential-family}

Let $\X $ be an open subset of $\R^d$, $\Theta \subset \R^m$.
Consider the case when the density $p$ lies in an exponential family, i.e.
$	p_{\theta}(x)  \propto \exp \left(\langle \theta, T(x)\rangle_{\R^m} - c(\theta) \right)\exp(b(x)),
$ where $\theta \in \mathbb{R}^m$ and sufficient statistic $T \;  = \; (T_1, \ldots, T_m): \X \to \R^m$. 
Then $\nabla T \in \Gamma(\X, \R^{m \times d})$ and $\nabla_x \log p_{\theta} = \nabla_x b + \theta \cdot \nabla_x T$,
$\nabla_\theta \nabla_x \log p_{\theta} =  \nabla_x T^{\top}$.


\subsubsection{Strong Consistency of the Minimum Diffusion Kernel Stein Discrepancy Estimator} \label{Matrix-kernels}

We consider a RKHS $\mathcal{H}^d$ of functions $f: \mathcal X \;\to\; \mathbb R^d$ with matrix kernel $K$.
Recall the Stein kernel is 
\begin{talign}
k^0 &=\nabla_x \log p \cdot m(x)Km(y)^{\top}\nabla_y \log p + \nabla_x\cdot \left(m(x)K\right)\cdot  \nabla_y \cdot m +\tr\left[m(y) \nabla_y \nabla_x\cdot(m(x)K) \right]\\
&+
\nabla_y \cdot \left( m(y) K \right) \cdot m(x)^{\top}\nabla_x \log p+
\nabla_x \cdot \left( m(x) K \right)\cdot m(y)^{\top}\nabla_y \log p
\end{talign}
Given a (i.i.d.) sample $X_i \sim \Q$, we can define an estimator 
using the $U$-statistic
\begin{talign}
	\hvf \dksd_{K,m}( \{X_i\}_{i=1}^n,\P_\theta)^2 & = \frac{2}{n(n-1)} \sum_{1 \leq i < j \leq n}\stein(X_i, X_j).
\end{talign}
For the case where the density $p$ lies in an exponential family, then $k^0 = \theta^{\top} A \theta+ v^{\top}\theta + c$ where $A\in \Gamma(\X\times \X, \R^{m\times m}), v \in \Gamma(\X\times \X, \R^{m})$ are given by (we set $\phi \defn m^{\top} \nabla T^{\top} \in \Gamma( \X , \R^{d\times m})$)
\begin{talign}
 A &= \phi(x)^{\top}K(x,y) \phi(y)\\
 v^{\top}&=  
 \nabla_y b \cdot m(y)K(y,x)\phi(x)+ \nabla_x b \cdot m(x)K(x,y)\phi(y) \\
 & +\nabla_x\cdot \left(m(x)K\right)\cdot \phi(y)+
 \nabla_y\cdot\left( m(y) K\right)\cdot \phi(x)
 \\
 c& = \nabla_x b \cdot m(x)K(x,y)m(y)^{\top} \nabla_y b + \nabla_x \cdot \left(m(x)K\right)\cdot  \nabla_y \cdot m +\tr\left[m(y) \nabla_y \nabla_x\cdot(m(x)K) \right] \\
 &+ \nabla_y\cdot\left( m(y) K\right) \cdot m(x)^{\top}\nabla_x b+
\nabla_x\cdot \left(m(x)K\right)\cdot m(y)^{\top}\nabla_y b
\end{talign}

\begin{Lemma}Suppose $K$ is IPD, that $\nabla T$ has linearly independent rows, that $m$ is invertible,  and $  \| \phi  \|_{ L^1(\Q)}< \infty$.
Then the matrix $\int_\X A \Q \otimes \Q$ is symmetric positive definite. 
\end{Lemma}\label{lemma-A}
\begin{proof}
 The matrix $B=\int_\X A \Q \otimes \Q$ is symmetric
\begin{talign}
(\int_\X A \Q \otimes \Q)^{\top}&= \int_\X A(x,y)^{\top} \Q(\dd x) \otimes \Q(\dd y)=
 \int_\X  \nabla_y T m(y)K(x,y)^{\top} m(x)^{\top} \nabla_x T^{\top}\Q(\dd x) \otimes \Q(\dd y)\\
 &= \int_\X \nabla_y T m(y) K(y,x) m(x)^{\top} \nabla_x T^{\top}\Q(\dd y) \otimes \Q(\dd x) = \int_\X A \Q \otimes \Q.
\end{talign} 

Moreover, set $\phi \defn m^{\top}\nabla T^{\top}$, so $A(x,y) = \phi(x)^{\top} K(x,y)\phi(y)$. 
If $v \neq 0$, then $u \defn\phi v \neq 0$ as $\nabla T^{\top}$ has full column rank (i.e., the vectors $\{\nabla T_i\}$ are linearly independent) and $m$ is invertible, 
and $\| \phi v \|_{L^1(\Q) } = \int_\X \| \phi(x) v \|_1 \dd x \leq  \| v\|_1 \int_\X \| \phi(x) \|_1 \dd x< \infty$ implies $\dd \mu_i \defn u_i \dd \Q$ is a finite signed Borel measure for each $i$. 
Clearly
\begin{talign}
v^{\top}(\int_\X A \Q \otimes \Q)v &= \int_\X u(x)^{\top} K(x,y) u(y) \Q(\dd x)\Q(\dd y) \\
&= \int_\X K(x,y)_{ij} u_i(x)u_j(y)  \Q(\dd x)\Q(\dd y) \\
&= \int_\X K(x,y)_{ij} \mu_i(\dd x) \mu_j( \dd y) \geq 0.
\end{talign}
Moreover since the kernel is IPD, if this equals zero then for all $i$: $0=  \mu_i(C)=  u_i \Q(C) = \phi_{ij}v_j \Q(C)$ for all measurable sets $C$, which implies $\phi v =0$ and thus $v=0$. 
\end{proof}

\begin{Theorem}
Suppose $K$ is IPD with bounded derivative up to order $2$, that $\nabla T$ has linearly independent rows, and $m$ is invertible.
Suppose  $  \|\phi\|, \|\nabla_x b \| \|m \|, \| \nabla_x m \|+ \|m \| \in L^1(\Q)$.
The minimiser $\hat \theta^{\dksd}_n$ of $\hvf \dksd_{K,m}(\{X_i\}_{i=1}^n,\P_{\theta})$ exists eventually, and converges almost surely to the minimiser $\theta^*$ of $\dksd_{K,m}(\Q,\P_{\theta})$.
\end{Theorem}
\begin{proof}

Let $X_i: \Omega \to \X \subset \R^d$ be independent $\Q$-distributed random vectors. 
The $U$-statistic $A_n\defn \frac 2 {n(n-1)}\sum_{1\leq i<j \leq n}A(X_i,X_j)$ is symmetric semi-definite. 
Since $ \int_\X \| A \| \dd \Q \otimes \Q< \infty$, by theorem 1 \cite{hoeffding1961strong} the components of $A_n$ converge to the components of $B$ almost surely, and since the matrix inverse is a continuous map, 
by the continuous mapping theorem the components of $A_n^{-1}$ (the inverse exists eventually) converge almost surely to $B^{-1}$.
Hence the minimiser of $\hvf \dksd_{K,m}(\{X_i\}_{i=1}^n,\P_{\theta})^2=\theta^{\top}A_n\theta+v_n^{\top}\theta+c$ 
where $v_n \defn \frac 2 {n(n-1)} \sum_{1\leq i<j \leq n}v(X_i,X_j)$
exists eventually.
\begin{talign}
 |A(x,y)| &\leqsim K_\infty\|\phi(x) \|  \| \phi(y) \| \\
\| v \| & \leqsim 
 K_\infty \|\nabla_y b \| \|m(y) \| \| \phi(x)\|+ K_\infty \|\nabla_x b \| \|m(x) \| \| \phi(y)\| \\
 & + (\| \nabla_x m \| + \| m(x) \|)K_\infty \| \phi(y) \|
 +
(\| \nabla_y m \| + \| m(y) \|)K_\infty \| \phi(x) \|
 \\
 |c|& \leqsim K_\infty \|\nabla_x b\|\| m(x)\|\| m(y)\| \|\nabla_y b\|+ K_\infty( \| \nabla_x m \| + \| m(x) \|) \|\nabla_y  m\|+\\
 &+ K_{\infty} \| m(y) \|(1+ \| m(x) \|+ \| \nabla_x m \|) \\
 &+K_\infty( \| \nabla_y m \| + \| m(y) \|) \|\nabla_x  m\|  \| \nabla_x b \|
 + K_\infty( \| \nabla_x m \| + \| m(x) \|) \|\nabla_y  m\|  \| \nabla_y b \|
\end{talign}
and it follows from the integrability assumptions that $ \Q \otimes \Q |\stein_\theta |<\infty$. 
Since the product and sum of random variables that converge a.s. converge a.s.,
we have that $\hat \theta^{\dksd}_n \rightarrow \theta^*$ a.s.,
\begin{talign}
\hat \theta^{\dksd}_n =- \frac1 2 A_n^{-1} v_n \xrightarrow[]{a.s.} - \frac1 2 B^{-1} v =\theta^*.
\end{talign}
\end{proof}

\subsubsection{Asymptotic Normality of the DKSD Estimator}

We now consider the distribution of $\sqrt{n}(\hat \theta^{\dksd}_n-\theta^*)$. 
Recall that $A \in \Gamma(\X,\R^{m\times m}),v \in \Gamma(\X, \R^{m })$, and  for $n$ large enough $A_n^{-1}$ exists a.s., and $\hat \theta^{\dksd}_n=- \frac1 2 A_n^{-1} v_n$.

\begin{Theorem} 
Suppose  $  \|\phi\|, \|\nabla_x b \| \|m \|, \| \nabla_x m \|+ \|m \| \in L^2(\Q)$.
Then the $\dksd$ estimator is asymptotically normal.

\end{Theorem}\label{CLT-diffusion-Stein}
\begin{proof}
From the integrability assumptions, it follows that 
$v,A\in L^2(\Q\otimes \Q)$, and since $\X$ has finite $\Q\otimes \Q$-measure, $v,A \in L^1(\Q\otimes \Q)$.  
Assume first that $m=1$.
Hence the tuple $U_n\defn (v_n,A_n): \Omega \to \R^2$, with $\E[U_n]= (\int_\X v \Q \otimes \Q ,\int_\X A \Q \otimes \Q) \defn (U_1,U_2)$ , is asymptotically normal
\begin{talign}
\sqrt{n}(U_n-\E[U_n])\xrightarrow[]{d} \N(0,4\Sigma)
\end{talign}
where, setting $v^0 = v-U_1$ and $A^0 = A-U_2$
\begin{talign}
\Sigma &= \E \left[\left( \int_\X v^0(X,y) \dd \Q(y), \int_\X A^0(X,y) \dd \Q(y) \right) \otimes \left( \int_\X v^0(X,y) \dd \Q(y), \int_\X A^0(X,y) \dd \Q(y) \right)\right] \\
 &=\begin{pmatrix}
  \int_\X v^0(x,y) \dd \Q(y) \int_\X v^0(x,z) \dd \Q(z) \dd \Q(x)& 
  \int_\X v^0(x,y) \dd \Q(y) \int_\X A^0(x,z) \dd \Q(z) \dd \Q(x) \\
  \int_\X v^0(x,y) \dd \Q(y) \int_\X v^0(x,z) \dd \Q(z) \dd \Q(x) 
  & \int_\X A^0(x,y) \dd \Q(y) \int_\X A^0(x,z) \dd \Q(z) \dd \Q(x)
  \end{pmatrix}
\end{talign}
Since 
$\hat \theta^{\dksd}_n=g(U_n)$, $\theta^* = g(U)$ where $g(x,y) \defn -\frac1 2 x/y$, 
we can apply the delta method which states
\begin{talign}
\sqrt{n}(\hat\theta^{\dksd}_n-\theta^*)= \sqrt{n}(g(U_n)-g(U)) \xrightarrow[]{d} \N \left(0,4\nabla g(U) \Sigma \nabla g(U)^{\top}\right)
\end{talign}
and $\nabla g (U) = \left(-1/2U_2, U_1/2U_2^2\right)$.
Now let $m$ be arbitrary.
Since $A \in L^2(\Q)$ then setting 
$A^0 \defn A - \int_\X A \Q \otimes \Q$
we find 
\begin{talign}
 \sqrt{n}(A_n-\E[A_n]) \xrightarrow[]{d} \mathcal N\left(0, 4\Sigma_1 \right),\qquad \Sigma_1 \defn \int_\X \left[\int_\X A^0(x,y) \dd \Q(y) \otimes \int_\X A^0(x,y) \dd \Q(y) \right] \dd \Q(x)
\end{talign}
and similarly , with $v^0 \defn v- \int v \dd \Q \otimes \dd \Q$
\begin{talign}
 \sqrt{n}(v_n-\E[v_n]) \xrightarrow[]{d} \mathcal N\left(0, 4\Sigma_2 \right),\qquad \Sigma_2 \defn \int_\X \left[\int_\X v^0(x,y) \dd \Q(y) \otimes \int_\X v^0(x,y) \dd \Q(y) \right] \dd \Q(x).
\end{talign}
and 
\begin{talign}
 \sqrt{n}((v_n,A_n)-\E[(v_n,A_n)]) \xrightarrow[]{d} \mathcal N\left(0, 4\Sigma \right) 
 \end{talign}
 where 
 \begin{talign}
 \Sigma = \int_\X \left[ \left( \int_\X v^0(x,y) \dd \Q(y), \int_\X A^0(x,y) \dd \Q(y) \right) \otimes \left( \int_\X v^0(x,y) \dd \Q(y), \int_\X A^0(x,y) \dd \Q(y) \right)  \right] \dd \Q(x).
 \end{talign}
Let $\D \defn \R^m \times \R^{m \times m}$, which we equip with coordinates $z_{ijk}=(x_i,y_{jk})$.
Consider the function $g: \D \to \R^m$, $(x,y) \mapsto -\frac1 2 y^{-1}x$, so $g(v_n,A_n) = \theta^{\dksd}_n$. 
 Note $\Sigma \in\D \times \D$ and $\nabla g : \D \to \End\left( \D,\R^m\right) \cong \R^m \times \D$, so that $\nabla g(U) \Sigma \nabla g(U)^{\top} \in \R^{m \times m}$.
 First consider the matrix inversion $h(y)=y^{-1}$, so 
 $\nabla h(y) \in \R^{(m \times m) \times (m \times m)}$, and $\nabla h(y)_{(ij)(kr)} = \partial_{y^{kr}} h_{ij}$. Since 
 $h(y)_{ij}y_{jl} =\delta_{il}$
 we have $0=\partial_{kr} (h(y)_{ij}y_{jl})=\partial_{kr} (h(y)_{ij}) y_{jl}+  h(y)_{ij}\delta_{jk}\delta_{rl}=
 \partial_{kr} (h(y)_{ij}) y_{jl}+  h(y)_{ik}\delta_{rl}$
 and  
 \begin{talign}
 \nabla h(y)_{(is)(kr)}= \partial_{kr} (h(y)_{ij}) y_{jl}h(y)_{ls}=-h_{ik} \delta_{rl} h(y)_{ls}=-h(y)_{ik}  h(y)_{rs}
\end{talign} 
and clearly $f: x \mapsto x$, then $\nabla f(x) =1_{m\times m}$.
Moreover 
\begin{talign}
\partial_{y^{ab}} g_i(z)=\partial_{y^{ab}} \left(h(y)_{ij} f(x)_j \right)=\partial_{y^{ab}}(h(y)_{ij})x_j =-h(y)_{ia}h(y)_{bj}x_j, \quad  
\partial_{x^{l}} g_i(z)=h(y)_{il}
\end{talign}
Then 
\begin{talign}
(\nabla g(z) \Sigma)_{ir} &= \partial_v g_{i} \Sigma_{vr}=
g_{i,x^l}\Sigma_{x^lr}+g_{i,y^{ab}}\Sigma_{y^{ab}r}=h(y)_{il}\Sigma_{x^lr}+
\partial_{y^{ab}}(h(y)_{is})x_s \Sigma_{y^{ab}r} \\
&=h(y)_{il}\Sigma_{x^lr}-h(y)_{ia}  h(y)_{bs}x_s \Sigma_{y^{ab}r},
\end{talign}
so 
\begin{talign}
(\nabla g(z) \Sigma \nabla g(z)^{\top})_{ic}&= (\nabla g(z) \Sigma)_{ir} (\nabla g(z))_{cr}=
 (\nabla g(z) \Sigma)_{ir} \partial_r g_c \\
 &= h(y)_{il}\Sigma_{x^lr}\partial_r g_c-h(y)_{ia}  h(y)_{bs}x_s \Sigma_{y^{ab}r}\partial_r g_c
\end{talign}
with 
\begin{talign}
h(y)_{il}\Sigma_{x^lr}\partial_r g_c&=
h(y)_{il}\Sigma_{x^lx^b}\partial_{x^b} g_c+h(y)_{il}\Sigma_{x^ly^{as}}\partial_{y^{as}} g_c\\
&=h(y)_{il}\Sigma_{x^lx^b}h(y)_{cb}-h(y)_{il}\Sigma_{x^ly^{as}}h_{ca}(y)h(y)_{sj}x_j
\end{talign}
and 
\begin{talign}
-h(y)_{ia}  h(y)_{bs} &x_s \Sigma_{y^{ab}r}\partial_r g_c=-h(y)_{ia}  h(y)_{bs}x_s \left(\Sigma_{y^{ab}x^k}\partial_{x^k} g_c+\Sigma_{y^{ab}y^{ld}}\partial_{y^{ld}} g_c\right) \\
&=-h(y)_{ia}  h(y)_{bs}x_s \left(\Sigma_{y^{ab}x^k}h(y)_{ck}-\Sigma_{y^{ab}y^{ld}}h(y)_{cl}h(y)_{dj}x_j\right).
\end{talign}
Note we have 
\begin{talign}
\Sigma_{xx} 
&= \int_\X \int_\X v^0(x,y)\dd \Q(y)\otimes \int_\X v^0(x,z)\dd \Q(z) \dd \Q(x) \equiv \int_\X T(x) \otimes T(x) \dd \Q(x) \\
\Sigma_{xy} &=\int_\X \int_\X v^0(x,y)\dd \Q(y)\otimes \int_\X A^0(x,z)\dd \Q(z) \dd \Q(x)
\equiv \int_\X T(x) \otimes L(x) \dd \Q(x)
\\
\Sigma_{yy} &=\int_\X \int_\X A^0(x,y)\dd \Q(y)\otimes \int_\X A^0(x,z)\dd \Q(z) \dd \Q(x)
\equiv \int_\X L(x) \otimes L(x) \dd \Q(x)
\end{talign}
then
\begin{talign}
4\nabla g(U_1,U_2) \Sigma \nabla g(U_1,U_2)^{\top} &=  \int_\X (U_2^{-1}T) \otimes (TU_2^{-1}) \dd \Q\\
&-2 \int_\X \left( U_2^{-1}LU_2^{-1}U_1\right) \otimes \left(TU_2^{-1}\right) \dd \Q\\
&+\int_\X \left(U_2^{-1}LU_2^{-1}U_1 \right) \otimes \left(U_2^{-1}LU_2^{-1}U_1 \right) \dd \Q
\end{talign}
\end{proof}


\subsubsection{Diffusion Score Matching Asymptotics}

Consider the loss function
\begin{talign}
 L(x,\theta)  =  \metric{\nabla \log p_{\theta}}{ mm^{\top}\nabla\log p_{\theta}}  +2\left(\nabla \cdot( m m^{\top} ) \cdot \nabla \log p_{\theta} + \tr \left[mm^{\top}  \nabla^2\log p_{\theta} \right]\right).
\end{talign}
For the exponential family $L(x,\theta)=\theta^{\top}A \theta+ v^{\top}\theta+c$, where (we set $S=mm^{\top}$)
\begin{talign}
A&= \nabla T S \nabla T^{\top} \\
v^{\top} &= 2\nabla b \cdot S \nabla T^{\top} +2 \nabla \cdot S \cdot \nabla T^{\top}+2\tr\left[S \nabla^2 T_i \right]e_i \\
c& = \nabla b \cdot S \nabla b+2 \nabla \cdot S \cdot \nabla b +2 \tr[S \nabla \nabla b].
\end{talign}

\begin{theorem} Suppose $m$ is invertible and $\{\nabla T_i\}$ are linearly independent. Then if $A,v \in L^1(\Q)$, $\hat \theta^{\dsm}_n$ eventually exists and is strongly consistent.
If we also have  $A,v \in L^2(\Q)$, then $\hat \theta^{\dsm}_n$ 
is asymptotically normal.
\end{theorem}
\begin{proof}
Let $\M \defn \int A \dd \Q$, $H \defn \int v \dd \Q$.
If  $A=\nabla T mm^{\top} \nabla T^{\top}= \nabla T m (\nabla T m)^{\top}$
so $\rank{A}=\rank{\nabla T m (\nabla T m)^{\top}}=\rank{\nabla T m} =\rank{\nabla T} =\rank{\nabla T^{\top}}$ if $m$ is invertible.
So if  the vectors $\{\nabla T_i\}$ are linearly independent, then $\nabla T^{\top}$ has full column rank.
Then $A$ it is symmetric positive (strictly) definite
and the minimum of $L(\theta) \defn \int L(x,\theta) \dd \Q(x)$ is
$\theta^* = -\frac1 2 \M^{-1}H$ which for sufficiently large $n$ can be estimated by the random variable $\hat \theta^{\dsm}_n \defn -\frac1 2 \M_n^{-1}H_n$ which converges a.s. to $\theta$.

We consider the tuple $U_n\defn (H_n,\M_n)$, so $\E[U_n] = (H,\M)$.
Since $A,v \in L^2(\Q)$, then 
\begin{talign}
 \sqrt{n}\left(U_n -(H,\M) \right)
 \xrightarrow[]{d} \N \left(0, \Gamma\right)
 \end{talign}
where, setting $v^0= v-H$, $A^0 = A-\M$
\begin{talign}
\Gamma=\E \left[ (v^0,A^0)\otimes (v^0,A^0) \right].
\end{talign}
Let $\D \defn \R^m \times \R^{m\times m}$, and consider $g:\D \to \R^m$, defined by $g(x,y) = -\frac1 2 y^{-1}x$. 
Using the Delta method
\begin{talign}
\sqrt{n}(\hat \theta^{\dsm}_n-\theta^*) \xrightarrow[]{d} \N \left(0,4\nabla g(H,\M) \Gamma \nabla g(H,\M)^{\top}\right)
\end{talign}
where, proceeding as in \cref{CLT-diffusion-Stein}, we find 
\begin{talign}
4\nabla g(H,\M) \Gamma \nabla g(H,\M)^{\top} &=  \int_\X (\M^{-1}v^0 ) \otimes (v^0\M^{-1}) \dd \Q\\
&-2 \int_\X \left( \M^{-1}A^0\M^{-1}H\right) \otimes \left(v^0\M^{-1}\right) \dd \Q\\
& + \int_\X \left(\M^{-1}A^0\M^{-1}H \right) \otimes \left(\M^{-1}A^0\M^{-1}H \right) \dd \Q
\end{talign}

\end{proof}



\section{Proofs of Robustness of Minimum Stein Discrepancy Estimators}\label{appendix:robustness}

In this section, we provide conditions on the Stein operator (and Stein class) to obtain robust estimators in the context of DKSD and DSM.  In particular we prove \cref*{prop:robustness_dksd} and derive the influence function of DSM.

\subsection{Robustness of Diffusion Kernel Stein Discrepancy}

Let $T: \mathcal P_\Theta \to \Theta$ with $T(\P)= \argmin_{\Theta} \dksd_{K,m}(\P \| \P_\theta)$ be defined by 
$\IF(z,\Q)\defn \lim_{t\to 0} (T(\Q+t(\delta_z-\Q))-T(\Q))/t$.
Denote $\Q_t=\Q+t(\delta_z-\Q)$, 
$\theta_t = T(\Q_t)$, $\theta_0 =T(\Q)$.
Note that by the first order optimality condition:
\begin{talign}
\nabla_\theta \int_\X \int_\X \stein \Q_t \otimes \Q_t |_{\theta_t}= \nabla_{\theta_t}  \dksd_{K,m}(\Q_t \| \P_\theta) =0.
\end{talign}
By the MVT, there exists $\bar \theta$ on the line joining $\theta_0$ and $\theta_t$ for which
\begin{talign}
 0 = \int_\X \int_\X \nabla_{\theta} \stein |_{\theta_0} \Q_t \otimes \Q_t+ \int_\X \int_\X \nabla_{\theta} \nabla_{\theta} \stein |_{\bar \theta} \Q_t \otimes \Q_t (\theta_t-\theta_0).
\end{talign}
Expanding 
\begin{talign}
\Q_t \otimes \Q_t \nabla_{\theta} \stein |_{\theta_0} 
& = t^2(\delta_z-\Q)\otimes (\delta_z-\Q) \nabla_\theta \stein |_{\theta_0} + 2t \Q_y  \nabla_\theta \stein |_{\theta_0}(z,y)
\end{talign}
where we have used the optimality condition. On the other hand 
\begin{talign}
 \Q_t \otimes \Q_t \nabla_{\theta} \nabla_{\theta} \stein |_{\bar \theta} =
 (1-2t)\Q\otimes \Q \nabla_{\theta} \nabla_{\theta} \stein |_{\bar \theta}+t^2(\delta_z -\Q)\otimes (\delta_z -\Q)\nabla_{\theta} \nabla_{\theta} \stein |_{\bar \theta}+2t \Q_y \nabla_{\theta} \nabla_{\theta} \stein |_{\bar \theta}(z,y).
\end{talign}
Hence 
\begin{talign}
\Q_y  \nabla_\theta \stein |_{\theta_0}(z,y) = \half \left( 
 (1-2t)\Q\otimes \Q \nabla_{\theta} \nabla_{\theta} \stein |_{\bar \theta}
 +2t \Q_y \nabla_{\theta} \nabla_{\theta} \stein |_{\bar \theta}(z,y)
\right)\frac{\theta_t-\theta_0}{t}+O(t),
\end{talign}
and taking the limit $t\to 0$, $\bar \theta \to \theta_0$ and using a derivation as in the proof of \cref{DKSD-normality}
\begin{talign}
\Q_y  \nabla_\theta \stein |_{\theta_0}(z,y)
& = \half 
\int_{\X} \int_{\X} \nabla_{\theta} \nabla_{\theta} \stein |_{\theta_0} \dd \Q \otimes \dd \Q \IF(z,\Q) =  g(\theta_0) \IF(z,\Q)
\end{talign}
hence the influence function is given by
\begin{talign}
\IF(z,\Q) = g(\theta_0)^{-1} \int_{\X}  \nabla_\theta \stein |_{\theta_0}(z,y) \dd \Q(y).
\end{talign}

 We aim to show the estimator is $B$-robust, that is $z \mapsto \|\mbox{IF}(z, \mathbb{Q})\|$ is bounded.  Suppose that the additional assumptions hold. Then there exists a function $c$ such that $\int \metric{s_p(x)}{K(x,y)\nabla_{\theta_0}s_p(y)}\mathbb{Q}(dy) \leq \lVert s_p(x) \rVert c(x)$ which is bounded in $x \in \mathcal{X}$.   Following a similar argument, and using the assumptions, a similar limit will hold for all terms in $\int \nabla_{\theta_0} \stein (z,y) \dd \Q(y)$.  It follows that $\sup_{z\in \mathcal{X}}\|\mbox{IF}(z, \mathbb{Q})   \| < \infty$.

\subsection{Robustness of Diffusion Score Matching}

The scoring rule $S:\X \times \mathcal P_\X \to \R$ of $\dsm$ is 
\begin{talign}
S(x,\P_{\theta}) \defn \half \left\|  m^{\top}\nabla_x\log p_{\theta} \right\|_{2}^2 
+\nabla \cdot\left( m m^{\top} \nabla \log p_{\theta}\right)(x)
\end{talign}
Indeed the proof of \cref{prop:SM_is_Stein} we have
\begin{talign}
\int_\X \left\| m^{\top} \nabla \log q \right\|^2 \dd \Q =  -  \int_\X \nabla \cdot\left( m m^{\top} \nabla \log q \right) \dd \Q.
\end{talign} 
which implies $\Q S(\cdot ,\Q)= -\half
\int_{\X} \left\| m^{\top} \nabla \log q \right\|^2 \dd \Q$, so 
\begin{talign}
\Q S(\cdot ,\P_\theta)-\Q S(\cdot ,\Q)
&= 
 \int_{\X} \left(\half \left\|  m^{\top}\nabla_x\log p_{\theta} \right\|_{2}^2 
 +\half  \left\| m^{\top} \nabla \log q \right\|^2 
+\nabla \cdot\left( m m^{\top} \nabla \log p_{\theta}\right) \right)\dd \Q\\
& =\dsm_m( \Q\| \P_\theta).
\end{talign}
From 4.2 \cite{Dawid2014} the influence function is then  $\IF(x,\P_{\theta})= g_{\text{DSM}}(\theta)^{-1}s(x,\theta)$,
where 
\begin{talign}
s(x,\theta) &\defn \nabla_\theta S(x,\theta)=\half \nabla_\theta\|  m^{\top}\nabla_x\log p_{\theta} \|_{2}^2 
+\nabla_\theta \nabla_x \cdot\left( m m^{\top} \nabla_x \log p_{\theta}\right)\\
&=\half \nabla_\theta\|  m^{\top}\nabla_x\log p_{\theta} \|_{2}^2 
+\nabla_\theta \left( \metric{\nabla_x \cdot( m m^{\top} )}{ \nabla \log p_\theta}+ \tr \left[mm^{\top}  \nabla^2_x\log p_\theta \right]\right)\\
&=
\nabla_x \nabla_\theta\log p_{\theta} mm^{\top}\nabla_x\log p_{\theta}
+ (\nabla_x \nabla_\theta  \log p_\theta )\nabla_x \cdot( m m^{\top} )+ \Tr[mm^{\top}\nabla_{x}  \nabla_{x} ] \nabla_\theta\log p_\theta 
\end{talign}
 and where
$g_{\text{DSM}}(\theta) \defn \P_\theta \nabla_\theta  \nabla_\theta S(\cdot,\theta)$ is the information metric associated with DSM. 
Hence the estimator is bias-robust iff 
$x \mapsto s(x,\theta_*^{\dsm})$ is bounded.

\section{Additional Numerical Experiments}\label{appendix:experiments}

In this section, we provide further details and expand on the numerical experiments in the main paper.

\subsection{Efficiency of Minimum SD Estimators for Scale Parameters of Symmetric Bessel distributions} \label{appendix:symbessel}

In this section, we extend the results from the main text and compares SM with KSD based on a Gaussian kernel and a range of lengthscale values for the scale parameter of the symmetric Bessel distribution. The results, given in \cref{fig:symmetric_bessel}, are also based on $n=500$ IID realisations in $d=1$. Similar results to those for the location parameter are obtained: KSD can deal with rougher densities, as illustrated when $s=0.6$.

\begin{figure}[H]
\begin{center}
\includegraphics[width=0.18\textwidth,clip,trim = 0 0 0 0]{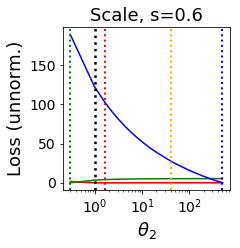}
\includegraphics[width=0.19\textwidth,clip,trim = 0 0 9.7cm 0]{Figures/symbessel_s1_scale.png}
\includegraphics[width=0.17\textwidth,clip,trim = 0 0 0 0]{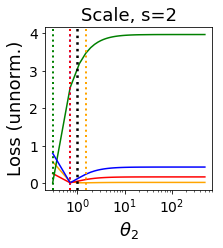}
\includegraphics[width=0.22\textwidth,clip,trim = 9cm 0 0 0]{Figures/symbessel_s1_scale.png}
\vspace{-5mm}
\end{center}
\caption{\textit{Minimum SD Estimators for the Scale of a Symmetric Bessel Distribution}. We consider the case where $\theta_1^*=0$ and $\theta_2^*=1$ and $n=500$ for a range of smoothness parameter values $s$ in $d=1$.}
\vspace{-3mm}
\label{fig:symmetric_bessel_scale}
\end{figure}

\subsection{Bias Robustness of Minimum SD Estimators for the Symmetric Bessel and Non-standardised Student-t Distributions} \label{appendix:experiments_robustness}

In this section, we explore the robustness of minimum SD estimators for the two other examples in the main paper: the symmetric Bessel distribution ($\nu =1000$) and the non-standardised student-t distribution. We once again select a diffusion matrix of the form $m(x) = 1/(1+\|x\|^\alpha)$, and fix $\alpha=1$ in both cases. This choice is refered to as ``robust DKSD''. On the other hand, we call ``efficient DKSD'' the DKSDs with choices of $m$ as highlighted in the main text (and which were chosen to improve efficiency in both cases). The results are provided in \cref{fig:twoexamples_robust}. In each case, we used $n=500$ data points, $80$ of which were corrupted by a Dirac at some value of given on the x-axis. Both in the student-t and symmetric Bessel distribution, we notice that the ``efficient DKSD'' has an $l_1$ error which grows with the value of the Dirac, whereas the ``robust DKSD'' is bounded as a function of this Dirac.
\begin{figure}[H]
\begin{center}
\includegraphics[width=0.2\textwidth,clip,trim = 0 0 4.9cm 0]{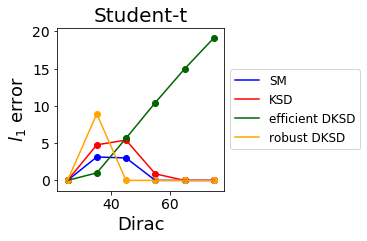}
\hspace{5mm}
\includegraphics[width=0.33\textwidth,clip,trim = 0 0 0 0]{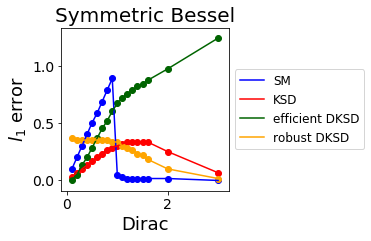}
\vspace{-5mm}
\end{center}
\caption{\textit{The Robustness of Minimum SD Estimators for the Symmetric Bessel and Student-t Distributions}. \textit{Left:} Student-t distribution. \textit{Right:} Symmetric Bessel distribution. }
\vspace{-3mm}
\label{fig:twoexamples_robust}
\end{figure}

\end{document}